\documentclass[a4paper,10pt]{article}

\usepackage[top=3.0cm,bottom=3.0cm,left=3.0cm,right=3.0cm]{geometry}

\usepackage[active]{srcltx}

\usepackage{amsfonts}
\usepackage{amsmath}
\usepackage{amsthm}
\usepackage{amssymb}
\usepackage{hyperref}
\usepackage{graphicx}
\usepackage{wrapfig}

\usepackage{array}
\usepackage{tabulary}
\usepackage{enumitem}
\usepackage{upref}
\usepackage{mathtools}

\newtheorem{theorem}{Theorem}[section]

\newtheorem{condition}[theorem]{Condition}
\newtheorem{corollary}[theorem]{Corollary}

\newtheorem{definition}[theorem]{Definition}
\newtheorem{example}[theorem]{Example}
\newtheorem{lemma}[theorem]{Lemma}

\newtheorem{proposition}[theorem]{Proposition}
\newtheorem{remark}[theorem]{Remark}

\usepackage{soul,xcolor}

\usepackage{bbm}

\newcommand\1{\mathbbm{1}}

\newcommand{\mc}{\mathcal}

\newcommand{\lan}{\langle}
\newcommand{\ran}{\rangle}

\newcommand{\td}{\tilde}
\newcommand{\eps}{\varepsilon}
\newcommand\loc{{\rm loc}}
\newcommand\lin{{\rm lin}}

\DeclareMathOperator*{\esssup}{ess\,sup}
\DeclareMathOperator*{\essinf}{ess\,inf}
\DeclareMathOperator{\sign}{sign}
\DeclareMathOperator{\diverg}{div}

\makeatletter
\let\@fnsymbol\@arabic
\makeatother

\numberwithin{equation}{section}

\allowdisplaybreaks[4]

\begin{document}

\title{Stochastic ODEs and stochastic linear PDEs with critical drift: regularity, duality and uniqueness}
\author{
  Lisa Beck\footnote{Institut f\"ur Mathematik, Universit\"at Augsburg, Universit\"atsstrasse 14, 86159 Augsburg, Germany. E-mail address: {\tt lisa.beck@math.uni-augsburg.de}.}, 
  Franco Flandoli\footnote{Scuola Normale Superiore, Piazza dei Cavalieri 7, 56126 Pisa, Italy. E-mail address {\tt franco.flandoli@sns.it}.}, 
  Massimiliano Gubinelli\footnote{Hausdorff Center for Mathematics and Institute for Applied Mathematics, Universit\"at, Bonn, Germany. E-mail address: {\tt gubinelli@iam.uni-bonn.de}}, 
  Mario Maurelli\footnote{Dipartimento di Matematica `Federigo Enriques', Universit\`a degli Studi di Milano, via Saldini 50, 20133 Milano, Italy. E-mail address: {\tt mario.maurelli@unimi.it}}}
\date{}
\maketitle

\begin{abstract}
In this paper linear stochastic transport and continuity equations with drift in critical $L^{p}$ spaces are considered. In this situation noise prevents shocks for the transport equation and singularities in the density for the continuity equation, starting from smooth initial conditions. Specifically, we first prove a result of Sobolev regularity of solutions, which is false for the corresponding deterministic equation. The technique needed to reach the critical case is new and based on parabolic equations satisfied by moments of first derivatives of the solution, opposite to previous works based on stochastic flows. The approach extends to higher order derivatives under more regularity of the drift term. By a duality approach, these regularity results are then applied to prove uniqueness of weak solutions to linear stochastic continuity and transport equations and certain well-posedness results for the associated stochastic differential equation (sDE) (roughly speaking, existence and uniqueness of flows and their $C^\alpha$ regularity, strong uniqueness for the sDE when the initial datum has diffuse law). Finally, we show two types of examples: on the one hand, we present well-posed sDEs, when the corresponding ODEs are ill-posed, and on the other hand, we give a counterexample in the supercritical case.
\\[1ex]
\end{abstract}

% 60H15 Stochastic partial differential equations
% 60H10 Stochastic ordinary differential equations
% 35A02 Uniqueness problems: global uniqueness, local uniqueness,
% non-uniqueness
% 35B65 Smoothness and regularity of solutions

\section{Introduction}

Let $b \colon [ 0,T] \times\mathbb{R}^{d}\rightarrow\mathbb{R}^{d}$, for $d \in \mathbb{N}$, be a deterministic, time-dependent vector field, that we call drift. Let $(W_{t}) _{t\geq0}$ be a Brownian motion in $\mathbb{R}^{d}$, defined on a probability space $(\Omega,\mathcal{A},P)$ with respect to a filtration $(\mathcal{G}_{t}) _{t\geq0}$ and let~$\sigma$ be a real number. The following three stochastic equations are (at least formally) related:
\begin{enumerate}
  \item the stochastic differential equation (\textit{sDE}) 
  \begin{equation}
  dX =b(t,X ) dt+\sigma dW_{t},\qquad X_{0}=x , \tag{sDE}
  \label{SDE}
  \end{equation}
  where $x\in\mathbb{R}^{d}$; the unknown $(X_{t}) _{t\in [0,T]}$ is a stochastic process in $\mathbb{R}^{d}$;
  \item the stochastic transport equation (\textit{sTE})
  \begin{equation}
  du+b\cdot\nabla udt+\sigma\nabla u\circ dW_{t}=0,\qquad u|_{t=0}=u_{0} , 
  \tag{sTE}  \label{stoch transport}
  \end{equation}
  where $u_{0} \colon \mathbb{R}^{d}\rightarrow\mathbb{R}$, 
  $b\cdot\nabla u=\sum_{i=1}^{d}b_{i}\partial_{x_{i}}u$, $\nabla u\circ dW_{t}=\sum_{i=1}^{d}\partial_{x_{i}}u\circ dW_{t}^{i}$, and Stratonovich multiplication is used (precise definitions will be given below); the unknown $(u(t,x))_{t\in [ 0,T] ,x\in \mathbb{R}^{d}}$ is a scalar random field;
  \item the stochastic continuity equation (\textit{sCE}) 
  \begin{equation}
  d\mu +\diverg(b\mu) dt+\sigma \diverg 
  (\mu \circ dW_{t}) =0,\qquad\mu|_{t=0}=\mu_{0} , \tag{sCE}
  \label{stoch cont}
  \end{equation}
  where $\mu_{0}$ is a measure, $\diverg (\mu \circ dW_{t})$ stands for $\sum_{i=1}^{d}\partial_{x_{i}}\mu \circ dW_{t}^{i}$, the unknown $(\mu_{t}) _{t\in [ 0,T] }$ is a family of random measures on $\mathbb{R}^{d}$, and thus the differential operations have to be understood in the sense of distributions.
\end{enumerate}

The aim of this paper is to investigate several questions for these equations in the case when the drift is in a subcritical or even critical space, a case not reached by any approach until now.

\subsection{Deterministic case $\protect\sigma=0$}\label{subsection deterministic}

For comparison with the results for the stochastic equations presented later on (due to the presence of noise), we first address the deterministic case $\sigma=0$. We start by explaining the link between the three equations and recall some classical results --~in the positive and in the negative direction~-- under various regularity
assumptions on the drift~$b$. When~$b$ is smooth enough, then:
\begin{enumerate}[font=\normalfont, label=(\roman{*}), ref=(\roman{*})]
  \item the sDE generates a flow $\Phi_{t}(x)$ of diffeomorphisms;
  \item the sTE is uniquely solvable in suitable spaces, and for the solution we have the representation formula $u(t,x) =u_{0}(\Phi_{t}^{-1}(x) )$;
  \item the sCE is uniquely solvable in suitable spaces, and the solution $\mu_{t}$
  is the push forward of $\mu_{0}$ under $\Phi_{t}$, i.e.~$\mu_{t}=(\Phi
  _{t}) _{\sharp}\mu_{0}$.
\end{enumerate}
These links between the three equations can be either proved a posteriori, after the equations have been solved by their own arguments, or they can be used to solve one equation by means of the other.

Well-posedness of the previous equations and links between them have been explored also when~$b$ is less regular. To simplify the exposition, let us summarize with the term ``weakly differentiable'' the classes of non-smooth~$b$ considered in~\cite{DIPELI89,AMBROSIO04}. In these works it has been proved that, whenever~$b$ is
weakly differentiable, sTE and sCE are well-posed in classes of weak solutions; moreover, a generalized or Lagrangian flow for the sDE exists. Remarkable is the fact that the flow is obtained by a preliminary solution of the sTE or of the sCE, see~\cite{DIPELI89,AMBROSIO04} (later on in~\cite{CRIDEL08}, similar results have been obtained directly on the sDE). However, when the regularity of~$b$ is too poor, several problems arise, for which, at the level of the sDE and its flow, we want to mention two types:

\begin{enumerate}
  \item[1)] non-uniqueness for the sDE, and, more generally, presence of discontinuities in the flow;
  \item[2)] non-injectivity of the flow (two trajectories can coalesce) and, more generally, mass concentration.
\end{enumerate}
These phenomena have counterparts at the level of the associated sCE and sTE:
\begin{enumerate}
  \item[1)]  non-uniqueness for the sDE leads to non-uniqueness for the sCE and sTE;
  \item[2)]  non-injectivity of the flow leads to shocks in the sTE (i.e.~absence of continuous solutions, even starting from a continuous initial datum), while mass concentration means that a measure-valued solution of the sCE does not remain distributed.
\end{enumerate}
Elementary examples can be easily constructed by means of continuous drifts in dimension~$1$; more sophisticated examples in higher dimension, with bounded measurable and divergence free drift, can be found in~\cite{AIZENMAN78}. Concerning regularity, let us briefly give some details for an easy example: Consider, in dimension $d=1$, the drift $b(x) \coloneqq - \sign(x) \vert x \vert^{\alpha }$ for some $\alpha \in (0,1)$. All trajectories of the ODE coalesce at $x=0$ in finite time; the solution to the deterministic TE develops a shock (discontinuity) in finite time, at $x=0$, from every smooth initial condition $u_{0}$ such that $u_{0}( x) \neq u_{0}( -x) $ for some $x\neq 0$; the deterministic CE concentrates mass at $x=0$ in finite time, if the initial mass is not zero. See also Section~\ref{ex_section} for similar examples of drift terms leading to non-uniqueness or coalescence of trajectories for the deterministic ODE (which in turn results in non-uniqueness and discontinuities/mass concentration for the PDEs).

Notice that the outstanding results of~\cite{DIPELI89,AMBROSIO04} (still in the deterministic case) are concerned only with uniqueness of weak solutions. The only results to our knowledge about regularity of solutions with rough drifts are those of \cite[Section~3.3]{BAHCHEDAN11} relative to the loss of regularity of solutions to the TE when the vector field satisfies a log-Lipschitz condition, which is a far better situation than those considered in this paper. We shall prove below that these phenomena disappear in the presence of noise. Of course they also disappear in the presence of viscosity, but random perturbations of transport type $\nabla u\circ dW_{t}$ and viscosity~$\Delta u$ are completely different mechanisms. The sTE remains an hyperbolic equation, in the sense that the solution follows the characteristics of the single particles (so we do not expect regularization of an irregular initial datum); on the contrary, the insertion of a viscous term corresponds to some averaging, making the equation of parabolic type. One could interpret transport noise as a turbulent motion of the medium where transport of a
passive scalar takes place, see~\cite{CHAGAWHORKUPVER03}, which is different from a dissipative mechanism, although some of the consequences on the passive scalar may have similarities.

\subsection{Stochastic case $\protect\sigma\neq0$}

In the stochastic case $\sigma\neq0$, when~$b$ is smooth enough, the existence of a stochastic flow of diffeomorphisms~$\Phi$ for the sDE, the well-posedness of sTE and the relation $u(t,x) =u_{0}(\Phi_{t}^{-1}(x) )$ are again known results, see \cite{KUNITA84a,KUNITA84,KUNITA90}; moreover, the link with the sCE could be established as well. However, the stochastic case offers a new possibility, namely that due to nontrivial \emph{regularization effects of the noise}, well-posedness of sDE, sTE and sCE remains true even if the drift~$b$ is quite poor, opposite to the deterministic case. Notice that we are not talking about the well-known regularization effect of a Laplacian or an expected value. By regularization we mean that some of the pathologies mentioned above about the deterministic case (non-uniqueness and blow-up) may disappear even at the level of a single trajectory~$\omega$; we do not address any regularization of solutions in time, i.e.~that solutions become more regular than the initial conditions, a fact that is clearly false when we expect relations like $u(t,x) =u_{0}(\Phi_{t}^{-1}(x))$. 

\subsection{Aim of this paper}

The aim of this work is to prove several results in this direction and develop a sort of comprehensive theory on this topic. The results in this paper are considerably advanced and are obtained by means of new powerful strategies, which give a more complete theory. The list of our main results is described in the Subsections~\ref{intro section regularity}-\ref{intro section SDE}; in a few sentences, we are concerned with:
\begin{enumerate}[font=\normalfont, label=(\roman{*}), ref=(\roman{*})]
\item regularity for the transport (and continuity) equation;
\item uniqueness for the continuity (and transport) equation;
\item uniqueness for the sDE and regularity for the flow.
\end{enumerate}

In the following subsections, we will explain the results in more detail and give precise references to previous works on the topics. Moreover, we will also analyze the crucial regularity assumptions on the drift term (discussing its criticality in a heuristic way and via appropriate examples, which are either classical or elaborated at the end of the paper).

\subsection{Regularity assumptions on~$b$}

As already highlighted before, the key point for the question of existence, uniqueness and regularity of the solutions to the relevant equations is the regularity assumption on the drift~$b$. In particular, we will not work with any kind of differentiability or H\"{o}lder condition, but merely with an integrability condition. 
We say that a vector field $f \colon [ 0,T] \times\mathbb{R}^{d}\rightarrow\mathbb{R}^{d}$ satisfies the Ladyzhenskaya--Prodi--Serrin condition~(LPS) with exponents $p,q\in(2,\infty)$ if
\begin{equation*}
f\in L^{q}([0,T];L^{p}(\mathbb{R}^{d},\mathbb{R}^{d})),\qquad\frac{d}{p}+\frac{2}{q}\leq1.
\end{equation*}
We shall write $f\in\mathcal{LPS}(p,q)$ (the precise definition will be given in Section~\ref{subsection regularity assumptions}), and we use the norm
\begin{equation*}
\Vert f \Vert_{L^{q}([0,T];L^{p})} \coloneqq \bigg(\int_{0}^{T}\Big(\int_{\mathbb{R}^{d}} \vert f(t,x)
\vert^{p} dx \Big)^{q/p} dt \bigg)^{1/q}.
\end{equation*}
We may extend the definition to the limit case $(p,q) =(\infty,2)$ in the natural way: we say that $f\in\mathcal{LPS}(\infty,2)$ if $f\in L^{2}(0,T;L^{\infty}(\mathbb{R}^{d}, \mathbb{R}^{d}) )$ and we use the norm 
\begin{equation*}
\Vert f \Vert_{L^{2}(0,T;L^{\infty})
}^{2}= \int_{0}^{T} \Vert f(t,\cdot) \Vert_{\infty}^{2} dt
\end{equation*}
with the usual meaning of $\Vert \cdot \Vert _{\infty}$ as the essential supremum norm. The extension to the other limit case $(p,q) =(d,\infty)$ is more critical (similarly to the theory of 3D Navier--Stokes equations, see below). The easy case is when $q=\infty$ is interpreted as continuity in time: $C([0,T];L^{d}(\mathbb{R}^{d},\mathbb{R}^{d}))$; on the contrary, $L^{\infty}(0,T;L^{d}(\mathbb{R}^{d},\mathbb{R}^{d}) )$ is too difficult in general and we shall impose an additional smallness assumption (which shall be understood implicitly whenever we address the case $(d,\infty)$ in this introduction).

Roughly speaking, our results will hold for a drift~$b$ which is the sum of a Lipschitz function of space (with some integrability in time) plus a vector field of LPS class. In the sequel of the introduction, for simplicity of the exposition, we shall not mention the Lipschitz component anymore, which is however important to avoid that the final results are restricted to drift with integrability (or at least boundedness) at infinity.

Let us note that if $p,q\in(2,\infty)$, then the space $L^{q}([0,T];L^{p}(\mathbb{R}^{d}, \mathbb{R}^{d}))$ is the closure in the topology $\Vert \cdot\Vert_{L^{q}([0,T];L^{p})}$ of smooth functions with compact support. The same is true for the space $C([ 0,T] ;L^{d}(\mathbb{R}^{d},\mathbb{R}^{d}))$. In the limit cases $(p,q) =(\infty,2)$ and $( p,q)=(d,\infty)$, using classical mollifiers, there exists a sequence of smooth functions with compact support which converges almost surely and has uniform bound in the corresponding norm. This fact will allow us to follow an approach of a~priori estimates, i.e.~perform all computations for solutions to the equation with smooth coefficients, obtain uniform estimates for the associated solutions, and then deduce the statement after passage to the limit.

We further want to comment on the significance of the LPS condition in fluid dynamics. The name LPS comes from the authors Ladyzhenskaya, Prodi and Serrin who identified this condition as a class where regularity and well-posedness of 3D Navier--Stokes equations hold, see \cite{KISLAD57,LIOPRO59,PRODI59,SERRIN62,LADYZ67,GALDI00}. The limit case $( p,q) =(d,\infty)$ generated a lot of research and can be treated almost as the other cases if there is continuity in time or some smallness condition, see for instance \cite{BEIRAO97,LIONS96}, but the full $L^{\infty }( 0,T;L^{d}(\mathbb{R}^{d},\mathbb{R}^{d}))$ case is very difficult, see 
\cite{ESCSERSVE03} and related works. It has been solved only recently, at the price of a very innovative and complex proof. A similar result for our problem is unknown. The deep connection of the LPS class, especially when $\frac{d}{p}+ \frac{2}{q}=1$, with the theory of 3D Navier--Stokes equations is one of our main motivations to analyze stochastic transport under such conditions.

We finally note that, while preparing the second version of this work (after the first version appeared on arXiv), one article~\cite{NEVOLI16} and two preprints~\cite{WLW17,NAM18} have appeared on the topic of this paper. In the article~\cite{NEVOLI16} pathwise (but not path-by-path) uniqueness is shown for the sCE under Krylov--R\"ockner conditions in the subcritical case. The preprints~\cite{WLW17} and~\cite{NAM18} go almost up to the critical case for weak and strong solution to the SDEs, the latter showing also Sobolev regularity of the stochastic flow. Respectively, the former preprint~\cite{WLW17}, while staying within the subcritical case in the interior $[0,T)$, allows a singularity at time $T$ which matches, and actually goes slightly beyond, the critical case. In the latter preprint~\cite{NAM18}, the limiting case is $d/p + 2/q =1$ is reached when replacing the $L^q$ integrability condition by a $L^{q,1}$ condition (for Lorenz space $L^{q,1} \subsetneq L^q$).

\subsection{Criticality}

We now show that the LPS condition is subcritical with strict inequality and critical with equality in the condition $d/p + 2/q \leq 1$. We have already emphasized that we treat the critical case because no other approach is known to attack this case, but the paper includes also the subcritical case. The general intuitive idea is that, near the singularities of~$b$, the Gaussian velocity field is ``stronger'' than (or comparable to)~$b$, which results in avoiding non-uniqueness or blow-up of solutions. The name ``critical'' comes from the following scaling argument (done only in a heuristic way since it only serves as motivation).

Let $u \colon[ 0,T] \times\mathbb{R}^{d}\rightarrow\mathbb{R}$ be a solution to the sTE. For $\lambda>0$  and $\alpha \in \mathbb{R}$, we introduce the scaled function $u_{\lambda} \colon [ 0,T/\lambda^{\alpha}] \times\mathbb{R}^{d}\rightarrow\mathbb{R}$ defined as $u_{\lambda}(t,x) \coloneqq u(\lambda^{\alpha }t,\lambda x)$. We denote by $\partial_t u$ and $\nabla u$ the derivative of~$u$ in the first argument and the gradient in the second one (similarly for $u_{\lambda}$). Since $\partial_t u_{\lambda}(t,x) =\lambda^{\alpha} \partial_t u(\lambda^{\alpha }t,\lambda x)$ and $\nabla u_{\lambda}(t,x)=\lambda\nabla u(\lambda^{\alpha}t,\lambda x)$, we get that $u_{\lambda}$ satisfies formally 
\begin{equation*}
\partial_t u_{\lambda}(t,x)+b_{\lambda}(t,x)\cdot\nabla
u_{\lambda}(t,x)+\lambda^{\alpha-1}\nabla u_{\lambda}(t,x)\circ W^{\prime}(\lambda^{\alpha}t)=0 ,
\end{equation*}
where $b_{\lambda}(t,x)=\lambda^{\alpha-1}b(\lambda^{\alpha}t,\lambda x)$ is the rescaled drift and $W^{\prime}(\lambda^{\alpha}t)$ formally denotes the derivative of $W$ at time $\lambda^{\alpha}t$. We now want to write the stochastic part in terms of a new Brownian motion. For this purpose, we define a process $(W_{\lambda
}(t))_{t\geq0}$, via $W_{\lambda}(t)\coloneqq \lambda^{-\alpha /2}W(\lambda^{\alpha}t)$ and notice that $W_{\lambda}$ is a Brownian motion with $W_{\lambda}^{\prime}(t)=\lambda^{\alpha/2}W^{\prime}(\lambda^{\alpha}t)$. Thus, the previous equation becomes
\begin{equation*}
\partial_t u_{\lambda}(t,x)+b_{\lambda}(t,x)\cdot\nabla
u_{\lambda}(t,x)+\lambda^{\alpha/2-1}\nabla u_{\lambda}(t,x)\circ W_{\lambda
}^{\prime}(t)=0 .
\end{equation*}
We first choose $\alpha=2$ such that the stochastic part $\lambda^{\alpha /2-1}\nabla u_{\lambda}(t,x)\circ W_{\lambda}^{\prime}(t)$ is comparable to the derivative in time $\partial_t u_{\lambda}$. Notice that this is the parabolic scaling, although sTE is not parabolic (but as we will see below, a basic idea of our approach is that certain expected values of the solution satisfy parabolic equations for which the above scaling is the relevant one). Next we require that, for small $\lambda$, the rescaled drift $b_{\lambda}$ becomes small (or at least controlled) in some suitable norm (in our case, $L^{q}( 0,T;L^{p}(\mathbb{R}^{d},\mathbb{R}^{d}) )$). It is easy to see that 
\begin{equation*}
\Vert b_{\lambda}\Vert_{L^{q}(0,T/\lambda^{2};L^{p})
}=\lambda^{1-(2/q+d/p)}\Vert b\Vert_{L^{q}(0,T;L^{p}) }
\end{equation*}
(here, the exponent~$d$ comes from rescaling in space and the exponent $2$ from rescaling in time and the choice $\alpha =2$). In conclusion, we find that
\begin{itemize}
 \item if the LPS condition holds with strict inequality, then $\Vert b_{\lambda}\Vert_{L^{q}(0,T/\lambda^{2};L^{p})} \to 0$ as $\lambda\rightarrow0$: the stochastic term dominates and we expect a regularizing effect (subcritical case);
 \item if the LPS condition holds with equality, then $\Vert b_{\lambda}\Vert_{L^{q}(0,T/\lambda^{\alpha};L^{p}) }=\Vert b\Vert_{L^{q}(0,T;L^{p}) }$ remains constant: the deterministic drift and the stochastic forcing are comparable (critical case).
\end{itemize}

This intuitively explains why the analysis of the critical case is more difficult. Notice that if the LPS condition does not hold, then we expect the drift to dominate, so that a general result for regularization by noise is probably false. In this sense, the LPS condition should be regarded as an optimal condition for expecting regularity of solutions.

\subsection{Regularity results for the sPDEs\label{intro section regularity}}

Concerning regularity, we proceed in a unified approach to attack the sTE and the sCE simultaneously (but for the sCE we have to assume the LPS condition also on~$\diverg b$). In fact, we shall treat a generalized stochastic equation of transport type which contains both the sTE and the sCE as special cases. For this equation we prove a regularity result which contains as a particular case the following:

\begin{theorem}
Assume the LPS condition on~$b$ \textup{(}and also on~$\diverg b$ for the sCE\textup{)}. If~$u_{0}$ is of class $\cap_{r\geq1}W^{1,r}(\mathbb{R}^{d})$, then there exists a solution to the sTE \textup{(}similarly for the sCE\textup{)} which is of class $\cap_{m\geq1}W_{\loc}^{1,m}( \mathbb{R}^{d})$.
\end{theorem}

A more detailed statement is given in Section~\ref{subsection regularity results} below. This result is false for $\sigma =0$, as we mentioned in Section~\ref{subsection deterministic}. Referring to some of the pathologies which may happen in the deterministic case, we may say that, under regular
initial conditions, noise prevents the emergence of \textit{shocks} (discontinuities) for the sTE, and \textit{singularities} of the density for
the sCE (the mass at time $t$ has a locally bounded density with respect to Lebesgue measure).

The method of proof is completely new. It is of analytic nature, based on PDE manipulations and estimates, opposite to the methods used before in \cite{FLAGUBPRI12,FEDFLA13a,MOHNILPRO15} and which are based on a preliminary construction of the stochastic flow for the sDE. We believe that, apart from the result, this new method of proof is the first important technical achievement of this paper (see Section~\ref{subsection_strategy_regularity} for a detailed description of the central ingredients of our method).

We now want to give some details on the precise statements, the regularity assumptions on drift and the strategy of proof for some known regularity results for the sTE, for the purpose of comparison with the results presented here. The paper~\cite{FLAGUBPRI12} deals with the case of \textit{H\"{o}lder continuous bounded drift} and is based on the construction of the stochastic flow from~\cite{FLAGUBPRI11}. The paper~\cite{FEDFLA13a} is concerned with the class called in the sequel as \textit{Krylov--R\"{o}ckner class}, after~\cite{KRYROE05}, where pathwise uniqueness and other results are proved for the sDE. We say that a vector field $f \colon [ 0,T] \times\mathbb{R}^{d}\rightarrow\mathbb{R}^{d}$ satisfies the Krylov--R\"{o}ckner (KR) condition if the LPS condition holds with strict inequality
\begin{equation*}
\frac{d}{p}+\frac{2}{q}<1  
\end{equation*}
and we shall write $f\in\mathcal{KR}(p,q)$. The improvement from $\frac{d}{p}+\frac{2}{q}<1$ to $\frac{d}{p}+\frac{2}{q}=1$ appeared also in the theory of 3D Navier--Stokes equations and required new techniques (which in turn opened new research directions on $L^{\infty}(0,T;L^{d}(\mathbb{R}^{d},\mathbb{R}^{d}) )$ regularity). Also here it requires a completely new approach. Under the condition $\frac{d}{p}+\frac {2}{q}=1$, we do not know how to solve the sDE directly (see however the recent preprints~\cite{NAM18,WLW17} mentioned above); even in a weak sense, by Girsanov theorem, the strict inequality seems to be needed (\cite{KRYROE05,PORTENKO90,GYOMAR01}). Similarly as in~\cite{FLAGUBPRI12}, the proof of regularity of solutions of the sTE from~\cite{FEDFLA13a} is based on the construction of stochastic flows for the sDE and their regularity in terms of weak differentiability. Finally,~\cite{MOHNILPRO15} and~\cite{NIL15} treat the case of \emph{bounded measurable drift} and, in~\cite{NIL15}, fractional Brownian motion (the classical work under this condition on pathwise uniqueness for the sDE is~\cite{VERETENNIKOV81}), again starting from a weak differentiability result for stochastic flows, proved however with methods different from~\cite{FEDFLA13a}.

Let us mention that proving that noise prevents blow-up or stabilizes the system (in cases where blow-up or instability phenomena are possible in the deterministic situation) is an intriguing
problem that is under investigation also for other equations, different from transport ones, see e.g.\ \cite{BARROCZHA17,BIABLOYAN16,CARCRALANROB07,CHOGUB15,DEBTSU11,DELFLAVIN14,DUBREV17,FLAGUBPRI10a,FLAMAUNEK14,GESSOU17,SCHEUTZOW93}.

\subsection{Uniqueness results for the sPDEs\label{intro section uniqueness}}

The second issue of our work is uniqueness of weak solutions to equations of continuity (and transport) type. More precisely, we prove a \textit{path-by-path uniqueness} result via a duality approach, which relies on the regularity results described in Section~\ref{intro section regularity}. When uniqueness is understood in a class of weak solutions, then the adjoint existence result must be in a class of sufficiently regular solutions (which is why the assumption for path-by-path uniqueness for the sCE will be the assumption for regularity for the sTE and vice versa); for this reason this approach cannot be applied in the deterministic case, when~$b$ is not sufficiently regular.

By path-by-path uniqueness we mean something stronger than pathwise uniqueness, namely that given $\omega $ a.s., the deterministic PDE
corresponding to that particular $\omega $ has a unique weak solution (note that our sPDE can be reformulated as a random PDE, which then can be read in a proper sense at $\omega$ fixed). Instead, pathwise uniqueness means that two processes, hence families indexed by $\omega $, both solutions of the equation, coincide for a.e.~$\omega $. We prove:

\begin{theorem}
Assume the LPS condition on~$b$ \textup{(}and also on~$\diverg b$ for the sTE\textup{)}. Then the sCE \textup{(}similarly the sTE\textup{)} has path-by-path uniqueness of weak $L^{m}$-solutions, for every finite $m$.
\end{theorem}

A more precise statement is given in Section~\ref{subsection path by path uniq} below. No other method is known to produce such a strong result of uniqueness. This duality method in the stochastic setting is the second important technical achievement of this paper.

The intuitive reason why, by duality, one can prove \textit{path-by-path} uniqueness (usually so difficult to be proven) is the following one. The duality approach gives us an identity of the form 
\begin{equation}
\left\langle \rho_{t},\varphi\right\rangle =\left\langle \rho_{0} 
,u_{0}^{t,\varphi}\right\rangle \label{duality intro} 
\end{equation}
where $\rho_{t}$ is any weak solution of the sCE ($\rho_{t}$ is the density of $\mu_{t}$) with initial condition $\rho_{0}$ and $\left(  u_{s}^{t,\varphi
}\right)  _{s\in\left[  0,t\right]  }$ is any regular solution of the sTE rewritten in backward form with final condition~$\varphi$ at time $t$. As we
shall see below, we use an approximate version of~\eqref{duality intro}, but the idea we want to explain here is the same. Identity~\eqref{duality intro}
holds a.s.~in~$\Omega$, for any given~$\varphi$ and~$t$. But taking a dense (in a suitable topology) countable set $\mathcal{D}$ of~$\varphi$'s, we 
have~\eqref{duality intro} for a.e.~$\omega$, uniformly on $\mathcal{D}$ and thus we may identify $\rho_{t}$. This is the reason why this approach is so powerful to prove path-by-path results. Of course behind this simple idea, the main technical point is the regularity of the solutions to the sTE, which
makes it possible to prove an identity of the form~\eqref{duality intro} for weak solutions $\rho_{t}$, for all those~$\omega$'s such that $u_{s}^{t,\varphi}$
is regular enough.

Concerning other uniqueness results for the sTE with poor drift, let us briefly comment on a few of them. In~\cite{FLAGUBPRI10} the case of H\"{o}lder continuous bounded drift is treated, by means of the differentiable flow associated to the sDE; \cite{NEVOLI15} extends the result and the approach to drifts in KR class with zero divergence. The paper~\cite{CATELLIER16} extends the results to the sTE with H\"older continuous drift but driven by fractional Brownian motion, relying again on the flow; the technique used there for the analysis of the sDE itself is instead different from~\cite{FLAGUBPRI10} and leads to path-by-path uniqueness. The paper~\cite{ATTFLA11} assumes weakly differentiable drift but relaxes the assumption on the divergence of the drift, with respect to the deterministic works~\cite{DIPELI89,AMBROSIO04}. The papers~\cite{MAU11,FEDNEVOLI18} use Wiener chaos expansion techniques to obtain uniqueness for the sTE for drifts close to KR class, see~\cite{MAU11}, or even beyond, see~\cite{FEDNEVOLI18}, at the price of uniqueness in a smaller class (namely among solutions adapted to the \textit{Brownian} filtration). A full solution of the uniqueness problem in the KR class was still open (apart from the paper~\cite{NEVOLI16} and the recent preprints~\cite{NAM18,WLW17} mentioned above) and this is a by-product of this paper, which solves the problem in a stronger sense in two directions:
\begin{itemize}
\item[i)] path-by-path uniqueness instead of pathwise uniqueness;
\item[ii)] drift in the LPS class instead of only KR class.
\end{itemize}
Let us mention that the approach to uniqueness of~\cite{ATTFLA11} shares some technical steps with the results described in Section~\ref{intro section regularity}: renormalization of solutions (in the sense of~\cite{DIPELI89}), It\^{o} reformulation of the Stratonovich equation and then expected value (a Laplacian arises from this procedure). However, in~\cite{ATTFLA11} this approach has been applied directly to uniqueness of weak solutions so the renormalization step required weak differentiability of the drift. Instead, here we deal with regularity of solutions and thus the renormalization is applied to regular solutions of approximate problems and no additional assumption on the drift is needed.

Finally, we comment on some related uniqueness results in the nonlinear case. The duality technique has been used in~\cite{GESMAU18}, for scalar conservation laws with linear transport noise, and in~\cite{GESSMI19}, for nonlinear transport noise, but in a different way and without producing a path-by-path uniqueness result. Other results on uniqueness by noise are available with different techniques and/or different choices of noise, e.g.~\cite{BIA13} for a dyadic model of turbulence and~\cite{BALGYOPAR94} for a parabolic model.

\subsection{Results for the sDE\label{intro section SDE}}

The last issue of our paper is to provide existence, uniqueness and regularity of stochastic flows for the sDE, imposing merely the LPS condition. The
strategy here is to deduce such results from the path-by-path uniqueness result established in Section~\ref{intro section uniqueness}. To understand the novelties, let us recall that the more general strong well-posedness result for the sDE is due to~\cite{KRYROE05} under the KR condition on~$b$. To simplify the exposition and unify the discussion of the literature, let us consider the autonomous case $b( t,x) =b( x) $ and an assumption of the form $b\in L^{p}( \mathbb{R}^{d},\mathbb{R}^{d}) $ (depending on the reference, various locality conditions and behavior at infinity are assumed). The condition $p>d$ seems to be the limit case for solvability in all approaches, see for instance \cite{KRYROE05,GYOMAR01,PORTENKO90,CHEZHA95,STANNAT99}, whether they are based on Girsanov theorem, Krylov estimates, parabolic theory or Dirichlet forms. There are some results on weak well-posedness for measure-valued drifts, see~\cite{BASCHE03}, and distribution-valued drifts, see~\cite{FLAISSRUS17,DELDIE16,CANCHO18}, but it is unclear whether they apply to the limit case $p=d$: for example, the result in~\cite{BASCHE03}, when restricted to measures with density~$b$ with respect to the Lebesgue measure, requires $p>d$, see \cite[Example 2.3]{BASCHE03}. The present paper is the first one to give information on sDEs in the limit case $p=d$ (apart from~\cite{NAM18,WLW17}).

Since path-by-path solvability is another issue related to our results, let us mention the paper~\cite{DAVIE07}, where the drift is bounded measurable: for a.e.~$\omega 
$, there exists one and only one solution. New results for several classes of noise and drift have been obtained by~\cite{CATGUB13}. In general, the problem of path-by-path solvability of an sDE with poor drift is extremely difficult, compared to pathwise uniqueness which is already nontrivial. Thus, it is remarkable that the
approach by duality developed here gives results in this direction. 

Our contribution on the sDE is threefold: existence, uniqueness and regularity of Lagrangian flows, pathwise uniqueness from a diffuse initial datum and path-by-path uniqueness from given initial condition. The following subsections detail these three classes of results. 

\subsubsection{Lagrangian flows}

We prove a well-posedness result among Lagrangian flows (see below for more explanations) under the LPS condition on the drift:

\begin{theorem}
Under LPS condition, for a.e.~$\omega$, there exists a unique Lagrangian flow $\Phi^\omega$ solving the sDE at~$\omega$ fixed. This flow is, at fixed time, $W_{\loc}^{1,m}(\mathbb{R}^{d},\mathbb{R}^{d})$-regular for every finite $m$, in particular it has a $C^\alpha(\mathbb{R}^{d},\mathbb{R}^{d})$ version \textup{(}at fixed time\textup{)} for every $\alpha<1$.
\end{theorem}

Uniqueness will follow from uniqueness of the sCE, regularity from regularity of the solution to the sTE. The result is new because our uniqueness result is path-by-path: for a.e.~$\omega$, two Lagrangian flows solving the sDE with that~$\omega$ fixed must coincide (notice that the sDE has a clear path-by-path meaning). A Lagrangian flow~$\Phi$, solving a given ODE, is a generalized flow, in the sense of~\cite{DIPELI89,AMBROSIO04}: a measurable map $x\mapsto\Phi_{t}(x)$ with a certain non-contracting property, such that $t\mapsto\Phi_{t}(x)$ verifies that ODE for a.e.~initial condition~$x \in \mathbb{R}^{d}$. However, in general, we do not construct solutions of the sDE in a classical sense, corresponding to a given initial condition $X_{0}=x$. In fact, we do not know whether or not strong solutions exist and are unique under the LPS condition with $\frac{d}{p}+\frac{2}{q}=1$ (while for $\frac {d}{p}+\frac{2}{q}<1$ strong solutions do exist, see~\cite{KRYROE05}).

Let us mention that regularity under more restrictive assumptions was already proved, see for example~\cite{FEDFLA13b,MOHNILPRO15} or~\cite{CATGUB13} (also for fractional Brownian motion). However, these results do not cover the full LPS condition and their proofs are based on the sDE, rather than on the sTE.

\subsubsection{Pathwise uniqueness from a diffuse initial datum}

We also prove a (classical) pathwise uniqueness result under the LPS condition, when the initial datum has a diffuse law. This is done by exploiting the regularity result of the sTE and by using a duality technique similar to the one mentioned before.

\begin{theorem}
If~$X_0$ is a diffuse random variable \textup{(}not a single~$x$\textup{)} on $\mathbb{R}^d$, then pathwise uniqueness holds among solutions having diffuse marginal laws \textup{(}more precisely, such that the law of~$X_t$ has a density in $L^\infty([0,T];L^m(\mathbb{R}^d))$, for a suitable $m$\textup{)}.
\end{theorem}

Finally we notice that uniqueness of the law of solutions (or at least of their one-dimensional marginals, namely the solutions of Fokker-Planck equations) may hold true for very irregular drift, i.e.~$b\in L^{2}$, if diffuse initial conditions with suitable density are considered; see \cite{FIG08,BOGDAPROE11}.

\subsubsection{Results of path-by-path uniqueness from given initial condition}

When the regularity results for the stochastic equation of transport type is improved from $W^{1,p}$ to $C^{1}$-regularity, then the uniqueness results of Section~\ref{intro section uniqueness} for the sCE holds in the very general class of finite measures and it is a path-by-path uniqueness result. As a consequence, we get an analogous path-by-path uniqueness result for the sDE with classical given initial conditions, a result competitive with~\cite{DAVIE07} and~\cite{CATGUB13}. The main problem is to find assumptions, as weak as possible, on the drift~$b$ which are sufficient to guarantee $C^{1}$ regularity of the solutions. We describe two cases. The first one, which follows the strategy described in Section 1.1, is when the weak derivatives of~$b$ (instead of only~$b$ itself) belongs to the LPS class, that is $\partial _{i}b\in \mathcal{LPS}( p,q) $ for $i =1,\ldots,d$. However, since this is a weak differentiability assumption, it is less general than expected. The second case is when~$b$ is H\"{o}lder continuous (in space) and bounded, but here we have to refer to~\cite{FLAGUBPRI10,FLAGUBPRI12} for the proof of the main regularity results.

\begin{theorem}
If~$Db$ belongs to the LPS class or if~$b$ is H\"older continuous \textup{(}in space\textup{)}, then, for a.e.~$\omega$, for every~$x$ in $\mathbb{R}^d$, there exists a unique solution to the sDE, starting from~$x$, at~$\omega$ fixed.
\end{theorem}

Notice that the ``good subset'' of~$\Omega$ is independent of the initial condition~$x$; this is not obvious from the approaches of~\cite{CATGUB13,DAVIE07}, due to the application of Girsanov transformation for a given initial condition.

Let us mention that in~\cite{SHA16}, generalized in~\cite{PRI18} to the case of L\'evy noise, path-by-path uniqueness is shown, from a fixed initial condition, for a H\"older continuous drift, using the regularity of the flow. This approach is the translation at the sDE level of the duality technique for the sPDE.

\subsubsection{Summary on uniqueness results}

Since the reader might not be acquainted with the various types of uniqueness, we resume here the possible path-by-path uniqueness results and their links.

\setlength{\extrarowheight}{5pt}
\begin{center}
\begin{tabular}{ >{\centering}m{4.5cm} >{\centering}m{2cm}  >{\centering}m{4.5cm} }
path-by-path uniqueness among trajectories & $\Rightarrow$ & pathwise uniqueness for deterministic initial data \tabularnewline
$\Downarrow$ &  \  &  $\Downarrow$ \tabularnewline
path-by-path uniqueness among flows & ``$\Rightarrow$'' & pathwise uniqueness for diffuse initial data
\end{tabular}
\end{center}

Let us explain more in detail these implications (this is in parts heuristics and must not be taken as rigorous proofs):
\begin{itemize}
\item Path-by-path uniqueness among trajectories implies path-by-path uniqueness among flows: Assume path-by-path (or pathwise) uniqueness among trajectories and let $\Phi$, $\Psi$ be two flows solving the sDE. Then, for a.e.~$x$ in $\mathbb{R}^d$, $\Phi(x)$ and~$\Psi(x)$ are solutions to the sDE, starting from~$x$, so, by uniqueness, they must coincide and hence $\Phi=\Psi$ a.e..
\item Path-by-path uniqueness among trajectories implies pathwise uniqueness for de\-terministic initial data: Assume path-by-path uniqueness among trajectories and let $X$, $Y$ be two adapted processes which solve the sDE. Then, for a.e.~$\omega$, $X(\omega)$ and $Y(\omega)$ solve the sDE for that fixed~$\omega$, so they must coincide and hence $X=Y$ a.e..
\item Pathwise uniqueness for deterministic initial data implies pathwise uniqueness for diffuse initial data: Assume pathwise uniqueness for deterministic initial data and let $X$, $Y$ be two solutions, on a probability space $(\Omega,\mc{A},P)$, starting from a diffuse initial datum~$X_0$. For~$x$ in $\mathbb{R}^d$, define the set $\Omega_x=\{\omega\in\Omega \colon X_0(\omega)=x\}$. Then, for $(X_0)_\#P$-a.e.~$x$, $X$ and~$Y$, restricted to $\Omega_x$, solve the sDE starting from~$x$, so they must coincide and hence $X=Y$ a.e..
\item Path-by-path uniqueness among flows (with non-concentration properties) ``implies'' pathwise uniqueness for diffuse initial data: The quotation marks are here for two reasons: because the general proof is more complicated than the idea below and because the pathwise uniqueness is not among all the processes (with diffuse initial data), but a restriction is needed to transfer the non-concentration property. Assume path-by-path uniqueness among flows and let $X$, $Y$ be two solutions on a probability space $(\Omega,\mc{A},P)$, starting from a diffuse initial datum~$X_0$. We give the idea in the case $\Omega=C([0,T];\mathbb{R}^d)\times B_R\ni\omega=(\gamma,x)$ (the Wiener space times the space of initial datum), $P=Q\times\mc{L}^d$, where $Q$ is the Wiener measure, $W(\gamma,x)=\gamma$, $X_0(\gamma,x)=x$. In this case (which is a model for the general case), for $Q$-a.e.~$\gamma$, $\Phi(\gamma,\cdot)$ and $\Psi(\gamma,\cdot)$ are flows solving the SDE for that fixed~$\omega$. If they have the required non-concentration properties, then, by uniqueness, they must coincide. Hence uniqueness holds among processes~$X$, with diffuse~$X_0$, such that $X(\gamma,\cdot)$ has a certain non-concentration property; this is the restriction we need.
\end{itemize}

We will prove: path-by-path uniqueness among Lagrangian flows, when~$b$ is in LPS class; path-by-path uniqueness among solutions starting from a fixed initial point, when~$b$ and~$Db$ are in LPS class or when~$b$ is H\"older continuous (in space). We will develop in detail pathwise uniqueness from a diffuse initial datum in Section~\ref{firstpathwise} (where the last implication will be proved) and in Section~\ref{secondpathwise} (where a somehow more general result will be given).

\subsubsection{Examples}

In Section~\ref{ex_section} we give several examples of equations with irregular drift of two categories:
\begin{itemize}
\item[i)] on one side, several examples of drifts which in the deterministic case give rise to non-uniqueness, discontinuity or shocks in the flow, while in the stochastic case our results apply and these problems disappear;
\item[ii)] on the other side, a counterexample of a drift outside of the LPS class, for which even the sDE is ill-posed.
\end{itemize}

\subsection{Concluding remarks and generalizations}\label{general}

The three classes of results described above are listed in logical order: we need the regularity results of Section~\ref{intro section regularity} for the sTE and sCE to prove the uniqueness results of Section~\ref{intro section uniqueness} for the sCE and sTE by duality; then we deduce the results of Section~\ref{intro section SDE} for the sDE from such uniqueness results. The fact that regularity for transport equations (with poor drift) is the starting point marks the difference with the deterministic theory, where such kind of results are absent. Hence, the results and techniques of the present paper are not generalizations of deterministic ones.

The two most innovative technical tools developed in this work are the analytic proof of regularity (as stated in Section~\ref{intro section regularity}) and the path-by-path duality argument yielding uniqueness in this very strong sense. The generality of the LPS condition seems to be unreachable with more classical tools, based on a direct analysis of the sDE. Moreover, in principle some of the analytic steps of Section~\ref{intro section regularity} and the duality argument could be applied to other classes of stochastic equations; however, the renormalization step in the regularity proof is quite peculiar of transport equations.

The noise considered in this work is the simplest one, in the class of multiplicative noise of transport type. The reason for this choice is that it suffices to prove the regularization phenomena and the exposition will not be obscured by unnecessary details. However, for nonlinear problems it seems that more structured noise is needed, see~\cite{FLAGUBPRI10a,DELFLAVIN14}. So it is natural to ask whether the results of this paper extend to such noise. Let us briefly discuss this issue. The more general sDE takes the form 
\begin{equation}
dX =b( t,X ) dt+\sum_{k=1}^{\infty }\sigma _{k}(X) \circ dW_{t}^{k},\qquad X_{0}=x  \label{SDE general}
\end{equation}
where $\sigma _{k}:\mathbb{R}^{d}\rightarrow \mathbb{R}^{d}$ and $W^{k}$ are
independent Brownian motions, and the associated sTE, sCE are now
\begin{equation}
du+b\cdot \nabla udt+\sum_{k=1}^{\infty }\sigma _{k}\cdot \nabla u\circ dW_{t}^{k}=0,\qquad u|_{t=0}=u_{0}  \label{stoch transport general}\end{equation}
\begin{equation}
d\mu +\diverg ( b \mu ) dt+\sum_{k=1}^{\infty }\diverg  
( \sigma _{k}\mu ) \circ dW_{t}^{k}=0,\qquad \mu |_{t=0}=\mu
_{0}  \label{stoch cont general}
\end{equation} 
Concerning the assumptions on $\sigma _{k}$, for simplicity, think of the case when they are of class $C_{b}^{4}$ with proper summability in~$k$. In order to generalize the regularity theory (Section~\ref{intro section regularity}) it is necessary to be able to perform parabolic estimates, and thus, the generator associated to this sDE must be strongly elliptic; a simple sufficient condition is that the covariance matrix function $Q(x,y) \coloneqq \sum_{k=1}^{\infty }\sigma _{k}( x) \otimes \sigma_{k}( y) $ of the random field $\eta ( t,x) =\sum_{k=1}^{\infty }\sigma _{k}( x) W_{t}^{k}$ depends only on $x-y$ (namely $\eta ( t,x) $ is space-homogeneous), $\diverg  \sigma _{k}( x) =0$ (this simplifies several lower order terms) and for the function $Q( x) =Q( x-y) $ we have 
\begin{equation*}
\det Q( 0) \neq 0 .
\end{equation*} 
This replaces the assumption $\sigma \neq 0$.

The duality argument (Section~\ref{intro section uniqueness}) is very general and in principle it does not require any special structure except the linearity of the equations. However, in the form developed here, we use auxiliary random PDEs associated to the sPDEs via the simple transformation $v( t,x) \coloneqq u( t,x+\sigma W_{t}) $;\ we do this in order to avoid troubles with backward and forward sPDEs at the same time. But this simple transformation requires additive noise. In the case of multiplicative noise, one has to consider the auxiliary stochastic equation 
\begin{equation}
dY = \sum_{k=1}^{\infty }\sigma _{k}( Y) \circ dW_{t}^{k},\qquad Y_{0}=y  \label{auxiliary SDE}
\end{equation} 
and its stochastic flow of diffeomorphisms $\psi_{t}( x) $, and use the transformation $v(t,x) \coloneqq u(t,\psi_{t}(x))$. This new random field satisfies 
\begin{equation*}
\partial_t v +\widetilde{b}\cdot \nabla v=0
\end{equation*} 
where
\begin{equation*}
\td{b}(t,x)\coloneqq D\psi^{-1}_t(\psi_t(x))b(t,\psi_t(x)),
\end{equation*}
and the duality arguments can be repeated, in the form developed here. The uniqueness results mentioned in Section~\ref{intro section uniqueness} then extend to this case.

The path-by-path analysis of the sDE~\eqref{SDE general} may look a priori less natural since this equation does not have a pathwise interpretation. However, when the coefficients $\sigma _{k}$ are sufficiently regular to generate, for the auxiliary equation~\eqref{auxiliary SDE}, a stochastic flow of diffeomorphisms $\psi _{t}( x) $, then we may give a (formally) alternative formulation of~\eqref{SDE general} as a random differential equation, of the form 
\begin{equation*}
dZ = \td{b}(t,Z) dt,
\end{equation*}
in analogy with the random PDE for the auxiliary variable $v(t,x) $ above. This equation can be studied pathwise, with the techniques of Section~\ref{intro section SDE}. Here, however, we feel that more work is needed in order to connect the results with the more classical viewpoint of equation~\eqref{SDE general} and thus we refrain to express strong claims here.

Concerning the path-by-path uniqueness, say for the sDE, note that this issue can be studied for any \emph{deterministic} path~$W$, not necessarily the sample paths from the Brownian motion. Hence, one can ask which conditions on a single, deterministic path~$W$ ensure uniqueness of~\eqref{SDE}, which is now a deterministic ODE. This is investigated in~\cite{CATGUB13}, where the concept of $(\rho,\gamma)$-irregular paths is given by means of Fourier analysis, and it is shown that such paths provide uniqueness for certain classes of non-Lipschitz drifts~$b$ (in particular if~$W$ is a sample path of the Brownian motion, uniqueness is shown for H\"older continuous drifts). In contrast to the present paper, the techniques used in~\cite{CATGUB13} are based on Young integration, and the results, when specialized to Brownian sample paths, are mostly concerned with H\"older continuous drifts. While for a general path it is not easy to verify the $(\rho,\gamma)$-irregularity condition, one can prove, see \cite[Proposition~1.4]{CHOGUB15},  that this condition implies that the path must be irregular (non-Lipschitz in time): this corresponds to the fact that a regular path does not regularize an ill-posed ODE, in general. It would be interesting to compare the $(\rho,\gamma)$-irregularity notion with the concept of truly rough paths (e.g.~\cite{FRISHE13}), which also quantifies the irregularity of a path. Another, somehow more explicit, sufficient condition on deterministic paths is given in \cite[equation (3.3)]{CHOGES19}, though it is used for the regularization of scalar conservation laws rather than ODEs. Here the regularization of nonlinear PDEs was achieved by means of a noise, that is here the derivative of the regularizing path, which is itself nonlinear and precisely multiplies the nonlinearity; see e.g.~\cite{CHOGES19,CHOGUB15,CHOGUB14}, and~\cite{GASGES19} for other pathwise arguments.

Throughout the paper, the drift~$b$ is assumed to be deterministic. In view of applications especially to nonlinear equations, it would be very important to extend the result to random drifts. While we do not see obstacles for the extension of the duality technique, being path-by-path in nature, the first step, namely the proof of regularity of solutions for the sTE, does not allow for such generalization: if the drift were random, then the equations for the moments of the derivative of the solution would not form a closed system. This is not simply a limitation of the techniques: there are in fact simple counterexamples to regularization by noise for general random drifts. Let us mention that, in some cases, it is possible to have regularization by noise even for random drifts, see~\cite{CATGUB13} and related work, assuming a suitable H\"older continuity of the drift, or~\cite{DUBREV16,OLIVEIRA19}, assuming Malliavin differentiability of the drift.

Finally, let us note that throughout the rest of the paper, concerning function spaces, we shall use for simplicity the same notation for scalar-valued and vector-valued functions (but it will be always clear from the context if the functions under consideration have values in~$\mathbb{R}^{d}$, like~$b$,~$X$,  or~$\Phi$, or in~$\mathbb{R}$, like~$c$ or solutions~$u$,~$v$).

\section{Regularity for sTE and sCE\label{section regularity}}

In order to unify the analysis of the sTE and sCE we introduce the stochastic generalized transport equation (\textit{sgTE}) in $\mathbb{R}^{d}$
\begin{equation}
\tag{sgTE}
du+(b\cdot\nabla u+cu)  dt+\sigma\nabla u\circ dW=0,\qquad
u|_{t=0}=u_{0} \label{SPDE 1} 
\end{equation}
where $b$, $\sigma$, $u$ and $u_{0}$ are as above for the sTE, $c \colon [0,T] \times\mathbb{R}^{d}\rightarrow\mathbb{R}$ and $W$ is a Brownian motion with respect to a given filtration $(\mc{G}_t)_t$. We shall prove regularity results for solutions to~\eqref{SPDE 1}.

\begin{remark}
We note that the case $c=0$ corresponds to~\eqref{stoch transport}, while the case $c=\diverg b$ corresponds to~\eqref{stoch cont}, with
\begin{equation}
du+\diverg  (bu)  dt+\sigma\diverg  
(  u\circ dW_{t} )  =0,\qquad u|_{t=0}=u_{0},
\label{sCE for density}
\end{equation}
where~$u$ stands for the density of the measure $\mu_t$ with respect to the Lebesgue measure.
\end{remark}

\subsection{Assumptions\label{subsection regularity assumptions}}

Throughout all the paper, we assume that $(\Omega,\mathcal{A},P)$ is a probability space, $(\mathcal{G}_t)_{t\in[0,T]}$ is a filtration satisfying the standard assumptions, that is, it is complete and right-continuous. The process $W$ denotes a Brownian motion with respect to $(\mathcal{G}_t)_t$, unless differently specified.

Concerning the general equation~\eqref{SPDE 1} we will always assume that we are in the purely stochastic case with $\sigma\neq0$ and that the coefficients~$b$ and~$c$ satisfy the following decomposition and regularity condition.

\begin{condition}[LPS+reg]
\label{LPSreg}
The fields~$b$ and~$c$ can be written as $b=b^{(1)}+b^{(2)}$, $c=c^{(1)}+c^{(2)}$, where 
\begin{enumerate}[font=\normalfont]
\item \emph{LPS-condition:} $b^{(1)}$, $c^{(1)}$ satisfy one of the following three assumptions:
 \begin{itemize}[font=\normalfont]
   \item[a)] $b^{(1)}$, $c^{(1)}$ are in $\mathcal{LPS}(p,q)$ for some $p$, $q$ in
$(2,\infty)$ \textup{(}with $\frac2q+\frac{d}{p}\le1$\textup{)} or $p=\infty$, $q=2$;
   \item[b)] $b^{(1)}$, $c^{(1)}$ are in $C([0,T];L^{d}(\mathbb{R}^{d}))$ with $d \geq 3$;
   \item[c)] $b^{(1)}$, $c^{(1)}$ are in $L^{\infty}([0,T];L^{d}(\mathbb{R}^{d}))$
with $d \geq 3$ and there hold
\[
\|b^{(1)}\|_{L^{\infty}([0,T];L^{d}(\mathbb{R}^{d}))}\le\delta \quad \text{and} \quad
\|c^{(1)}\|_{L^{\infty}([0,T];L^{d}(\mathbb{R}^{d}))} \le\delta,
\]
with $\delta$ small enough; precisely, given an exponent $m$ as in Theorem
\ref{theorem a priori estimate}, $\delta$ depends on $m,\sigma,d$, as given by
inequality~\eqref{precise smallness};
 \end{itemize}
   
\item \emph{Regularity condition:} $b^{(2)}$ is in $L^{2}([0,T];C^{1}_{\lin}(\mathbb{R}^{d}))$ and $c^{(2)}$
is in $L^{2}([0,T];C^{1}_{b}(\mathbb{R}^{d}))$, i.e., for a.e.~$t \in [0,T]$,
$b^{(2)}(t,\cdot)$ and $c^{(2)}(t,\cdot)$ are in $C^{1}(\mathbb{R}^{d})$ and
\begin{align*}
\|b^{(2)}\|_{L^{2}([0,T];C^{1}_{\lin}(\mathbb{R}^{d}))}^2 & \coloneqq \int^{T}_{0}\Big(
\Big\| \frac{b^{(2)}(s,\cdot)}{1+|\cdot|}\Big\|_{\infty}+\|Db^{(2)} 
(s,\cdot)\|_{\infty}\Big)^2 ds<\infty,\\
\|c^{(2)}\|_{L^{2}([0,T];C^{1}_{b}(\mathbb{R}^{d}))}^2 & \coloneqq \int^{T}_{0}\Big(
\|c^{(2)}(s,\cdot)\|_{\infty}+\|Dc^{(2)}(s,\cdot)\|_{\infty}\Big)^2 ds<\infty .
\end{align*}
\end{enumerate}
\textup{(}The expression ``$b$ is in a certain class $A$'' must be understood componentwise.\textup{)}
\end{condition}

\begin{remark}\label{L2time}
The hypotheses on $b^{(2)}$ and $c^{(2)}$ are slightly stronger than the natural ones, namely $b^{(2)}$ in $L^{1}([0,T];C^{1}_{\lin}(\mathbb{R}^{d}))$, $c^{(2)}$ in $L^{1}([0,T];C^{1}_{b}(\mathbb{R}^{d}))$: we require $L^2$~integrability in time instead of~$L^1$. This is mainly due to a technical point which will appear in Section~\ref{pbp_section}. However, with minor modifications, this assumption could be relaxed to $L^1$ integrability throughout this section.
\end{remark}

\begin{remark}
A simple extension of Condition~\ref{LPSreg} is to ask that $b=\sum_{j=1}^{N} \hat{b}^{(j)}$, where, for every~$j$, $\hat{b}^{(j)}$ is a vector field satisfying Condition~\ref{LPSreg} with exponents $p_{j}$, $q_{j}$ that can depend on~$j$; similarly for~$c$. This extension is easy and we refrain to discuss it explicitly.
\end{remark}

\begin{remark}
\label{Remark assumptions sCE}
The sTE is just equation~\eqref{SPDE 1} with $c=0$ and thus we do write explicitly the assumptions for~\eqref{stoch transport}. The sCE instead corresponds to~\eqref{SPDE 1} with $c=\diverg b$ and for completeness let us note that we hence need to assume for~\eqref{stoch cont} that we have $b=b^{(1)}+b^{(2)}$, with
\begin{itemize}[font=\normalfont]
  \item[(i$_{0}$)] for some $p,q\in(2,\infty)$, or $(p,q)
=(\infty,2)$, $b^{(1)},\diverg b^{(1)} \in \mathcal{LPS}(p,q)$;
  \item[(i$_{1}$)] for $(p,q)  =(d,\infty)$, $d\geq3$, either we assume $b^{(1)}, \diverg b^{(1)} \in C([0,T];L^{d}(\mathbb{R}^{d}))$ or we require the smallness assumption in Condition~\ref{LPSreg}, \textup{1c)};
  \item[(ii)] $b^{(2)},\diverg b^{(2)}\in L^{2}(0,T;C^{1}(\mathbb{R}^{d}))$, with
\[
\int_{0}^{T}\Big(\Big\Vert \frac{b^{(2)}(s,\cdot)}{1+\vert \cdot\vert} \Big\Vert_{\infty
}+\big\Vert Db^{(2)}(s)  \big\Vert_{\infty}+ \big\Vert D\diverg b^{(2)} (s)
\big\Vert_{\infty}\Big)^2  ds<\infty .
\]
\end{itemize}
\end{remark}

\subsection{Strategy of proof}
\label{subsection_strategy_regularity}

In order to prove the regularity results, we follow the approach of a priori estimates: we prove regularity estimates for the smooth solutions of approximate problems with smooth coefficients, be careful to show that the regularity estimates have constants independent of the approximation; then we deduce the regularity for the solution of the limit problem by passing to the limit.

The strategy of proof is made of several steps which bear some similarities with the computations done in literature of theoretical physics of passive scalars, see for instance~\cite{CHAGAWHORKUPVER03}.

First, we differentiate the sgTE (which is possible because we deal with smooth solutions of regularized problems), with the purpose of estimating the derivatives of the solutions. However, terms like $\partial_{i}b_{k}$ appear. In the deterministic case, \emph{unless~$b$ is Lipschitz}, these terms spoil any attempt to prove differentiability of solutions by this method. In the stochastic case, we shall integrate these bad terms by parts at the price of a second derivative of the solution, which however will be controlled, as it will be explained below.

Second, we use the very important property of transport type equations of being invariant under certain transformations of the solution. For the classical sTE, the typical transformation is $u \mapsto \beta(u)$ where $\beta\in C^{1}(\mathbb{R})$: if~$u$ is a solution, then $\beta(u)$ is (at least formally) again a solution. For
regular solutions, as in our case, this can be made rigorous; let us only mention that, for weak solutions, this is a major issue, which gives rise to the
concept of renormalized solutions~\cite{DIPELI89} (namely those for which $\beta(u)$ is again a weak solution) and the so called commutator lemma; we
do not meet these problems here, in the framework of regular solutions. Nevertheless, to recall the issue, we shall call this step \emph{renormalization}, namely that suitable transformations of the solution lead to solutions. In our case, since we consider the differentiated sgTE, we work on the level of derivatives of the solution~$u$ and therefore we apply transformations to $\partial_{i}u$. In order to find a closed system, we have to consider, as transformations, all possible products of $\partial_{i}u$, and~$u$ itself. This leads to some complications in the book-keeping of indices, but the essential idea is still the renormalization principle.

Third, we reformulate the sPDE from the Stratonovich to the It\^{o} form. The corrector is a second order differential operator. It is strongly elliptic in itself, but combined with the It\^{o} term (containing first derivatives of solutions), it does not give a parabolic character to the equation. The equation is indeed equivalent to the original, hyperbolic (time-reversible) formulation.

Fourth, we take the expectation. This projection annihilates the It\^{o} term and gives a true parabolic equation. The expected value of powers of~$\partial_{i}u$ (or any product of them) solves a parabolic equation, and, as a system in all possible products, it is a closed system. For other functionals of the solution, as the two-point correlation function $E[u(t,x)  u(t,y)]$, the fact that a closed parabolic equation arises was a basic tool in the theory of passive scalars~\cite{CHAGAWHORKUPVER03}.

Finally, on the parabolic equation we perform energy-type estimates. The elliptic term puts into play, on the positive side of the estimates, terms like
$\nabla E[ (\partial_{i}u)^{m}]$. They are the key tool to estimate those terms coming from the partial integration of $\partial_{i}b_{k}$ (see the comments above). The good parabolic terms $\nabla E[(\partial_{i}u)^{m}]$ come from the Stratonovich-It\^{o} corrector, after projection by the expected value.  This is the technical difference to the deterministic case. 

\subsection{Preparation}

The following preliminary lemma is essentially known, although maybe not explicitly written in all details in the literature; we shall therefore sketch the proof. As
explained in the last section, given non-smooth coefficients, we shall approximate them with smooth ones. Their role is only to allow us to perform certain computations on the solutions (such as It\^{o} formula, finite expected values, finite integrals on $\mathbb{R}^{d}$ and so on). More precisely, the outcome of the next lemma are $C_{c}^{\infty}$-estimates (infinitely differentiable with compact support in all variables) in space for all times, for the solutions corresponding to the equation with smooth (regularized) coefficients. However, we emphasize that these estimates are not uniform in the approximations, in contrast to our main regularity estimates concerning Sobolev-type regularity established later on in Theorem~\ref{theorem a priori estimate}. 

\begin{lemma}
\label{lemma preliminare sul caso smooth}
If $b , c\in C_{c}^{\infty}([0,T] \times \mathbb{R}^{d})$, $u_{0} \in C_{c}^{\infty}(\mathbb{R}^{d})$, then there exists an adapted solution~$u$ of equation~\eqref{SPDE 1} with paths of class $C([0,T];C_{c}^{\infty}(\mathbb{R}^{d}))$ \textup{(}where the support of~$u$ depends on the path\textup{)}. We have
\begin{equation}
\sup_{(t,x) \in [0,T] \times \mathbb{R}^{d}}E \big[ \vert D^{\alpha}u(t,x) \vert^{r}\big]
<\infty \label{ineq 1 lemma smooth}
\end{equation}
for every $\alpha\geq0$ and $r\geq1$. Moreover, we have
\begin{equation}
\sup_{(t,x,\omega)  \in [0,T] \times \mathbb{R}^{d} \times \Omega} \vert u(t,x,\omega) \vert
\leq \Vert u_{0} \Vert_{\infty} e^{\int_{0}^{T}\Vert c (s,\cdot) \Vert_{\infty} ds} \label{ineq 2 lemma smooth}
\end{equation}
and for every $r,R\geq1$
\begin{equation}
\sup_{t\in [0,T]} \int_{\mathbb{R}^{d}} \big(1+\vert x \vert^{R} \big)  E \big[\vert D^{\alpha} u(t,x) \vert^{r} \big]  dx < \infty,\qquad \text{for }\alpha=0,1,2. 
\label{ineq 3 lemma smooth}
\end{equation}
\end{lemma}

\begin{proof}
\emph{Step 1: Existence of a solution.} Under the assumption $b\in C_{c}^{\infty}([0,T] \times\mathbb{R}^{d})$, equation~\eqref{SDE} has a pathwise unique strong solution $X_{t}^{x}$ for every given $x\in\mathbb{R}^{d}$. As proved in~\cite{KUNITA84}, the random field $X_{t}^{x}$ has
a modification $\Phi_{t}(x)$ which is a stochastic flow of diffeomorphisms of class $C^{\infty}$ (since~$b$ is infinitely differentiable with bounded derivatives). Moreover, in view of \cite[Theorem~6.1.9]{KUNITA90} we know that, given $u_{0}\in C_{c}^{\infty}(\mathbb{R}^{d})$, the process
\begin{equation}
u(t,x)  \coloneqq u_{0}\big(\Phi_{t}^{-1}(x) \big) e^{\int_{0}^{t}c(s,\Phi_{s}^{-1}(x))
ds}\label{explicit formula}
\end{equation}
(which has paths of class $C([0,T];C^{\infty}_c(\mathbb{R}^{d}))$ by the properties of $\Phi_{t}^{-1}$) is an adapted strong solution to~\eqref{SPDE 1}. Inequality~\eqref{ineq 2 lemma smooth} then follows from~\eqref{explicit formula}.

\emph{Step 2: Regularity of the solution.} For the flow $\Phi_{t}(x)$ we have the simple inequality 
\begin{equation*}
\vert \Phi_{t}(x) \vert \leq \vert x \vert +T \Vert b \Vert_{\infty}+ \vert \sigma \vert \vert
W_{t} \vert
\end{equation*}
and thus, for every $R>0$ there exists a constant $C_{R}>0$ such that
\begin{equation}
E\big[ \vert \Phi_{t}(x) \vert^{R} \big] \leq
C_{R} \big( \vert x \vert^{R}+ T^{R} \Vert b\Vert_{\infty}^{R}+ \vert \sigma \vert^{R} T^{R/2}\big) .
\label{bound on the flow}
\end{equation}
This bound will be used below. For the derivative of the flow with respect to the initial condition in the direction $h$, $D_{h}\Phi_{t}(x) =\lim_{\eps\rightarrow0} \eps^{-1} ( \Phi_{t}(x+\eps h)-\Phi_{t}(x) )$, one has
\begin{equation*}
\frac{d}{dt}D_{h}\Phi_{t}(x)  =Db(t,\Phi_{t}(x))  D_{h}\Phi_{t}(x)  , \qquad D_{h}\Phi_{0}(x)  =h
\end{equation*}
and thus, since~$Db$ is bounded,
\begin{equation}
\vert D_{h}\Phi_{t}(x) \vert \leq C_{1} \vert h \vert \qquad \text{for } t\in [0,T],  \label{ineq smooth uniform} 
\end{equation}
where $C_1 \geq 1$ is a deterministic constant. The same is true for higher derivatives and for the inverse flow. This proves inequality~\eqref{ineq 1 lemma smooth} for $\alpha>0$, while the inequality for $\alpha=0$ comes from~\eqref{ineq 2 lemma smooth}.

Concerning the claim~\eqref{ineq 3 lemma smooth}, for $\alpha=0$ and $t \in [0,T]$ we have
\begin{align*}
E\Big[  \int_{\mathbb{R}^{d}} \big(1+ \vert x \vert^{R} \big)  \vert u (t,x) \vert^{r} dx \Big] & \leq e^{r\int_{0}^{T} \Vert c(s,\cdot) \Vert_{\infty}ds} E\Big[
\int_{\mathbb{R}^{d}} \big(1+ \vert x \vert^{R} \big) \big\vert u_{0} \big( \Phi_{t}^{-1}(x) \big) \big\vert^{r} dx \Big]  \\
&  =e^{r\int_{0}^{T} \Vert c(s,\cdot) \Vert_{\infty}ds} E \Big[ \int_{\mathbb{R}^{d}} \big(1+ \vert \Phi_{t} (y) \vert^{R} \big)  \vert u_{0}(y) \vert^{r} \vert \det D\Phi_{t}(y) \vert dy \Big]  \\
&  \leq C_{2,r}\int_{\mathbb{R}^{d}} \big(1+ E\big[ \vert \Phi_{t}(y) \vert^{R}\big]  \big) \vert u_{0}(y) \vert^{r}dy
\end{align*}
by~\eqref{ineq smooth uniform}, where $C_{2,r}=C_{1}^d e^{r\int_{0}^{T} \Vert c(s,\cdot)  \Vert_{\infty} ds}$. Combined with~\eqref{bound on the flow} 
this implies~\eqref{ineq 3 lemma smooth} for $\alpha=0$ since $u_{0}$ has compact support. The proof of~\eqref{ineq 3 lemma smooth} for $\alpha=1,2$ is similar: 
we first differentiate~$u$ by using the explicit formula~\eqref{explicit formula} and get several terms, then we control them by means of boundedness of~$c$ and its derivatives, boundedness of derivatives of direct and inverse flow, and the change of variable formula used above for $\alpha=0$. The computation is lengthy but elementary. For instance, we have
\begin{align*}
D_{h}u (t,x)   &  =e^{\int_{0}^{t}c(s,\Phi_{s}^{-1}(x) ) ds} Du_{0}\big(\Phi_{t}^{-1}(x)  \big)D_{h}\Phi_{t}^{-1}(x)  \\
&  \quad{} + u_{0}\big(\Phi_{t}^{-1}(x) \big)  e^{\int_{0}^{t}c(s,\Phi_{s}^{-1}(x))ds} \int_{0}^{t} Dc\big(s,\Phi_{s}^{-1}(x) \big)  D_{h}\Phi_{s}^{-1}(x)  ds , \\
\vert D_{h} u(t,x)\vert^{r} & \leq C_{r}C_{2,r} C_{1}^{r} \big\vert Du_{0}\big(\Phi_{t}^{-1}(x)  \big) \big\vert^{r} \vert h\vert^{r} + C_{r}C_{2,r} T^{r} \Vert Dc \Vert_{\infty}^{r} C_{1}^{r} \big\vert u_{0}\big(\Phi_{t}^{-1}(x) \big) \big\vert^{r} \vert h\vert^{r}.
\end{align*}
Hence, we obtain
\begin{multline*}
E\Big[ \int_{\mathbb{R}^{d}} \big(1+\vert x\vert^{R} \big) \vert D_{h}u(t,x)\vert^{r} dx\Big] \\
\leq C_{r} C_{2,r}C_{1}^{r+d} \vert h\vert^{r} \int_{\mathbb{R}^{d}} \big(1+E\big[ \vert \Phi_{t}(y)\vert^{R} \big] \big) \vert Du_{0}(y) \vert^{r} dy\\
+C_{r}C_{2,r}T^{r} \Vert Dc\Vert_{\infty}^{2} C_{1}^{r+d} \vert h\vert^{r} \int_{\mathbb{R}^{d}} \big(1+E\big[ \vert \Phi_{t}(y) \vert^{R}\big] \big) \vert
u_{0}(y) \vert^{r}dy
\end{multline*}
which implies~\eqref{ineq 3 lemma smooth} for $\alpha=1$. The proof is complete.
\end{proof}

\subsection{Main result on a priori estimates}

In the sequel we take the regular solution given by Lemma \ref{lemma preliminare sul caso smooth} and prove a~priori estimates. For the formulation of the result, let us introduce a $C^{1}$-function $\chi \colon \mathbb{R}^{d}\rightarrow [0,\infty)$ such that
\begin{equation}
\vert \nabla \chi (x) \vert \leq C_{\chi} \frac{\chi(x)}{1+\vert x\vert} \qquad \text{for all } x \in \mathbb{R}^{d}
\label{property of phi}
\end{equation}
for some constant $C_{\chi}>0$. For example, we might take $\chi(x)  =(1+\vert x\vert^{2})^{s/2}$ which satisfies $\vert \nabla\chi(x)\vert \leq 2 \vert s \vert \chi(x)/(1+\vert x \vert)$, for every $s\in\mathbb{R}$ (all cases $s<0$, $s=0$, and $s>0$ are of interest). The associated norm $\Vert u_{0}\Vert_{W_{\chi}^{1,r}(\mathbb{R}^{d})}$ is defined by
\begin{equation*}
\Vert u_{0} \Vert_{W_{\chi}^{1,r} (\mathbb{R}^{d})}^{r} = \sum_{i=0}^{d} \int_{\mathbb{R}^{d}} \vert \partial_{i}u_{0}(x) \vert^{r} \chi(x)  dx
\end{equation*}
where we have used the notation $\partial_{0}f=f$.

\begin{theorem}
\label{theorem a priori estimate} 
Let $p,q$ be in $(2,\infty)$ satisfying $\frac{2}{q}+\frac{d}{p}\leq1$ or $(p,q)=(\infty,2)$, let $m$ be a positive integer, let $\sigma \neq 0$, and let $\chi$ be a function satisfying~\eqref{property of phi}. Assume that~$b$ and~$c$ are a vector field and a scalar field, respectively, such that $b=b^{(1)}+b^{(2)}$, $c=c^{(1)}+c^{(2)}$, with $b^{(i)}$, $c^{(i)}$ in $C_{c}^{\infty}([0,T]\times \mathbb{R}^{d})$ for $i=1,2$. Then there exists a constant $C$ such that, for every $u_{0}$ in $C_{c}^{\infty}(\mathbb{R}^{d})$, the smooth solution~$u$ of equation~\eqref{SPDE 1} starting from $u_{0}$, given by Lemma~\ref{lemma preliminare sul caso smooth}, verifies
\begin{equation*}
\sup_{t\in [0,T]}\sum_{i=0}^{d} \int_{\mathbb{R}^{d}} E\big[ (\partial_{i}u(t,x))^{m} \big]^{2} \chi(x) dx\leq C \Vert u_{0} \Vert_{W_{\chi}^{1,2m}(\mathbb{R}^{d}) }^{2m}.
\end{equation*}
Moreover, the constant $C$ can be chosen to have continuous dependence on $m,d,\sigma,\chi,p,q$ and on the $L^{q}([0,T];L^{p}(\mathbb{R}^{d}))$ norms of~$b^{(1)}$ and~$c^{(1)}$, on the $L^{1}([0,T];C_{\lin}^{1}(\mathbb{R}^{d}))$ norm of~$b^{(2)}$, and on the $L^{1}([0,T];C_{b}^{1}(\mathbb{R}^{d}))$ norm of~$c^{(2)}$.

The result holds also for $(p,q)=(d,\infty)$ with the additional hypothesis that the $L^{\infty}([0,T];L^{d}(\mathbb{R}^{d}))$ norms of~$b^{(1)}$ and~$c^{(1)}$ are smaller than~$\delta$, see Condition \ref{LPSreg}, \textup{1c)} \textup{(}in this case the continuous dependence of $C$ on these norms is up to $\delta$\textup{)}.
\end{theorem}

\begin{corollary}
\label{Corollary a priori estimate}With the same notations of the previous
theorem, if $m$ is an even integer, then for every $s\in\mathbb{R}$ there exists a
constant $C$ depending also on $s$ \textup{(}in addition to the dependencies from the theorem\textup{)}
such that
\begin{equation*}
\sup_{t\in [0,T]}E \Big[ \Vert u(t,\cdot)\Vert_{W_{(1+\vert \cdot \vert)^{s}}^{1,m}(
\mathbb{R}^{d})}^{m} \Big] \leq C \Vert u_{0} \Vert_{W_{(1+\vert \cdot \vert)^{2s+d+1}}^{1,2m}(\mathbb{R}^{d})}^{m}.
\end{equation*}
\end{corollary}

\begin{proof}
Via H\"older's inequality we have as a consequence of Theorem~\ref{theorem a priori estimate}
\begin{align*}
\lefteqn{ \int_{\mathbb{R}^{d}} \big(1+\vert x \vert \big)^{s} E\big[ \vert \partial_{i} u(t,x) \vert^{m}\big] dx } \\
& =\int_{\mathbb{R}^{d}} \big(1+\vert x \vert \big)^{-\frac{d+1}{2}} \big(1+\vert x\vert \big)^{s+\frac{d+1}{2}} E\big[ \vert \partial_{i}u(t,x) \vert^{m}\big]  dx \\
&  \leq \Big( \int_{\mathbb{R}^{d}} \big(1+\vert x \vert \big)^{-d-1}dx \Big)^{1/2} \Big( \int_{\mathbb{R}^{d}} \big(1+ \vert x \vert \big)^{2s+ d+1} E\big[ \vert \partial_{i}u(t,x) \vert^{m} \big]^{2} dx \Big)^{1/2} \\
&  \leq C \Vert u_{0} \Vert_{W_{(1+\vert \cdot \vert )^{2s+d+1}}^{1,2m}(\mathbb{R}^{d})}^{m}
\end{align*}
for a suitable constant $C>0$.
\end{proof}

\begin{remark}
\label{rem_weihted_spaces_reflexive}
Such power-type weights play a crucial role for later applications. Therefore, let us note that, for every $s \in \mathbb{R}$ and $m \in (1,\infty)$, $W^{1,m}_{(1+|\cdot|)^s}(\mathbb{R}^d)$ is a reflexive Banach space. We can show this, for instance, by observing that the dual of $L^m_{(1+|\cdot|)^s}(\mathbb{R}^d)$ is isomorphic to $L^{m'}_{(1+|\cdot|)^{sm'/m}}(\mathbb{R}^d)$ with $1/m + 1/m' = 1$. Hence, the $L^m$ spaces with these weights are reflexive, which directly carries over to the weighted Sobolev spaces since they are closed subspaces via the mapping $f\mapsto(f,\partial_1f,\ldots,\partial_df)$. The same holds for spaces like $L^m([0,T]\times\Omega;W^{1,m}_{(1+|\cdot|)^s}(\mathbb{R}^d))$ and $L^m([0,T]\times\Omega;L^m_{(1+|\cdot|)^s}(\mathbb{R}^d))$. In particular, the Banach--Alaoglu theorem is at our disposal.
\end{remark}

The next subsections are devoted to the proof of the a priori estimate of the theorem. At the end, they will be used to construct a (weaker) solution
corresponding to non-smooth data. Thus, in the sequel,~$u$ refers to a smooth solution, with smooth and compactly supported data $b,u_{0}$.

\subsection{Formal computation}

This section serves as a formal explanation of the first main steps of the proof, those based on renormalization, passage from It\^{o} to Stratonovich
formulation and taking the expectation. A precise statement and proof is given in the next Section~\ref{section_rigorous_w}. 

The aim of the following computations is to write, given any positive integer~$m$, a closed system of parabolic equations for the quantities $E[ \prod_{i\in I}\partial_{i}u]$, where $I$ varies in the finite multi-indices with elements in $\{0,1,\ldots, d\}$ of length at most~$m$. In principle, we need only the quantities $E[(\partial_{i}u)^{m}]$ for $i=1,\ldots,d$, but they do not form a closed system.

Equation~\eqref{SPDE 1} is formally of the form
\begin{equation*}
\mathcal{L}u+cu=0
\end{equation*}
where $\mathcal{L}$ is the differential operator
\begin{equation*}
\mathcal{L}f=\partial_{t}f+b\cdot\nabla f+\sigma\nabla f\circ\dot{W}.
\end{equation*}
Being a first order differential operator, it formally satisfies the Leibniz rule
\begin{equation}
\mathcal{L} \Big(\prod_{j=1}^{m}f_{j}\Big)  =\sum_{i=1}^{m}\prod_{j\neq
i}f_{j}\mathcal{L}f_{i}. \label{renormalization} 
\end{equation}
This is the step that we call renormalization, following~\cite{DIPELI89}: in the language of that paper, if $\beta \colon \mathbb{R\rightarrow R}$ is a $C^{1}$-function and if~$v$ is a solution of $\mathcal{L}v=0$, then formally $\mathcal{L}\beta(v)=0$, and solutions which satisfy this rule rigorously are called renormalized solutions. Property~\eqref{renormalization} is a variant of this idea. We apply the renormalization to first derivatives of~$u$. Precisely, if~$u$ is a solution of $\mathcal{L}u+cu=0$, we set
\begin{equation*}
v_{i}\coloneqq \partial_{i}u,\qquad \text{for } i=1,\ldots,d.
\end{equation*}
One has $\partial_{i}(\mathcal{L}u+cu)  =0$ and thus
\begin{equation*}
\mathcal{L}v_{i}=-(\partial_{i}b\cdot\nabla u+u\partial_{i} 
c+cv_{i})  ,\qquad i=1,\ldots,d.
\end{equation*}
With the notation $v_{0}=u$ we also have 
\begin{equation*}
\mathcal{L}v_{0}=-cv_{0}.
\end{equation*}

In the sequel, $I$ will be a finite multi-index with elements in $\{0,1,\ldots,d\}$, namely an element of $\cup_{m\in\mathbb{N}}\{0,1,\ldots, d\}^{m}$. If $I\in\{0,1,\ldots, d\}^{m}$ we set $\left\vert I\right\vert =m$. Given a function $h \colon \{0,1,\ldots, d\}\rightarrow\mathbb{R}$, $\sum_{i\in I}h(i)$ means the sum over all the components of $I$ (counting repetitions), and similarly for the product. The multi-index $I\setminus i$ means that we drop in $I$ a component of value~$i$; the multi-index $I\setminus i\cup k$ means that we substitute in $I$ a component of value~$i$ with a component of value~$k$; which component~$i$ is dropped or replaced does not matter because we consider only expressions of the form $\sum_{i\in I}h(i)$ and similar ones. Let us set
\begin{equation*}
v_{I}\coloneqq \prod_{i\in I}v_{i} \quad \text{which satisfies} \quad \mathcal{L}v_{I}=\sum_{i\in I}v_{I\setminus i}\mathcal{L}v_{i}
\end{equation*}
in view of the Leibniz rule~\eqref{renormalization}. Now, the equations for $v_{i}$ differ depending on whether $i=0$ or $i \in \{1,\ldots,d\}$. The
term $cv_{i}$ appears in all of them, but not $\partial_{i}b\cdot\nabla
u+u\partial_{i}c$. Hence, we find 
\begin{align*}
\mathcal{L}v_{I}  &  =-c\left\vert I\right\vert v_{I}-\sum_{i\in I,i\neq
0}v_{I\setminus i}(\partial_{i}b\cdot\nabla u+u\partial_{i}c) \\
&  =-c\left\vert I\right\vert v_{I}-\sum_{i\in I,i\neq0}\sum_{k=1} 
^{d}v_{I\setminus i\cup k}\partial_{i}b_{k}-\sum_{i\in I,i\neq0}v_{I\setminus
i\cup0}\partial_{i}c.
\end{align*}
Next we want to take the expected value. The problem is the Stratonovich term $\sigma\nabla
v_{I}\circ\dot{W}$ in $\mathcal{L}v_{I}$. Rewriting it as an  It\^{o} term with correction, we get 
\begin{equation*}
\sigma\nabla v_{I}\circ dW=\sigma\nabla v_{I}dW+\frac{1}{2}\sum_{k=1} 
^{d}d\big[  \sigma\partial_{k}v_{I},W^{k}\big]
\end{equation*}
where the brackets $[\cdot,\cdot]$ denote the joint quadratic variation. Since $dv_{I}$ has as local martingale term $-\sigma\sum_{k^{\prime}=1}^{d}\partial_{k^{\prime}} v_{I}dW^{k^{\prime}}$, we have $d[\sigma\partial_{k}v_{I},W^{k}]=-\sigma^{2}\partial_{k}^{2}v_{I}dt$, and thus, we find
\begin{equation*}
\sigma\nabla v_{I}\circ dW=\sigma\nabla v_{I}dW-\frac{\sigma^{2}}{2}\Delta
v_{I}dt.
\end{equation*}
Taking (formally) the expectation in the equation for $v_{I}$, we arrive at 
\begin{equation}
\partial_{t}w_{I}+b\cdot\nabla w_{I}+c\left\vert I\right\vert w_{I}+\sum_{i\in I,i\neq0}\sum_{k=1}^{d}w_{I\setminus i\cup k}\partial_{i}b_{k}+\sum_{i\in
I,i\neq0}w_{I\setminus i\cup0}\partial_{i}c=\frac{\sigma^{2}}{2}\Delta w_{I}\label{PDE for w} 
\end{equation}
where 
\begin{equation*}
w_{I}\coloneqq E[  v_{I}]  .
\end{equation*}
This is the first half of the proof of Theorem~\ref{theorem a priori estimate}, which will be carried out rigorously in Section~\ref{section_rigorous_w}. The second half is the estimate on~$w$ coming from the parabolic nature of this equation, which will be established in Section~\ref{section_parabolic_estimates}.

\subsection{Rigorous proof of~\eqref{PDE for w}}
\label{section_rigorous_w}

We work with the regular solution~$u$ given by Lemma~\ref{lemma preliminare sul caso smooth} and we use the notations $I$, $I\setminus i\cup k$, $v_{i}$, $v_{I}$, $w_{I}$ as introduced in the previous section. We observe that, since~$u$ is smooth in~$x$, the~$v_{i}$'s and their spatial derivatives are well-defined. Moreover, due to inequality~\eqref{ineq 1 lemma smooth}, also the expected values $w_{I}$'s are well-defined and smooth in~$x$.

\begin{lemma}
The function $w_{I}(t,x)$ is continuously differentiable in
time and satisfies the \textup{(}pointwise\textup{)} equation~\eqref{PDE for w}.
\end{lemma}

\begin{proof}
\textbf{\ }The solution provided by Lemma
\ref{lemma preliminare sul caso smooth} is a pointwise regular solution to~\eqref{SPDE 1}, namely it satisfies with probability one the
identity
\begin{equation*}
u(t,x)  +\int_{0}^{t} \big(b\cdot\nabla u+cu\big)  (
s,x)  ds+\int_{0}^{t}\sigma\nabla u(s,x)  \circ
dW_{s}=u_{0}(x)
\end{equation*}
for every $(t,x) \in [0,T]  \times\mathbb{R}^{d}$.
Since $\nabla u(t,x)$ is a semimartingale (from the definition of~$u$), the Stratonovich integral is well-defined. Using \cite[Theorem 7.10]{KUNITA84} of differentiation under stochastic integrals one deduces
\begin{equation*}
\partial_{i}u(t,x)  +\int_{0}^{t} \big(\partial_{i}b\cdot\nabla
u+b\cdot\nabla\partial_{i}u+u\partial_{i}c+c\partial_{i}u \big)  (
s,x)  ds+\int_{0}^{t}\sigma\nabla\partial_{i}u(s,x)  \circ
dW_{s}=\partial_{i}u_{0}(x),
\end{equation*}
hence, for $i=1,\ldots,d$, we have
\begin{equation*}
v_{i}(t,x)  +\int_{0}^{t}\Big(\sum_{k=1}^{d}\partial_{i}
b_{k}v_{k}+b\cdot\nabla v_{i}+u\partial_{i}c+cv_{i}\Big)  (
s,x)  ds+\int_{0}^{t}\sigma\nabla v_{i}(s,x)  \circ
dW_{s}=v_{i}(0,x),
\end{equation*}
while for $i=0$, we obtain just from the solution property
\begin{equation*}
v_{0}(t,x)  +\int_{0}^{t} \big(b\cdot\nabla v_{0}
+cv_{0}\big)  (s,x)  ds+\int_{0}^{t}\sigma\nabla v_{0}(
s,x)  \circ dW_{s}=v_{0}(0,x)  .
\end{equation*}
Setting
\begin{equation*}
r_{0} =0 \quad \text{and} \quad r_{i} =\sum_{k=1}^{d}\partial_{i}b_{k}v_{k}+u\partial_{i}c\text{ \ for
}i=1,\ldots,d ,
\end{equation*}
we may write for all $i=0,1,\ldots,d$
\begin{equation*}
v_{i}(t,x)  +\int_{0}^{t} \big(b\cdot\nabla v_{i}+cv_{i}
+r_{i}\big)  (s,x)  ds+\int_{0}^{t}\sigma\nabla v_{0}(
s,x)  \circ dW_{s}=v_{0}(0,x).
\end{equation*}
Now we use It\^{o} formula in Stratonovich form, see \cite[Theorem 8.3]{KUNITA84}, to get
\begin{align*}
v_{I}(t,x)   &  =v_{I}(0,x)  +\sum_{i\in I}\int 
_{0}^{t}v_{I\setminus i}(s,x)  \circ dv_{i}(s,x) \\
&  =v_{I}(0,x)  -\sum_{i\in I}\int_{0}^{t}v_{I\setminus i}(
s,x)  \big(b\cdot\nabla v_{i}+cv_{i}+r_{i}\big)  (
s,x)  ds\\
& \quad {} -\sum_{i\in I}\int_{0}^{t}v_{I\setminus i}(s,x)  \sigma\nabla
v_{i}(s,x)  \circ dW_{s} .
\end{align*}
Moreover, we have $\partial_{j}v_{I}=\sum_{i\in I}v_{I\setminus i}\partial_{j}v_{i}$, and thus, we may rewrite the previous identity as
\begin{multline}
 v_{I}(t,x)  +\int_{0}^{t} \big(b\cdot\nabla v_{I}+\left\vert
I\right\vert cv_{I}\big)  (s,x)  ds+\sum_{i\in I}\int_{0} 
^{t}(v_{I\setminus i}r_{i})  (s,x)
ds\label{identity for products}\\
 =v_{I}(0,x)  -\sigma\int_{0}^{t}\nabla v_{I}(
s,x)  \circ dW_{s} .
\end{multline}
By the definition of~$r_i$, for the last integral on the left-hand side of~\eqref{identity for products}, it holds
\begin{equation*}
\sum_{i\in I}\int_{0}^{t}(v_{I\setminus i}r_{i})  (
s,x)  ds=\sum_{i\in I,i\neq0}\sum_{k=1}^{d}\int_{0}^{t}(
v_{I\setminus i\cup k}\partial_{i}b_{k})  (s,x)
ds+\sum_{i\in I,i\neq0}\int_{0}^{t}(v_{I\setminus i\cup0}\partial
_{i}c)  (s,x)  ds .
\end{equation*}
Furthermore, before taking expectations, we want to pass in~\eqref{identity for products} from the Stratonovich to the It\^o formulation. To this end, we first note (again by \cite[Theorem 7.10]{KUNITA84}) that
\[
\partial_{j}v_{I}(t,x)  +\int_{0}^{t}g(s,x)
ds=\partial_{j}v_{I}(0,x)  -\sigma\sum_{k=1}^{d}\int_{0} 
^{t}\partial_{j}\partial_{k}v_{I}(s,x)  \circ dW_{s}^{k} 
\]
for a bounded process $g$. Hence, for the stochastic integral in~\eqref{identity for products} we find
\begin{align*}
\sigma\sum_{j=1}^{d}\int_{0}^{t}\partial_{j}v_{I}(s,x)  \circ
dW_{s}^{j}  &  =\sigma\int_{0}^{t}\nabla v_{I}(s,x)  \cdot
dW_{s}+\frac{\sigma}{2}\sum_{j=1}^{d}\big[  \partial_{j}v_{I}(
\cdot,x)  ,W^{j}\big]_{t}\\
&  =\sigma\int_{0}^{t}\nabla v_{I}(s,x)  \cdot dW_{s} 
-\frac{\sigma^{2}}{2}\sum_{j=1}^{d}\int_{0}^{t}\partial_{j}^{2}v_{I}(
s,x)  ds .
\end{align*}
We have proved so far
\begin{multline*}
v_{I}(t,x)  +\int_{0}^{t} \big(b\cdot\nabla v_{I}+\left\vert
I\right\vert cv_{I}\big)  (s,x)  ds\\
 +\sum_{i\in I,i\neq0}\sum_{k=1}^{d}\int_{0}^{t}(v_{I\setminus i\cup
k}\partial_{i}b_{k})  (s,x)  ds+ \sum_{i\in I,i\neq0} 
\int_{0}^{t}(v_{I\setminus i\cup0}\partial_{i}c)  (
s,x)  ds\\
 =v_{I}(0,x)  -\sigma\int_{0}^{t}\nabla v_{I}(
s,x)  dW_{s}+\frac{\sigma^{2}}{2}\int_{0}^{t}\Delta v_{I}(
s,x)  ds .
\end{multline*}
The process $\nabla v_{I}(s,x)$ is bounded (via Lemma
\ref{lemma preliminare sul caso smooth}), thus $\int_{0}^{t}\nabla
v_{I}(s,x)  dW_{s}$ is a martingale. All other terms have also
finite expectation, due to estimate~\eqref{ineq 1 lemma smooth} of Lemma
\ref{lemma preliminare sul caso smooth}. Hence, taking expectation, we have
\begin{multline*}
w_{I}(t,x)  +\int_{0}^{t}  \big(b\cdot\nabla w_{I}+\left\vert
I\right\vert cw_{I} \big)  (s,x)  ds\\
+\sum_{i\in I,i\neq0}\sum_{k=1}^{d}\int_{0}^{t}(w_{I\setminus i\cup
k}\partial_{i}b_{k})  (s,x)  ds+\sum_{i\in I,i\neq0} 
\int_{0}^{t}(w_{I\setminus i\cup0}\partial_{i}c)  (
s,x)  ds\\
=v_{I}(0,x)  +\frac{\sigma^{2}}{2}\int_{0}^{t}\Delta
w_{I}(s,x)  ds .
\end{multline*}
This identity implies that $w_{I}(t,x)$ is continuously differentiable in~$t$ and that equation~\eqref{PDE for w} holds. The proof of the lemma is complete.
\end{proof}

\subsection{Estimates for the parabolic deterministic equation}
\label{section_parabolic_estimates}

The system for the $w_I$'s is a parabolic deterministic linear system of partial differential equations. In this section we will obtain energy estimates for $w_I$ which will allow us to obtain the desired a~priori bounds. The fact that we have a system instead of a single equation will not affect the estimate (to have an idea of what the final parabolic estimate should be, one could think that~$w_I$ is independent of~$I$).
 
For every smooth function $\chi:\mathbb{R}^{d}\rightarrow \lbrack0,\infty)$ as in the statement of the Theorem~\ref{theorem a priori estimate} we multiply the identity~\eqref{PDE for w} by $\chi w_I$ and get
\[
\partial_{t}(\chi w_{I}^{2})  +2\chi w_{I}b\cdot\nabla
w_{I}+2\chi c\vert I\vert w_{I}^{2}+2\sum_{i\in I,i\neq0}\sum
_{k=0}^{d}\chi w_{I}w_{I\setminus i\cup k}\partial_{i}b_{k}=\sigma^{2}\chi
w_{I}\Delta w_{I},
\]
where, for a shorter notation, we have set $b_0 \coloneqq  c$. From estimate~\eqref{ineq 3 lemma smooth} of Lemma \ref{lemma preliminare sul caso smooth} we know that all terms in this identity are integrable on $\mathbb{R}^{d}$, uniformly in time. Hence, we obtain
\begin{align*}
&  \int_{\mathbb{R}^{d}}\chi w_{I}^{2}(t,x)  dx+\sigma^{2} 
\int_{0}^{t}\int_{\mathbb{R}^{d}}\chi\vert \nabla w_{I}\vert
^{2} (s,x)dxds\\
&  =-2\sigma^{2}\int_{0}^{t}\int_{\mathbb{R}^{d}}w_{I}\nabla w_{I}\cdot\nabla\chi  (s,x) dxds\\
& \quad {}  -2\int_{0}^{t}\int_{\mathbb{R}^{d}} \big(\chi w_{I}b\cdot\nabla
w_{I}+\chi c\vert I\vert w_{I}^{2} \big)  (s,x)
dxds\\
& \quad {}  -2\sum_{i\in I,i\neq0}\sum_{k=0}^{d}\int_{0}^{t}\int_{\mathbb{R}^{d}} 
\chi \big(w_{I}w_{I\setminus i\cup k}\partial_{i}b_{k}\big)  (
s,x)  dxds .
\end{align*}
The term with $\partial_{i}b_{k}$ would spoil all our efforts of proving estimates depending only on the LPS norm of the coefficients, but fortunately we may integrate by parts that term. This is the first key ingredient of this second half of the proof of Theorem~\ref{theorem a priori estimate}. The second key ingredient is the presence of the term $\sigma^{2}\int_{0}^{t}\int_{\mathbb{R}^{d}}\chi\vert \nabla w_{I}\vert ^{2}dxds$, ultimately coming from the passage Stratonovich to It\^o formulation plus taking expectation; this allows us to ask as little as possible on the drift~$b$ to close the estimates: we may keep first derivatives of $w_{I}$ on the right-hand-side of the previous identity, opposite to the deterministic case.

Before starting with the estimates, we recall that~$b=b^{(1)}+b^{(2)}$ and $c=c^{(1)}+c^{(2)}$ are assumed to be the sum of two smooth vector fields. Since the desired estimates in Theorem~\ref{theorem a priori estimate} differ for the rough part $b^{(1)}$ and the regular (but possibly with linear growth) part~$b^{(2)}$, we now split~$b$ and~$c$ and use the integration by parts formula, in the following way: when a term with $\partial_i b^{(1)}$ appears, we bring the derivative on the other terms; when we have~$b^{(2)}$ multiplied by the derivative of some object, we bring the derivative on $b^{(2)}$. In this way we obtain
\begin{equation*}
\int_{\mathbb{R}^{d}}\chi w_{I}^{2}(t,x)  dx+\sigma^{2} 
\int_{0}^{t}\int_{\mathbb{R}^{d}}\chi\vert \nabla w_{I}\vert
^{2}(s,x) dxds 
=A_{I,0}+A_{I,1}^{(1)}+A_{I,1}^{(2)} 
+A_{I,2}^{(1)}+A_{I,2}^{(2)},
\end{equation*}
where we have defined
\begin{align*}
A_{I,0} &  \coloneqq -2\sigma^{2}\int_{0}^{t}\int_{\mathbb{R}^{d}} \big( w_{I}\nabla
w_{I}\cdot\nabla\chi \big) (s,x) dxds\\
A_{I,1}^{(1)} & \coloneqq -2\int_{0}^{t}\int_{\mathbb{R}^{d}} \big(
\chi w_{I}b^{(1)}\cdot\nabla w_{I}+\chi c^{(1)
}\vert I\vert w_{I}^{2}\big)  (s,x)  dxds\\
A_{I,1}^{(2)} & \coloneqq \int_{0}^{t}\int_{\mathbb{R}^{d}} \big(
\chi \diverg  b^{(2)} w_{I}^2 + \nabla \chi \cdot b^{(2)} w_{I}^{2} - 2 \chi c^{(2)
}\vert I\vert w_{I}^{2} \big)  (s,x)  dxds\\
A_{I,2}^{(1)} & \coloneqq  2\sum_{i\in I,i\neq0}\sum_{k=0}^{d}\int
_{0}^{t}\int_{\mathbb{R}^{d}} \big( \big( \partial_{i} \chi w_{I}w_{I\setminus i\cup k} + \chi \partial_{i} w_{I} w_{I\setminus i\cup k} + \chi w_{I} \partial_{i} w_{I\setminus i\cup k} \big) b_{k}^{(1)}\big) (s,x) dxds \\
A_{I,2}^{(2)} & \coloneqq -2\sum_{i\in I,i\neq0}\sum_{k=0}^{d}\int_{0}^{t} \int_{\mathbb{R}^{d}} \big(\chi w_{I}w_{I\setminus i\cup k} \partial_{i}b_{k}^{(2)}\big) (s,x) dxds .
\end{align*}
To estimate these terms we essentially use the following consequence of H\"older's inequality
\begin{align}
\label{fgh-est}
\int_{0}^{t}\int_{\mathbb{R}^{d}} \big(f g h \big) (s,x) dxds & \leq \int_{0}^{t} \Vert f \Vert _{\infty}(s)  \int_{\mathbb{R}^{d}} \big( \vert g \vert^2  + \vert h \vert^2 \big)(s,x) dx ds
\end{align}
for functions $f,g,h$ defined over $[0,T] \times \mathbb{R}^d$ such that the relevant integrals are well-defined. Moreover, we shall use at several instances the estimate~\eqref{property of phi} on $\vert \nabla \chi \vert$, and we further drop the notation $(s,x)$ inside the integrals. For the first term we now employ inequality~\eqref{fgh-est} with $f \equiv 1$ (the special case of H\"older's inequality), $g = \sqrt{\eps \chi} \vert \nabla w_{I}\vert$ and $h = 2\sigma^{2} \sqrt{\eps^{-1} \chi} \vert w_{I}\vert$ for an arbitrary positive number $\eps > 0$ to find 
\begin{align*}
A_{I,0} &  \leq\eps\int_{0}^{t}\int_{\mathbb{R}^{d}}\chi\vert \nabla
w_{I}\vert ^{2}dxds+C_{\eps,\sigma,\chi}\int_{0}^{t}\int 
_{\mathbb{R}^{d}}\chi\vert w_{I}\vert ^{2}dxds .
\end{align*}
Similarly for the second term, we have
\begin{align*}
A_{I,1}^{(1)  } 
&  \leq\eps\int_{0}^{t}\int_{\mathbb{R}^{d}}\chi\vert \nabla
w_{I}\vert ^{2}dxds+C_{\eps}\int_{0}^{t}\int_{\mathbb{R}^{d}}(
\vert b^{(1)  }\vert ^{2}+\vert c^{(
1)  }\vert \vert I\vert )  \chi w_{I}^{2}dxds .
\end{align*}
Next, with $g = h = \sqrt{\chi} w_I$ and $f$ chosen as $\diverg  b^{(2)  }$, $b^{(2)} C_\chi (1+ \vert x \vert)^{-1}$ and $2\vert c^{(2)  }\vert \vert I\vert$, respectively, we obtain via~\eqref{property of phi} the estimate 
\begin{align*}
A_{I,1}^{(2)} &  \leq \int_{0}^{t}\int_{\mathbb{R}^{d}} \big(  \big(\diverg  
b^{(2)}- 2c^{(2)  }\vert I\vert \big) \chi w_{I}^{2}\big) dxds +C_{\chi}\int_{0}^{t}\int_{\mathbb{R}^{d}}w_{I}^{2}\frac{\vert
b^{(2)}\vert }{1+\vert x\vert }\chi dxds\\
&  \leq\int_{0}^{t} \Big(\Vert \diverg b^{(2)}(s) \Vert _{\infty}+C_{\chi} \Big\Vert \frac{b^{(2)}(s,\cdot)}{1+\vert \cdot\vert } \Big\Vert_{\infty}
+ 2\vert I\vert \Vert c^{(2)}(s)  \Vert_{\infty}\Big) 
\int_{\mathbb{R}^{d}}\chi w_{I}^{2}dxds .
\end{align*}
Similarly as for the terms $A_{I,1}^{(2)}$ and $A_{I,1}^{(2)}$, we now proceed for the remaining terms $A_{I,2}^{(2)}$ and $A_{I,2}^{(2)}$, with the main difference that $w_I$ eventually needs to be replaced with $w_{I \setminus i \cup k}$. In this way, we get
\begin{align*}
A_{I,2}^{(1)} &  \leq 2 \sum_{i\in I,i\neq0} \sum_{k=0}^{d}\int_{0}^{t} \int_{\mathbb{R}^{d}} \big(\chi\vert w_{I\setminus i\cup k}\vert \vert \partial_{i}w_{I}\vert +\chi\vert w_{I}\vert \vert \partial_{i}w_{I\setminus i\cup k}\vert +C_{\chi}\vert w_{I\setminus i\cup k}\vert \vert w_{I}\vert \chi\big) \vert b_{k}^{(1)  }\vert dxds\\
&  \leq \eps \sum_{i\in I,i\neq0} \sum_{k=0}^{d}\int_{0}^{t} \int_{\mathbb{R}^{d}} \big(\chi\vert \partial_{i}w_{I}\vert^{2} + \chi\vert \partial_{i}w_{I\setminus i\cup k}\vert ^{2}\big) dxds\\
& \quad {}+ C_{\eps,\chi}\sum_{i\in I,i\neq0}\sum_{k=0}^{d}\int_{0}^{t} 
\int_{\mathbb{R}^{d}} \big(\chi w_{I\setminus i\cup k}^{2}+\chi w_{I} 
^{2}\big)  \vert b_{k}^{(1)  }\vert^{2} dxds
\end{align*}
and finally
\begin{equation*}
A_{I,2}^{(2)} \leq \sum_{i\in I,i\neq0} \sum_{k=0}^{d} \int_{0}^{t} \Vert \nabla b_k^{(2)}(s)  \Vert_{\infty} \int_{\mathbb{R}^{d}} \chi \big(w_{I}^{2} + w_{I\setminus i\cup k}^{2} \big)  dxds .
\end{equation*} 
Given $m\in\mathbb{N}$, we abbreviate 
\begin{equation*}
\theta_{m}= \Big(\sum_{\vert I\vert =m} \chi w_{I}^{2} \Big)^{1/2} \quad \text{and} \quad  \rho_{m}= \Big(\sum_{\vert I\vert =m} \sum_{i=1}^{d}\chi\vert \partial_{i}w_{I}\vert^{2} \Big)^{1/2} .
\end{equation*}
Collecting the previous estimates and summing over $|I|=m$, we have proved so far
\begin{align*}
\lefteqn{ \int_{\mathbb{R}^{d}}\theta_{m}^{2}(t,x)dx+\sigma^{2}\int_{0}^{t} 
\int_{\mathbb{R}^{d}}\rho_{m}^{2}dxdr\leq\sum_{\vert I\vert =m} 
\big( A_{I,0} + A_{I,1}^{(1)}+A_{I,1}^{(2)} + A_{I,2}^{(1)}+A_{I,2}^{(2)} \big)} \\
&  \leq \eps\int_{0}^{t}\int_{\mathbb{R}^{d}}\rho_{m}^{2}dxds+C_{\eps} \int_{0}^{t} \int_{\mathbb{R}^{d}} \big( \vert b^{(1)} \vert^{2} + m\vert c^{(1)}\vert \big)  \theta_{m}^{2}dxds\\
&  \quad {} + \int_{0}^{t} \Big(\Vert \diverg b^{(2)}(s) \Vert_{\infty}+ C_{\chi} \Big\Vert \frac{b^{(2)}(s,\cdot)} {1+\vert \cdot\vert} \Big\Vert_{\infty} + 2m \Vert c^{(2)}(s)  \Vert_{\infty} \Big) \int_{\mathbb{R}^{d}} \theta_{m}^{2} dxds\\
& \quad {}+2 \eps C_{m,d}\int_{0}^{t} \int_{\mathbb{R}^{d}} \rho_{m}^{2} dxds + 2C_{\eps,\chi} C_{m,d}\int_{0}^{t} \int_{\mathbb{R}^{d}} \theta_{m}^{2} \big(\vert b^{(1)} \vert^{2}+(c^{(1)})^{2} \big)  dxds\\
& \quad {}+ 2C_{m,d}\int_{0}^{t} \big( \Vert \nabla b^{(2)}(s) \Vert_{\infty} + \Vert \nabla c^{(2)}(s) \Vert_{\infty} \big) \int_{\mathbb{R}^{d}}\theta_{m}^{2} dxds ,
\end{align*}
where we have repeatedly employed the identities 
\begin{equation*}
\sum_{|I|=m}\sum_{i\in I}\sum^d_{k=0}\chi|w_{I\setminus i\cup k}|^2= m(d+1)\theta_m
\end{equation*}
and 
\begin{equation*}
\sum_{|I|=m}\sum_{i\in I}\sum^d_{k=0}\chi|\partial_i w_{I\setminus i\cup k}|^2= m(d+1)\rho_m
\end{equation*}
(since every $J$ of length $m$ is counted $m(d+1)$ times in the previous sum); so here we have $C_{m,d}=m(d+1)$. We can then continue to estimate (using H\"older inequality for~$m|c^{(1)}|$) and find
\begin{align*}
\lefteqn{ \int_{\mathbb{R}^{d}}\theta_{m}^{2}(t,x)dx+\sigma^{2}\int_{0}^{t} 
\int_{\mathbb{R}^{d}}\rho_{m}^{2}dxdr} \\
&  \leq \eps(1+4C_{m,d})  \int_{0}^{t}\int_{\mathbb{R}^{d} 
}\rho_{m}^{2}dxds\\
& \quad {} +C_{\eps,m,d,\chi}\int_{0}^{t}\int_{\mathbb{R}^{d}} \big(\vert b^{(1)  }\vert ^{2}+(c^{(1)  })
^{2} +1 \big)  \theta_{m}^{2}dxds\\
& \quad {} +C_{m,d} \int_{0}^{t} \Big( \Vert \nabla b^{(2)}(s) \Vert_{\infty} + C_{\chi} \Big\Vert \frac{b^{(2)}(s,\cdot)}{1+\vert \cdot\vert} \Big\Vert_{\infty} + \Vert c^{(2)}(s)\Vert_{\infty} + \Vert\nabla c^{(2)}(s) \Vert_{\infty} \Big) \int_{\mathbb{R}^{d}}\theta_{m}^{2}dxds
\end{align*}
for new positive constants $C_{\eps,m,d,\chi},C_{m,d}$ (which incorporate the $m$ inside the integrals). We choose~$\eps$ so small
that $\eps(1+4C_{m,d})  \leq\sigma^{2}$ and rename $C_{\eps,m,d,\chi}$ by $C_{m,d,\sigma}$. Therefore, we end up with the preliminary estimate
\begin{align}
\label{quasi final inequality}
\lefteqn{ \int_{\mathbb{R}^{d}}\theta_{m}^{2}(t,x)dx+\frac{\sigma^{2}}{2}\int_{0} 
^{t}\int_{\mathbb{R}^{d}}\rho_{m}^{2}dxdr } \\
&  \leq C_{m,d,\sigma,\chi} \int_{0}^{t} \int_{\mathbb{R}^{d}} \big(\vert b^{(1)}\vert ^{2}+(c^{(1)})^{2} +1\big) \theta_{m}^{2} dxds \nonumber \\& \quad {} +C_{m,d} \int_{0}^{t} \Big( \Vert \nabla b^{(2)}(s)  \Vert _{\infty}+C_{\chi} \Big\Vert \frac{b^{(2)}(s,\cdot)}{1+\vert \cdot\vert } \Big\Vert_{\infty}+ \Vert c^{(2)}(s)  \Vert _{\infty}+\Vert \nabla c^{(2)}(s) \Vert_{\infty} \Big) \int_{\mathbb{R}^{d}} \theta_{m}^{2} dxds . \nonumber
\end{align}

\subsection{End of the proof of Theorem~\ref{theorem a priori estimate}}

Starting from the previous inequality~\eqref{quasi final inequality}, we can now continue to estimate its right-hand side by taking into account the LPS-condition on~$b$ and~$c$. To this end, we need to distinguish the three cases $(p,q)=(\infty,2)$, $(p,q) \in (2,\infty)$ and $(p,q) =(d,\infty)$. The main difficulty will be to estimate the term
\begin{equation*}
\int_{0}^{t} \int_{\mathbb{R}^{d}} \big(\vert b^{(1)}\vert ^{2}+(c^{(1)})^{2} \big) \theta_{m}^{2} dxds .
\end{equation*}
From the resulting inequality we can then conclude the proof of Theorem~\ref{theorem a priori estimate} via the Gronwall lemma. For the sake of simplicity, let us first restrict ourselves to the important particular case where $b^{(1)}$ and $c^{(1)}$ can be estimated in the $L^{\infty}$-topology.

\begin{proof}[Proof of Theorem~\ref{theorem a priori estimate} in the case
$(p,q)=(\infty,2)$] Here we have
\begin{align*}
\lefteqn{ \int_{\mathbb{R}^{d}}\theta_{m}^{2}(t,x)dx+\frac{\sigma^{2}}{2}\int_{0} 
^{t}\int_{\mathbb{R}^{d}}\rho_{m}^{2}dxdr }\\
&  \leq C_{m,d,\sigma,\chi}\int_{0}^{t} \big(\Vert b^{(1)}(s) \Vert_{\infty}^{2}+\Vert c^{(1)}(s)  \Vert_{\infty}^{2} +1 \big)  \int_{\mathbb{R}^{d}}\theta_{m}^{2}dxds\\
& \quad {} +C_{m,d}\int_{0}^{t} \Big( \Vert \nabla b^{(2)}(s) \Vert_{\infty} +C_{\chi} \Big\Vert \frac{b^{(2)}(s,\cdot)}{1+\vert \cdot\vert }\Big\Vert
_{\infty}+ \Vert c^{(2)}(s)  \Vert_{\infty}+ \Vert \nabla c^{(2)}(s) \Vert_{\infty} \Big)  \int_{\mathbb{R}^{d}} \theta_{m}^{2}dxds .
\end{align*}
Thus, we find via Gronwall's lemma a constant $C_{0} 
=C_{0}(m,d,\sigma,b^{(1)  },b^{(2)
},c^{(1)  },c^{(2)  })$, which depends on
$b^{(1)  },b^{(2)  },c^{(1)
},c^{(2)  }$ through the norms 
\begin{align*}
&  \int_{0}^{T} \Vert b^{(1)}(s) \Vert_{\infty}^{2}ds ,\qquad \int_{0}^{T} \Big( \Big\Vert \frac{b^{(2)}(s,\cdot) }{1+\vert \cdot\vert
}\Big\Vert_{\infty}+\Vert \nabla b^{(2)}(s) \Vert_{\infty} \Big)  ds,\\
&  \int_{0}^{T} \Vert c^{(1)}(s)  \Vert_{\infty}^{2} ds , \qquad \int_{0}^{T}\big(\Vert c^{(2)}(s)  \Vert _{\infty}+ \Vert \nabla c^{(2)}(s) \Vert_{\infty} \big)  ds
\end{align*}
such that
\begin{equation}
\sup_{t\in [  0,T]  }\int_{\mathbb{R}^{d}}\theta_{m}^{2}(t,x)dx\leq
C_{0}\int_{\mathbb{R}^{d}}\theta_{m}^{2}(0,x)dx.\label{final bound on theta} 
\end{equation}
We then notice that by the definition of~$\theta_m$ and by Young's inequality, there holds
\begin{equation*}
\sum_{i=0}^{d}\int_{\mathbb{R}^{d}}E[(\partial_{i}u(t,x))
^{m}]^{2}\chi(x)  dx\leq\int_{\mathbb{R}^{d}}\theta
_{m}^{2}(t,x)dx\leq C_{m,d}\sum_{i=0}^{d}\int_{\mathbb{R}^{d}}E[
\vert \partial_{i}u(t,x)\vert ^{m}]^{2}\chi(x)  dx
\end{equation*}
for some constant $C_{m,d}>0$, hence
\begin{equation*}
\sup_{t\in[0,T]} \sum_{i=1}^{d}\int_{\mathbb{R}^{d}}E[(\partial_{i}u(t,x))^{m}]^{2} \chi(x)
dx\leq C_{0}\int_{\mathbb{R}^{d}}\theta_{m}^{2}(0,x)dx\leq C_{0} 
C_{m,d}\Vert u_{0}\Vert_{W_{\chi}^{1,2m}(\mathbb{R}^{d})}^{2m}.
\end{equation*}
This finishes the proof of Theorem~\ref{theorem a priori estimate} in the case $(p,q)=(\infty,2)$.
\end{proof}

Let us come to the general case. Notice that it is only here, for the first and only time, that the exponents $(p,q)$ of the LPS condition enter. By $\left\Vert \cdot\right\Vert _{W^{1,2}}$ and $\left\Vert \cdot\right\Vert _{L^{p}}$ we denote the usual norms in $W^{1,2}(\mathbb{R}^{d})$ and $L^{p}(\mathbb{R}^{d})$ respectively. We first prove the following technical lemma, which will be relevant to continue with the estimate for the terms on the right-hand side of inequality~\eqref{quasi final inequality} for the general case $p \neq \infty$.

\begin{lemma}
\label{lemma interpolation}If $p>d\vee2$, then for every $\varepsilon>0$ there
is a constant $C_{\varepsilon}>0$, depending only on $p,d$ and~$\eps$, such that for all $f,g\in C_{c}^{\infty
}(\mathbb{R}^{d})$ we have 
\begin{equation}
\int_{\mathbb{R}^{d}}\vert f(x)  g(x) \vert^{2}dx\leq\varepsilon \Vert g\Vert_{W^{1,2}} 
^{2}+C_{\varepsilon} \Vert f \Vert_{L^{p}}^{\frac{2p}{p-d} 
} \Vert g \Vert_{L^{2}}^{2}. \label{interpol ineq} 
\end{equation}
If $p=d\geq3$, we have
\begin{equation}
\int_{\mathbb{R}^{d}} \vert f(x)  g(x)
 \vert^{2}dx\leq C_{d} \Vert f \Vert _{L^{d}}^{2} \Vert
\nabla g \Vert _{L^{2}}^{2} \label{interpol ineq p=d} 
\end{equation}
with a constant $C_{d}>0$ depending only on~$d$.
\end{lemma}

\begin{proof}
Let us start by recalling the Gagliardo--Nirenberg interpolation inequality on $\mathbb{R}^{d} 
$ for $d\neq2$: for every $0\leq\beta\leq1$ and $\alpha\geq2$ which satisfy
\begin{equation*}
\frac{1}{\alpha}=\frac{1}{2}-\frac{\beta}{d} 
\end{equation*}
the following holds:  there exists a constant $C > 0$ depending only on $\beta$ and~$d$ such that every $g \in W^{1,2}(\mathbb{R}^{d})$ belongs to $L^{\alpha}(\mathbb{R}^{d})$ with
\begin{equation*}
 \Vert g \Vert _{L^{\alpha}}\leq C \Vert g \Vert _{L^{2} 
}^{1-\beta} \Vert \nabla g \Vert_{L^{2}}^{\beta}.
\end{equation*}
The result is true also for $d=2$ but requires the additional condition $\beta<1$. We apply this inequality with $\beta=\frac{d}{p}$, $\alpha=\frac{2p}{p-2}$.
The assumptions of the Gagliardo--Nirenberg inequality are satisfied because $\beta\leq1$ for $d\neq2$ and $\beta<1$ for $d=2$. Then
\begin{equation*}
 \Vert g \Vert _{L^{\frac{2p}{p-2}}}\leq C \Vert g \Vert
_{L^{2}}^{1-\frac{d}{p}} \Vert \nabla g \Vert_{L^{2}}^{\frac{d}{p} 
}.
\end{equation*}

Now we come the proof of the lemma. We first apply H\"{o}lder's inequality with exponents $\frac{p}{2}$ and
$\frac{p}{p-2}$ and then the previous inequality to find
\begin{equation*}
\int_{\mathbb{R}^{d}} \vert fg \vert ^{2}dx \leq \Vert f \Vert_{L^{p}} 
^{2} \Vert g \Vert_{L^{\frac{2p}{p-2}}}^{2} \leq C\left\Vert f\right\Vert _{L^{p}} 
^{2}\left\Vert g\right\Vert _{L^{2}}^{2(1-\frac{d}{p})
}\left\Vert \nabla g\right\Vert _{L^{2}}^{2\frac{d}{p}},
\end{equation*}
which is the claim~\eqref{interpol ineq p=d} for $p=d$. For $p>d$, we use Young's
inequality 
\[
ab\leq\frac{a^{r}}{r}+\frac{b^{r^{\prime}}}{r^{\prime}},\qquad r,r^{\prime
}>1,\qquad\frac{1}{r}+\frac{1}{r^{\prime}}=1
\]
with
\[
r=\frac{p}{d},\qquad a=(r\varepsilon)  ^{\frac{d}{p}} \Vert
\nabla g \Vert _{L^{2}}^{2\frac{d}{p}},\qquad b=(r\varepsilon
)  ^{-\frac{d}{p}}C \Vert f \Vert _{L^{p}}^{2} \Vert
g \Vert _{L^{2}}^{2(1-\frac{d}{p})  }.
\]
With $r^{\prime}=\frac{p}{p-d}$ we get 
\begin{align*}
\int_{\mathbb{R}^{d}} \vert fg \vert ^{2}dx  &  \leq\varepsilon \Vert \nabla
g \Vert _{L^{2}}^{2}+\frac{(r\varepsilon)  ^{-\frac{d} 
{p}r^{\prime}}C^{r^{\prime}}}{r^{\prime}} \Vert f \Vert _{L^{p} 
}^{2r^{\prime}} \Vert g \Vert _{L^{2}}^{2(1-\frac{d} 
{p})  r^{\prime}}\\
&  =\varepsilon \Vert \nabla g \Vert _{L^{2}}^{2}+p^{-1}(
p-d)  (p\varepsilon/d)  ^{-\frac{d}{p-d}}C^{\frac{p}{p-d} 
} \Vert f \Vert _{L^{p}}^{\frac{2p}{p-d}} \Vert g \Vert
_{L^{2}}^{2},
\end{align*}
and thus, we have found a constant~$C_{\varepsilon}$ such that~\eqref{interpol ineq} holds. This concludes the proof.
\end{proof}

\begin{lemma}
If $b\in\mathcal{LPS}(p,q)$ with $q<\infty$ \textup{(}hence $p>d$\textup{)}, then
\begin{equation}
\int_{0}^{T} \Big( \int_{\mathbb{R}^{d}} \vert b(s,x) \vert
^{p}dx \Big)^{\frac{2}{p-d}}ds\leq \Vert b \Vert _{L^{q}([0,T];L^{p})  }^{q}. \label{bound on b} 
\end{equation}
\end{lemma}

\begin{proof}
From $\frac{2}{q}\leq1-\frac{d}{p}=\frac{p-d}{p}$ we see $\frac{2} 
{p-d}\leq\frac{q}{p}$. Therefore, the assumption $b\in\mathcal{LPS}(p,q)$ with $q<\infty$ implies
\begin{equation*}
\int_{0}^{T} \Big(\int_{\mathbb{R}^{d}} \vert b(s,x)  \vert^{p}dx \Big)^{\frac{2}{p-d}}ds \leq T^{1-\frac{2p}{q(p-d)}} \bigg( \int 
_{0}^{T} \Big( \int_{\mathbb{R}^{d}} \vert b(s,x) \vert^{p} dx \Big)^{\frac{q}{p}}ds \bigg)^{\frac{2p}{q(p-d)  }}<\infty . \qedhere
\end{equation*}
\end{proof}

The previous interpolation Lemma~\ref{lemma interpolation} now allows us to continue with the proof of Theorem~\ref{theorem a priori estimate} in the remaining cases.

\begin{proof}[Proof of Theorem~\ref{theorem a priori estimate} in the case
$(p,q) \in (2,\infty)$]
We start by observing
\begin{align*}
\vert \partial_{i}\theta_{m}\vert  &  \leq\frac{1}{2\theta_{m}} 
\sum_{\vert I\vert =m} \vert \partial_{i}\chi \vert
w_{I}^{2}+\frac{1}{\theta_{m}}\sum_{\vert I\vert =m}\chi\vert
w_{I}\vert \vert \partial_{i}w_{I}\vert \\
& \leq\frac{C_{\chi}}{2\theta_{m}}\sum_{\vert I \vert =m}\chi
w_{I}^{2}+\sqrt{\sum_{\vert I \vert =m}\chi \vert \partial
_{i}w_{I} \vert ^{2}}
\leq\frac{C_{\chi}}{2}\sqrt{\theta_{m}}+\sqrt{\rho_{m}},
\end{align*}
and thus
\[
\Vert \theta_{m} \Vert _{W^{1,2}}^{2}= \Vert \theta
_{m} \Vert_{L^{2}}^{2}+\sum_{i=1}^{d} \Vert \partial_{i}\theta
_{m} \Vert _{L^{2}}^{2}\leq C \Vert \theta_{m} \Vert _{L^{2} 
}^{2}+C \Vert \rho_{m} \Vert _{L^{2}}^{2} 
\]
for some constant $C$ depending only on $\chi$, $d$, $m$. Therefore, the application of inequality~\eqref{interpol ineq} to the terms of the second line of inequality~\eqref{quasi final inequality} shows
\begin{align*}
\int_{0}^{t}\int_{\mathbb{R}^{d}} \vert b^{(1)} \vert
^{2}\theta_{m}^{2}dxds &  \leq\int_{0}^{t} \Big(\varepsilon \Vert
\theta_{m} \Vert _{W^{1,2}}^{2}+C_{\varepsilon} \Vert b^{(
1)} \Vert _{L^{p}}^{\frac{2p}{p-d}} \Vert \theta
_{m} \Vert _{L^{2}}^{2} \Big)  ds\\
&  \leq\varepsilon C\int_{0}^{t} \Vert \rho_{m} \Vert _{L^{2}} 
^{2}ds+C_{\varepsilon}^{\prime}\int_{0}^{t} \Big( 1+ \Vert b^{(1)
}\Vert _{L^{p}}^{\frac{2p}{p-d}} \Big) \Vert \theta_{m} \Vert
_{L^{2}}^{2}ds
\end{align*}
for some constant $C_{\varepsilon}^{\prime}>0$. We use this inequality and the
similar one for $c^{(1)}$ in the second line of inequality~\eqref{quasi final inequality} and get, for $\varepsilon$ small enough and by
means of the Gronwall lemma (applicable because of the inequality~\eqref{bound on b}), a
bound of the form~\eqref{final bound on theta}. With the final arguments used
above in the case $(p,q)=(\infty,2)$, this completes the proof of Theorem~\ref{theorem a priori estimate} in the case $p,q\in(2,\infty)$.
\end{proof}

\begin{proof}[Proof of Theorem~\ref{theorem a priori estimate} in the case
$(p,q) =(d,\infty)$] 
In this case we apply inequality~\eqref{interpol ineq p=d} to the terms of the second line of
inequality~\eqref{quasi final inequality} to find
\begin{equation*}
\int_{0}^{t}\int_{\mathbb{R}^{d}} \vert b^{(1)} \vert
^{2}\theta_{m}^{2}dxds  \leq C_d \int_{0}^{t} \Vert b^{(1) }\Vert _{L^{d}}^2 \Vert \nabla \theta_m \Vert_{L^{2}}^2 ds
\end{equation*}
and an analogous inequality for the term with $c^{(1)}$. We then estimate $\Vert \nabla\theta_{m} \Vert _{L^{2}}^{2}$ as above by $C^{\ast} \Vert \theta_{m} \Vert _{L^{2}}^{2}+C^{\ast} \Vert \rho_{m} \Vert _{L^{2} 
}^{2}$ and get
\begin{align*}
\lefteqn{ \int_{\mathbb{R}^{d}}\theta_{m}^{2}(t,x)dx+\frac{\sigma^{2}}{2}\int_{0}^{t}\int_{\mathbb{R}^{d}}\rho_{m}^{2}dxdr } \\
& \leq C_{m,d,\sigma}C_{d}C^{\ast}\int_{0}^{t} \Big(  \Vert b^{(1) }\Vert _{L^{d}}^2 + \Vert c^{(1) }\Vert _{L^{d}}^2 \Big)  \Vert \rho_{m} \Vert _{L^{2}}^{2}ds\\
& \quad{}  +C_{m,d}\int_{0}^{t} \!\! \Big( \Vert \nabla b^{(2)}(s) \Vert_{\infty}+ C_{\chi}\Big\Vert \frac{b^{(2)}(s,\cdot) }{1+\vert \cdot \vert }\Big \Vert_{\infty} \! + \Vert c^{(2)}(s) \Vert _{\infty}+ \Vert \nabla c^{(2)
}(s) \Vert _{\infty} +2 \Big)  \int_{\mathbb{R}^{d}} \! \theta_{m}^{2}dxds .
\end{align*}
If the smallness condition
\begin{equation}
2C_{m,d,\sigma}C_{d}C^{\ast} \big(\sup_{t\in [0,T]  }  \Vert b^{(1) }\Vert _{L^{d}}^2 +\sup_{t\in [0,T]}  \Vert c^{(1) }\Vert _{L^{d}}^2 \big)  \leq\sigma^{2}\label{precise smallness} 
\end{equation}
is satisfied, we may again apply the Gronwall lemma and the other computations above to conclude the proof of Theorem~\ref{theorem a priori estimate} also in the remaining case $(p,q) =(d,\infty)$.
\end{proof}

\subsection{Existence of global regular solutions for sTE and sCE\label{subsection regularity results}}

In this section we deduce, from the a priori estimates of Theorem
\ref{theorem a priori estimate}, the existence of global regular
solutions for the stochastic generalized transport equation~\eqref{SPDE 1} and consequently also for the stochastic transport equation~\eqref{stoch transport} and the stochastic continuity equation~\eqref{stoch cont}. This can be interpreted, at least for the sTE, as a no-blow-up result. Uniqueness will be treated separately in the next section, see also Remark
\ref{Remark uniqueness regular} below. 

In what follows, we assume that the LPS-integrability condition on $b,c$ with exponents $p \in [d,\infty]$ and $q \in [2,\infty]$ as stated in Section~\ref{subsection regularity assumptions} is satisfied. We further denote by $p^{\prime}=p/(p-1)$ the conjugate exponent of $p$ (with $p^{\prime}=1$ if $p=\infty$). We now start by defining the notion of solutions of class $L^{\theta}(W_{\loc}^{1,m})$ of equation~\eqref{SPDE 1}, for some  $\theta\geq2$ and $m\geq p^{\prime}$. To this end, we require first of all some measurability and continuous semimartingale properties for terms appearing in~\eqref{SPDE 1} after testing against smooth functions. We say that a map $u \colon [0,T] \times \mathbb{R}^{d} \times \Omega \rightarrow \mathbb{R}$ is \emph{weakly progressively measurable} with respect to~$(\mathcal{G}_{t})_t$ if $x\mapsto u(t,x,\omega)  \in L_{\loc}^{1}(\mathbb{R}^{d})$ for a.e.~$(t,\omega)$ and the process $(t,\omega)  \mapsto \int_{\mathbb{R}^{d}} u(t,x,\omega) \varphi(x) dx$ is progressively measurable with respect to $(\mathcal{G}_{t})_t$, for every $\varphi\in C_{c}^{\infty}(\mathbb{R}^{d})$. Secondly, we need all relevant integrals to be well-defined. Due to the choice $\theta\geq2$ we have well-defined stochastic integrals; hence we only need to take care that $b(s) \cdot\nabla u(s)$ and $c(s) u(s)$ are in $L_{\loc}^{1}(\mathbb{R}^{d})$ for a.e.~$s \in [0,T]$. Keeping in mind the decompositions $b = b^{(1)} + b^{(2)}$ and $c = c^{(1)} + c^{(2)}$ into (roughly) a vector field of LPS class and a Lipschitz function, we first note that with $\theta\geq2$ we also have $\theta\geq q/(q-1)$ (recalling $q\geq2$). Therefore, $s\mapsto\left\langle b(s) \cdot\nabla u(s)-c(s)  u(s) ,\varphi\right\rangle $ is integrable according to the choice $m\geq p^{\prime}$ (here, the symbol $\left\langle \cdot,\cdot\right\rangle$ stands for the usual inner product in $L^{2}(\mathbb{R}^{d})$).

These introductory comments now motivate the following definition. 

\begin{definition}
\label{Defin regular solutions}Given $\theta\geq2$, $m\geq p^{\prime}$, \emph{a solution of equation~\eqref{SPDE 1} of class $L^{\theta}(  W_{\loc}^{1,m})$} is a map $u:[0,T]\times\mathbb{R}^{d}\times\Omega \rightarrow \mathbb{R}$ with the following properties:
\begin{enumerate}[font=\normalfont, label=(\roman{*}), ref=(\roman{*})]
 \item[(o)] it is weakly progressively measurable with respect to $(\mathcal{G}_{t})_t$; 
 \item $u$ is in $L^{\theta}([0,T]\times\Omega ;W^{1,m}(B_{R}))$ for every $R>0$; \label{strat0}
 \item \label{strat1} $t\mapsto\left\langle u(t)  ,\varphi\right\rangle $ has a modification that is a continuous semimartingale, for every~$\varphi$ in $C_{c}^{\infty}(\mathbb{R}^{d})$; 
 \item for every~$\varphi$ in $C_{c}^{\infty}(\mathbb{R}^{d})$, for this continuous modification \textup{(}still denoted by $\langle u(t),\varphi\rangle$\textup{)} it holds, with probability one, for all $t \in [0,T]$ \label{strat2}
 \begin{equation}
 \langle u(t),\varphi\rangle=\langle u_{0},\varphi\rangle -\int_{0}^{t}\left\langle b(s) \cdot\nabla u(s) +c(s) u(s),\varphi\right\rangle ds+\sigma \sum_{i=1}^{d}\int_{0}^{t}\left\langle u(s),\partial_{i} \varphi\right\rangle \circ dW_{s}^{i}.\label{SPDEdiff} 
 \end{equation}
\end{enumerate}
\end{definition}

As mentioned above we will now prove the existence of such solutions by exploiting the a~priori Sobolev-type estimates for solutions to approximate equations with smooth coefficients. The crucial point is that the estimates only depend on the LPS norms of the coefficients~$b$ and~$c$, but not on the approximation itself. Hence, from the regular solutions to these approximate equations we may then pass to a limit function which still has the same Sobolev-type regularity, provided that the approximate coefficients remain bounded in these norms. In a second step we then need to verify that the limit function is indeed a solution to the original equation in the sense of Definition~\ref{Defin regular solutions}. 

Concerning the approximation of the coefficients, we first observe that, since~$b$ and~$c$ are assumed to belong to the LPS class (satisfying Condition~\ref{LPSreg}), we may choose sequences $(b_\eps)_\eps$, $(c_\eps)_\eps$ which verify the following assumptions:

\begin{condition}\label{LPSappprox}
We assume $b_\eps=b^{(1)}_\eps+b^{(2)}_\eps$, $c_\eps=c^{(1)}_\eps+c^{(2)}_\eps$, such that:
\begin{itemize}
\item $(b^{(1)}_\eps)_\eps$ is a $C^\infty_c([0,T]\times\mathbb{R}^d)$ approximation of $b^{(1)}$ a.e.~and in $LPS$, in the following sense: if $p,q \in (2,\infty)$, then $b^{(1)}_\eps\rightarrow b^{(1)}$ a.e.~in $[0,T]\times\mathbb{R}^d$ and in $L^q([0,T];L^p(\mathbb{R}^d))$ as $\eps\rightarrow0$; otherwise, if $p$ or $q$ is $\infty$, then $b^{(1)}_\eps\rightarrow b^{(1)}$ a.e.~in $[0,T]\times\mathbb{R}^d$ as $\eps\rightarrow0$ and, for every $\eps>0$, $\|b^{(1)}_\eps\|_{L^q([0,T];L^p(\mathbb{R}^d))}\le 2\|b^{(1)}\|_{L^q([0,T];L^p(\mathbb{R}^d))}$;
\item in case of Condition~\ref{LPSreg}, \textup{1b)}, the $\|b^{(1)}_\eps\|_{C([0,T];L^d(\mathbb{R}^d))}$ norms are small enough, uniformly in~$\eps$ \textup{(}in case of Condition~\ref{LPSreg} \textup{1c)}, this follows from the previous assumption\textup{)};
\item $(c^{(1)}_\eps)_\eps$ is as $(b^{(1)}_\eps)_\eps$ \textup{(}with $c^{(1)}$ in place of $b^{(1)}$\textup{)};
\item $(b^{(2)}_\eps)_\eps$ is a $C^\infty_c([0,T]\times\mathbb{R}^d)$ approximation of $b^{(2)}$ a.e.~and in $L^2(C^1_{\lin})$, in the following sense: $b^{(2)}_\eps\rightarrow b^{(2)}$ a.e.~in $[0,T]\times\mathbb{R}^d$ and in $L^2([0,T];C^1_{\lin}(\mathbb{R}^d))$ as $\eps\rightarrow0$;
\item $(c^{(2)}_\eps)_\eps$ is a $C^\infty_c([0,T]\times\mathbb{R}^d)$ approximation of $c^{(2)}$ a.e.~in $[0,T]\times\mathbb{R}^d$ and in $L^2(C^1_b)$, in the following sense: $c^{(2)}_\eps\rightarrow c^{(2)}$ a.e.~in $[0,T]\times\mathbb{R}^d$ and in $L^2([0,T];C^1_b(\mathbb{R}^d))$ as $\eps\rightarrow0$.
\end{itemize}
\end{condition}

\begin{remark}
Let us briefly explain how Condition~\ref{LPSappprox} allows us to treat general coefficients $b^{(1)}$ and $c^{(1)}$ in $C([0,T];L^d(\mathbb{R}^d))$ \textup{(}see Condition~\ref{LPSreg}, \textup{1b)}\textup{)}, without imposing a smallness condition of the associated norm as for the case of coefficients in $L^\infty([0,T];L^d(\mathbb{R}^d))$. In fact, we can rewrite any coefficients $b^{(1)}$ in $C([0,T];L^d(\mathbb{R}^d))$ as a sum of a regular, compactly supported term \textup{(}say $f$\textup{)} and the remaining, possibly irregular term $b^{(1)} - f$, whose $C([0,T];L^d(\mathbb{R}^d))$ norm can be made arbitrarily small as a consequence of the density of $C([0,T];C^\infty_c(\mathbb{R}^d))$-functions in $C([0,T];L^d(\mathbb{R}^d))$. Thus, we can approximate $b^{(1)} - f$ with $(b^{(1)}_\eps)_\eps$ and $f+b^{(2)}$ with $(b^{(2)}_\eps)_\eps$ \textup{(}analogously~$c$\textup{)}, which in turn ensures that Condition~\ref{LPSappprox} is fulfilled, in particular the smallness of the norm of $b^{(1)}_\eps$.
\end{remark}

\begin{remark}
\label{b_eps-m-prime-convergence}
Notice that, in any case of Condition~\ref{LPSreg} \textup{(}or also for more general~$b$'s\textup{)}, the family $(b_\eps)_\eps$ converges to~$b$ in $L^1([0,T];L^{m'}(B_R))$; the same holds for~$c$.
\end{remark}

\begin{theorem}
\label{Theo existence regular sol}
Let $m\geq2$ be an even integer and let $s\in\mathbb{R}$. Assume that $b$, $c$ satisfy Condition~\ref{LPSreg} and let
$u_{0}\in W_{(1+ \vert \cdot \vert)^{2s+d+1}}^{1,2m}(\mathbb{R}^{d})$. Then there exists a solution~$u$ of equation~\eqref{SPDE 1} of class $L^{m}(W_{\loc}^{1,m})$, which further satisfies $u(t,\cdot)\in W_{( 1+ \vert \cdot \vert)^{s}}^{1,m}(\mathbb{R}^{d})$ for a.e.~$(t,\omega)$. Moreover, there holds
\begin{equation}
\esssup_{t\in[0,T]}E\Big[  \Vert u(t,\cdot) \Vert_{W_{(1+ \vert \cdot \vert)^{s}}^{1,m}(\mathbb{R}^{d})}^{m}\Big]  <\infty . 
\label{regularity of solution}
\end{equation}
\end{theorem}

\begin{proof}
\emph{Step 1: Compactness argument.} Take $(b_\eps)_\eps$ and $(c_\eps)_\eps$ as in Condition~\ref{LPSappprox}; take $(u_0^\eps)_\eps$ as $C^\infty_c(\mathbb{R}^d)$ approximations of the initial datum $u_0$, converging to it a.e.~in~$\mathbb{R}^d$ and in $W_{(1+|\cdot|)^{2s+d+1}}^{1,2m}(\mathbb{R}^d)$. Let $u_\eps$ be the regular solution to~\eqref{SPDE 1} corresponding to coefficients $b_\eps$, $c_\eps$ instead of $b$, $c$, and with initial value $u_0^\eps$, given by Lemma~\ref{lemma preliminare sul caso smooth}. From Corollary~\ref{Corollary a priori estimate} (notice that, in the limit case $p=d$, $b_\eps^{(1)}$ is small enough in view of Condition~\ref{LPSappprox}), we deduce that the family $(u_\eps)_\eps$ is bounded in $L^\infty([0,T];L^m(\Omega;W^{1,m}_{(1+|\cdot|)^{s}}(\mathbb{R}^d)))$. Hence, by Remark~\ref{rem_weihted_spaces_reflexive}, we can extract a subsequence (for simplicity the whole sequence), which converges weakly-$*$ to some~$u$ in that space; in particular, weak convergence in $L^m([0,T]\times\Omega;W^{1,m}(B_R))$ holds for every $R>0$, i.e.~Definition~\ref{Defin regular solutions}~\ref{strat0}.

\emph{Step 2: Weak progressive measurability}. Given $\varphi\in C_{c}^{\infty}(\mathbb{R}^{d})$, the stochastic processes $(t,\omega) \mapsto \langle u_{\eps}(t),\varphi\rangle$ are progressively measurable, weakly convergent in $L^{m}([0,T]\times\Omega)$ to $\langle u,\varphi\rangle$ and the space of progressively measurable processes is closed, so weakly closed, in $L^{m}([0,T]\times\Omega)$. Thus,~$u$ is weakly progressively measurable, i.e.~Definition~\ref{Defin regular solutions}~(o).

\emph{Step 3: Passage from Stratonovich to It\^o and vice versa.} It will be useful to notice that the last two requirements, namely the semimartingale~\ref{strat1} and the solution property~\ref{strat2}, in Definition~\ref{Defin regular solutions} can be replaced by the following It\^o formulation: for every~$\varphi$ in $C_{c}^{\infty}(\mathbb{R}^{d})$, for a.e.~$(t,\omega)$, there holds
\begin{multline}
\langle u(t),\varphi\rangle=\langle u_{0},\varphi\rangle -\int_{0}^{t}\left\langle b(s) \cdot\nabla u(s) + c(s) u(s),\varphi\right\rangle ds \\
  +\sigma \sum_{i=1}^{d}\int_{0}^{t}\left\langle u(s),\partial_{i} \varphi\right\rangle dW_{s}^{i} +\frac{\sigma^2}{2} \int_{0}^{t} \lan u(s),\Delta\varphi\ran ds .\label{SPDEdiffIto}
\end{multline}
Let us prove this fact. Suppose we have the Stratonovich formulation (with~\ref{strat1} and~\ref{strat2}). The Stratono\-vich integral $\sum_{i=1}^{d}\int_{0}^{t}\left\langle u(  s)  ,\partial_{i}\varphi\right\rangle \circ dW_{s}^{i}$ is well-defined, thanks to~\ref{strat1} and our integrability assumptions (with $m,\theta\ge2$), and it is equal to
\begin{equation*}
\sum_{i=1}^{d}\int_{0}^{t}\left\langle u(s) ,\partial_{i}\varphi\right\rangle \circ dW_{s}^{i}=\sum_{i=1}^{d}\int_{0}^{t}\left\langle u(  s)  ,\partial_{i}\varphi \right\rangle dW_{s}^{i}+\sum_{i=1}^{d}\big[  \left\langle u(\cdot),\partial_{i}\varphi\right\rangle ,W^{i}\big]_{t} 
\end{equation*}
where the brackets $[\cdot,\cdot]$ again denote the quadratic covariation. The semimartingale decomposition of $\left\langle u(t),\partial_{i}\varphi\right\rangle $ is taken from the equation for~$u$ (just use $\partial_i\varphi$ instead of~$\varphi$): the martingale part of $\left\langle u(t),\partial_{i}\varphi\right\rangle $ is $\sigma\sum_{j=1}^{d}\int_{0}^{t}\left\langle u(s),\partial_{j}\partial_{i}\varphi\right\rangle \circ dW_{s}^{j}$, so that we have
\begin{equation}
\label{ito-strato-corrector}
\big[\left\langle u(\cdot),\partial_{i} \varphi\right\rangle ,W^{i} \big]_{t} = \sigma\int_{0}^{t} \left\langle
u(  s) ,\partial_{i}^{2}\varphi\right\rangle ds .
\end{equation}
Thus, we get precisely formula~\eqref{SPDEdiffIto} from~\ref{strat2}.

Now suppose we have the It\^o formulation~\eqref{SPDEdiffIto}. This implies that $t\mapsto \langle u(t),\varphi \rangle $ has a modification that is a continuous semimartingale, i.e.~Definition~\ref{Defin regular solutions}~\ref{strat1}. The same is true for $t\mapsto\left\langle u(t),\partial_{i} \varphi\right\rangle $ for $i=1,\ldots,d$, and thus the quadratic covariation $[\langle u(\cdot),\partial_{i}\varphi \rangle ,W^{i}]_{t}$ and the Stratonovich integral $\int_{0}^{t} \langle u(s),\partial_{i}\varphi \rangle \circ dW_{s}^{i}$ exist; moreover, by the equation itself we again find~\eqref{ito-strato-corrector}. It follows that
\begin{equation*}
\sigma\sum_{i=1}^{d}\int_{0}^{t}\left\langle u(  s)  ,\partial
_{i}\varphi\right\rangle dW_{s}^{i}+\frac{\sigma^{2}}{2}\int_{0} 
^{t}\left\langle u(  s)  ,\Delta\varphi\right\rangle
ds=\sigma\sum_{i=1}^{d}\int_{0}^{t}\left\langle u(  s)
,\partial_{i}\varphi\right\rangle \circ dW_{s}^{i}
\end{equation*}
and so we deduce~Definition~\ref{Defin regular solutions}~\ref{strat2} from~\eqref{SPDEdiffIto}.

\emph{Step 4: Verification of the equation.} We want to show that~$u$ satisfies~\eqref{SPDE 1}, in the sense of distributions. In view of Step 3, we can use the It\^o formulation~\eqref{SPDEdiffIto}. Fix~$\varphi$ in $C^\infty_c(\mathbb{R}^d)$ with support in $B_R$. We already know from Step 2 that $\lan u_\eps(t),\varphi\ran-\lan u^\eps_0,\varphi\ran$ converges to $\lan u_t,\varphi\ran-\lan u_0,\varphi\ran$ weakly in $L^2([0,T]\times\Omega)$. We will prove that also the other terms in~\eqref{SPDEdiffIto} converge, weakly in $L^1([0,T]\times\Omega)$. The idea for the convergence is the following: assume we have a linear continuous map $G=G(u)$ between two Banach spaces and a bilinear map $F=F(b,u)$ mapping from two suitable Banach spaces into a third one; then, if $b_\eps$ converges to~$b$ strongly and $u_\eps$ converges to~$u$ weakly in the associated topologies, $G(u_\eps)$ and $F(b_\eps,u_\eps)$ converge weakly to $G(u)$ and $F(b,u)$, respectively.

For the term $\int^t_0 \lan b(s)\nabla u(s), \varphi \ran ds$, we take 
\begin{gather*}
F \colon L^1\big([0,T];L^{m'}(B_R) \big) \times L^\infty \big([0,T]; L^m(\Omega;W^{1,m}(B_R))\big) \rightarrow L^1([0,T]\times\Omega) \\
F(b,u)(t,\omega) \coloneqq  \int^t_0 \lan b(s)\nabla u(s)(\omega), \varphi \ran ds;
\end{gather*}
Then $F$ is a bilinear continuous map. Fix $Z$ in $L^\infty([0,T]\times\Omega)$; for the weak $L^1([0,T]\times\Omega)$-convergence we now have to prove that, as $\eps\rightarrow0$,
\begin{equation*}
\int^T_0E \big[ \big(F(b_\eps,u_\eps)-F(b,u) \big) Z \big]dt\rightarrow 0 .
\end{equation*}
Since $b_\eps$ converges strongly to~$b$ in $L^1([0,T];L^{m'}(B_R))$ (see Remark~\ref{b_eps-m-prime-convergence}) and since $u_\eps$ has uniformly (in~$\eps$) bounded norm in $L^\infty([0,T];L^m(\Omega;W^{1,m}(B_R)))$ (according to Step 1), the norm $\|F(b_\eps,u_\eps)-F(b,u_\eps)\|_{L^1([0,T]\times\Omega)}$ is small for~$\eps$ small, and in particular 
\begin{equation*}
\int^T_0E[(F(b_\eps,u_\eps)-F(b,u_\eps))Z]dt \rightarrow 0
\end{equation*}
as $\eps \rightarrow 0$. It remains to prove that $\int^T_0E[(F(b,u_\eps)-F(b,u))Z]dt \to 0$ as $\eps \to 0$. For this purpose, we notice that, by the Fubini--Tonelli theorem,
\begin{equation*}
\int^T_0 E \big[ (F(b,u_\eps)-F(b,u))Z \big] dt=\int^T_0\int_{B_R}E \big[Y (\nabla u_\eps - \nabla u) \big]dxds ,
\end{equation*}
where $Y(s,x,\omega) \coloneqq b(s,x)\varphi(x)\int^T_s Z(t,\omega)dt$. The convergence of the right-hand side now follows easily, since~$Y$ is in $L^1([0,T];L^{m'}(B_R\times\Omega))$ and~$\nabla u_\eps$ converges weakly-$*$ to~$\nabla u$ in $L^\infty([0,T];L^m(B_R \times \Omega))$ (by Step 1), as $\eps \rightarrow 0$. This finishes the proof of convergence for~$F$, and the convergence of the term $\int^t_0 \lan c(s) u(s), \varphi \ran ds$ is established analogously.

For the term $\int^t_0\lan u(s),\partial_i\varphi\ran dW^i_s$, we define $G \colon L^2([0,T]\times B_R\times\Omega)\rightarrow L^2([0,T]\times\Omega)$ by
\begin{equation*}
G(u)(t,\omega)\coloneqq \int^t_0\lan u(s),\partial_i\varphi\ran dW^i_s(\omega);
\end{equation*}
$G$ is a linear continuous map, hence weakly continuous. Therefore, as a consequence of the weak convergence $\lan u_\eps (s),\partial_i\varphi\ran$ to $\lan u (s),\partial_i\varphi\ran$ in $L^2([0,T]\times\Omega)$, we find that also $\int^t_0\lan u_\eps (s),\partial_i\varphi\ran dW^i_s$ converges weakly (to the obvious limit) in $L^2([0,T]\times\Omega)$. The convergence of the last terms in~\eqref{SPDEdiffIto} is easier. Thus, the limit function~$u$ satisfies the identity~\eqref{SPDEdiffIto}, i.e.~it is a solution to~\eqref{SPDE 1} in the It\^o sense, and via Step~3 it then satisfies the properties~\ref{strat1} and~\ref{strat2} in Definition~\ref{Defin regular solutions}. This concludes the proof of Theorem~\ref{Theo existence regular sol}.
\end{proof}

\begin{remark}
\label{Remark uniqueness regular}
Uniqueness of solutions to~\eqref{SPDE 1} will be treated later in great generality \textup{(}note that uniqueness of weak solutions of class $L^{m}(L^{m}_\loc)$, defined in Definition~\ref{Def weak sol SPDE}, implies uniqueness of solutions of class $L^m(W_{\loc}^{1,m})$\textup{)}. However, uniqueness of weak solutions requires the formulation itself of weak solutions, in which we have to assume some integrability of $\diverg b$ which plays no role in Definition~\ref{Defin regular solutions} and Theorem~\ref{Theo existence regular sol}. One may ask whether it is possible to prove uniqueness of solutions of class $L^m(W_{\loc}^{1,m})$ directly, without the theory of weak solutions. The answer is affirmative but we do not repeat the proofs, see for instance \textup{\cite[Appendix A]{FLAGUBPRI10}}. 
\end{remark}

The previous result holds for~\eqref{SPDE 1} and therefore, it covers the sTE by taking $c=0$. The case of the sCE requires $c=\diverg b$ and therefore, it is better to state explicitly the assumptions. The divergence is understood in the sense of distributions.

\begin{corollary}
Let $m\geq2$ be an even integer and let $s\in\mathbb{R}$. Consider the sCE in the form~\eqref{sCE for density} under the assumptions of
Remark~\ref{Remark assumptions sCE} and let $u_{0}\in W_{(1+\vert \cdot\vert )^{2s+d+1}}^{1,2m}(\mathbb{R}^{d})$. Then there exists a solution~$u$ of equation~\eqref{SPDE 1} of class $L^m(W_{\loc}^{1,m})$, which further satisfies $u(t,\cdot)\in W_{(  1+\vert \cdot \vert )^{s}}^{1,m}(\mathbb{R}^{d})$ for a.e.~$(t,\omega)$
and the analogous estimate of~\eqref{regularity of solution}.
\end{corollary}

\subsection{$W^{2,m}$-regularity}\label{Section W^2}

In this section we are interested in the existence of solutions to equation~\eqref{SPDE 1} of higher regularity, more precisely of local $W^{2,m}$-regularity in space. To this end we shall essentially follow the strategy of the local $W^{1,m}$-regularity in space presented above. First, we consider second order derivatives of equation~\eqref{SPDE 1} (instead of first ones) for the smooth solutions of approximate problems with smooth coefficients and derive a parabolic (deterministic) equation for averages of second order derivatives. For this reason we have to assume some LPS condition not only on the coefficients~$b$ and~$c$, but also on their first space derivatives. Once the parabolic equation is derived, we may proceed analogously to above, that is, via the parabolic theory we establish a~priori regularity estimates involving second order derivatives, and finally we pass to the limit to get the regularity statement. 

Let us now start to clarify the assumptions of this section. As motivated above, we roughly assume that in addition to the coefficients~$b$ and~$c$ also their first order derivatives $\partial_{k}b$ and $\partial_{k}c$ (for $k=1,\ldots,d$) satisfy the assumptions of Section
\ref{subsection regularity assumptions}. More precisely, we assume

\begin{condition}
\label{LPSreg second order}
The coefficients~$b$ and~$c$ can be written as $b = b^{(1)} + b^{(2)}$, $c = c^{(1)} + c^{(2)}$ with weakly differentiable functions $b^{(1)}, b^{(2)}, c^{(1)}, c^{(2)}$, and  
for every $k \in \{0,1,\ldots,d\}$ each of the decompositions $\partial_k b = \partial_k b^{(1)} + \partial_k b^{(2)}$ and $\partial_k c = \partial_k c^{(1)} + \partial_k c^{(2)}$ satisfies Condition~\ref{LPSreg}. Note that if Condition~\ref{LPSreg} \textup{1b)} or \textup{1c)} applies, then we need to assume in addition $d \geq 3$.
\end{condition}

We start by deriving, in the smooth setting, suitable a~priori estimates involving second order derivatives of the regular solution, following the strategy of Theorem~\ref{theorem a priori estimate}.

\begin{lemma}
Let $p,q$ be in $(2,\infty)$ satisfying $\frac{2}{q}+\frac{d}{p}\leq1$ or $(p,q)=(\infty,2)$,
let $m$ be positive integer, let $\sigma \neq 0$, and let $\chi$ be a function 
satisfying~\eqref{property of phi}. Assume that~$b$ and~$c$ are a vector field and a scalar field, respectively, such that $b=b^{(1)}+b^{(2)}$, $c=c^{(1)}+c^{(2)}$, with $b^{(i)}$, $c^{(i)}$ in $C_{c}^{\infty}([0,T]\times \mathbb{R}^{d})$ for $i=1,2$. Then there exists a constant $C$ such that, for every $u_{0}$ in $C_{c}^{\infty}(\mathbb{R}^{d})$, the smooth solution~$u$ of equation~\eqref{SPDE 1} starting from $u_{0}$, given by Lemma~\ref{lemma preliminare sul caso smooth}, verifies
\begin{equation*}
\sup_{t\in [0,T]}\sum_{i,j=0}^{d} \int_{\mathbb{R}^{d}} E\big[ (\partial_j \partial_{i}u(t,x))^{m} \big]^{2} \chi(x) dx\leq C \Vert u_{0} \Vert_{W_{\chi}^{2,2m}(\mathbb{R}^{d}) }^{2m}.
\end{equation*}
Here, the constant~$C$ can be chosen similarly as in Theorem~\ref{theorem a priori estimate}, now depending also on the $L^{q}([0,T];L^{p}(\mathbb{R}^{d}))$ norms of $\partial_k b^{(1)}$ and $\partial_k c^{(1)}$, on the $L^{1}([0,T];C_{\lin}^{1}(\mathbb{R}^{d}))$ norms of~$\partial_k b^{(2)}$, and on the $L^{1}([0,T];C_{b}^{1}(\mathbb{R}^{d}))$ norms of $\partial_k c^{(2)}$, for all $k \in \{0,1,\ldots,d\}$.

The result holds also for $(p,q)=(d,\infty)$, provided that the $L^{\infty}([0,T];L^{d}(\mathbb{R}^{d}))$ norms of $\partial_k b^{(1)}$ and $\partial_k c^{(1)}$ are sufficiently small \textup{(}depending only on $m,\sigma$ and~$d$\textup{)} for all $k \in \{0,1,\ldots,d\}$, see Condition~\ref{LPSreg}, \textup{1c)}.
\end{lemma}

\begin{proof}[Sketch of proof]
Let us start again from the formal computation: using the abbreviations $v_{i}\coloneqq \partial_{i}u$ and $\nu_{ij}\coloneqq  \partial_j \partial_{i} u$ (thus $\nu_{ij} = \nu_{ji}$) for $i,j \in \{0,1,\ldots,d\}$ (and with $\partial_0$ the identity operator), we have 
\begin{equation*}
\mathcal{L}v_{i}=-(\partial_{i}b\cdot\nabla u+ \partial_{i}
c u +cv_{i})  ,\quad \text{for } i=1,\ldots,d, \quad \text{and} \quad \mathcal{L}v_{0}=-c u  .
\end{equation*}
Differentiating once more, we find for $i,j \in \{1,\ldots,d\}$ the identity 
\begin{align*}
\mathcal{L} \partial_{j} v_{i} & =  \partial_{j} \mathcal{L} v_{i} - \partial_{j} b \cdot \nabla v_{i} \\
  & = -  \partial_{j} \partial_{i}b\cdot\nabla u - \partial_{i}b\cdot \partial_{j} \nabla u 
  - \partial_{j} \partial_{i} c u -  \partial_{i} c \partial_{j} u  - \partial_{j}c v_{i} 
  - c \partial_{j} v_{i} - \partial_{j} b \cdot \nabla v_{i} .
\end{align*}
Hence, setting again $b_0 \coloneqq  c$, we end up with the equations 
\begin{equation*}
  \mathcal{L} \nu_{ij} =  \left\{ \begin{array}{llll} \displaystyle - \sum_{k=0}^{d} \partial_{j} \partial_{i}b_k \nu_{k0} & \displaystyle - \sum_{k=0}^{d} ( \partial_{i} b_k \nu_{kj} + \partial_j b_k \nu_{ik})  & - b_0 \nu_{ij} \quad & \text{for } i,j \neq 0 , \\
   & \displaystyle - \sum_{k=0}^{d} \partial_{i} b_k \nu_{k0} & - b_0 \nu_{i0} & \text{for } i>j=0 , \\[0.6cm]
  & & - b_0 \nu_{00} & \text{for } i=j=0 .
  \end{array} \right.
\end{equation*}
We next would like to pass to products of the $\nu_{ij}$'s. To this end we consider~$K$ to be an element in $\cup_{m \in \mathbb{N}} ( \{0,1,\ldots,d\} \times  \{0,1,\ldots,d\} )^m$ and set $|K| = m$ if $K \in ( \{0,1,\ldots,d\} \times  \{0,1,\ldots,d\} )^m$. Moreover, we may assume $i \geq j$ for every $(i,j) \in K$. As before, $K \setminus (i,j)$ means that we drop one component in~$K$ of value $(i,j)$, and similarly $K \setminus (i,j) \cup \{k,\ell\}$ now means that we substitute in~$K$ one component of value $(i,j)$ by one of value $(k,\ell)$ if $k \geq \ell$ or by one of the value $(\ell,k)$ otherwise. Again, which component is dropped does not matter because in the end we will only consider expressions which depend on the total number of the single components, but not on their numbering. We now set $\nu_K \coloneqq  \prod_{(i,j) \in K} \nu_{ij}$, and we then infer from the previous equations satisfied by $\nu_{ij}$, via the Leibniz rule and by distinguishing the cases when $j \neq 0$, $i>j=0$ and $i=j=0$, that
\begin{align*}
  \mathcal{L} \nu_K & = \sum_{(i,j) \in K} \nu_{K \setminus (i,j)} \mathcal{L} \nu_{ij} \\
  & = - b_0 |K| \nu_K - \sum_{(i,j) \in K, i>0} \sum_{k=0}^{d} \nu_{K \setminus (i,j) \cup \{k,j\} }\partial _i b_k \\
  & \quad {}- \sum_{(i,j) \in K, i,j>0} \sum_{k=0}^d \Big( \nu_{K \setminus (i,j) \cup \{k,0\} } \partial_{j} \partial_{i}b_k + \nu_{K \setminus (i,j) \cup \{i,k\} } \partial_j b_k \Big).
\end{align*}
Rewriting the Stratonovich term in $\mathcal{L} \nu_K$ via $\sigma \nabla \nu_{K}\circ dW=\sigma\nabla \nu_{K}dW-\frac{\sigma^{2}}{2}\Delta \nu_{K}dt$ and by (formally) taking the expectation, we then obtain that $\omega_K \coloneqq  E[\nu_K]$ satisfies the equation
\begin{multline}
\partial_{t} \omega_{K}+ b\cdot\nabla \omega_{K}+ b_0 |K| \omega_{K}
+ \sum_{(i,j) \in K, i>0} \sum_{k=0}^{d} \omega_{K \setminus (i,j) \cup \{k,j\} }\partial _i b_k \\
+ \sum_{(i,j) \in K, i,j>0} \sum_{k=0}^d \Big( \omega_{K \setminus (i,j) \cup \{k,0\} } \partial_{j} \partial_{i}b_k + \omega_{K \setminus (i,j) \cup \{i,k\} } \partial_j b_k \Big)
 = \frac{\sigma^{2}}{2} \Delta \omega_{K}. \label{PDE for omega}
\end{multline}
This system of equations is of the same structure as the system~\eqref{PDE for w} derived for the averages of products of first order space derivatives of the solution~$u$, with the only difference that now also second order derivatives of the coefficients appear. Analogously to Section~\ref{section_rigorous_w}, one can make the previous computations rigorous for the regular solution of Lemma~\ref{lemma preliminare sul caso smooth} to~\eqref{SPDE 1}, i.e.~the functions $\omega_K(t,x)$ are continuously differentiable in time and satisfy the pointwise equation~\eqref{PDE for omega}.

From here on, we can proceed completely analogously to the proof of Theorem~\ref{theorem a priori estimate}, since -- even though there are more terms involved -- the structure of the system is essentially the same (note that $\chi$ was only introduced after having derived the parabolic equation, hence no second order derivatives of $\chi$ appear in the computations). This finishes the sketch of the proof.
\end{proof}

With the previous lemma we can then deduce the existence of a global regular solution for the stochastic generalized equation. To this end, we introduce in analogy to Definition~\ref{Defin regular solutions} the notion of a solution~$u$ to~\eqref{SPDE 1} of class $W_{\loc}^{2,m}$, just with the additional $W_{\loc}^{2,m}$-regularity in space.

\begin{theorem}
\label{Theo existence regular sol second order}Let $m\geq2$ be an even
integer and let $s\in\mathbb{R}$ be given. Assume that $b$, $c$ satisfy Condition~\ref{LPSreg second order} and let $u_{0}\in W_{(1+\vert \cdot\vert )^{2s+d+1}}^{2,2m}(\mathbb{R}^{d})$. Then there exists a solution~$u$ of equation~\eqref{SPDE 1} of class $L^{m}(W_{\loc}^{2,m})$, which further satisfies $u(t,\cdot)\in W_{(1+\vert \cdot \vert)^{s}}^{2,m}(\mathbb{R}^{d})$ for a.e.~$(t,\omega)$. Moreover, there holds
\begin{equation*}
\esssup_{t\in[  0,T]  }E\Big[  \Vert u(t,\cdot) \Vert_{W_{(1+ \vert \cdot \vert)^{s}}^{2,m}(
\mathbb{R}^{d})}^{m}\Big]  <\infty . 
\end{equation*}
\end{theorem}

\begin{proof}[Sketch of proof]
Since~$b$ and~$c$ are assumed by Condition~\ref{LPSreg second order} to belong to the extended LPS class (extended in the sense that the decomposition into LPS-part and regular part is available for $\partial_k b$ and $\partial_k c$ for each $k \in \{0,1,\ldots, d\}$), we find approximations $b_{\eps}^{(1)}, b_{\eps}^{(2)}, c_{\eps}^{(1)}$ and $c_{\eps}^{(2)}$ of class $C_{c}^{\infty}([0,T]\times\mathbb{R}^d)$ such that all assumptions concerning boundedness or convergence in Condition~\ref{LPSappprox} are satisfied for $\partial_k b_{\eps}^{(1)}, \partial_k b_{\eps}^{(2)}, \partial_k c_{\eps}^{(1)}$ and $\partial_k c_{\eps}^{(2)}$, for every $k=0,1,\ldots,d$. We then set  $b_{\eps}=b_{\eps}^{(1)}+b_{\eps}^{(2)}$,  $c_{\eps}=c_{\eps}^{(1)} + c_{\eps}^{(2)}$. We further choose an $C^\infty_c(\mathbb{R}^d)$-approximation $(u_0^\eps)_\eps$ of the initial values $u_0$ with respect to $W_{(1+\vert \cdot\vert )^{2s+d+1}}^{2,2m}(\mathbb{R}^{d})$ and denote by $u_{\eps}$ the regular solution to~\eqref{SPDE 1} given by Lemma~\ref{lemma preliminare sul caso smooth}, corresponding to coefficients $b_\eps$, $c_\eps$ and initial values $u_0^\eps$.

We now take $\chi = (1+|x|)^{2s+d+1}$ in the previous lemma and then deduce from H\"older's inequality, as in Corollary~\ref{Corollary a priori estimate}, the bound
\begin{equation*}
\sup_{t\in [0,T]}E \Big[ \Vert u_\eps(t,\cdot)\Vert_{W_{(1+\vert \cdot \vert)^{s}}^{2,m}(\mathbb{R}^{d})}^{m} \Big] \leq C \Vert u_{0} \Vert_{W_{(1+\vert \cdot \vert)^{2s+d+1}}^{2,2m}(\mathbb{R}^{d})}^{m},
\end{equation*}
with a constant $C$ which does not depend on the particular approximation, but only on its norms, and therefore this bound holds uniformly in $\eps \in (0,1)$. From this stage we can follow the strategy of the proof of Theorem~\ref{Theo existence regular sol}. Indeed, the previous inequality yields that the family $(u_\eps)_\eps$ is bounded in $L^\infty([0,T];L^m(\Omega;W^{2,m}(B_R)))$ for every $R>0$. Hence, there exists a subsequence weakly-$*$ convergent to a limit process~$u$ in this space. This yields the asserted Sobolev-type regularity involving derivatives up to second order, while the fact that~$u$ is indeed a solution to~\eqref{SPDE 1} with coefficients $b,c$ was already established in the proof of Theorem~\ref{Theo existence regular sol}.
\end{proof}

\begin{remark}
In a similar way one can show higher order Sobolev regularity of type $W^{\ell,m}_{\loc}$, provided that~$b$ and~$c$ are more regular, in the sense that they can be decomposed into $b^{(1)}+b^{(2)}$ and $c^{(1)}+c^{(2)}$ such that each derivative of these decompositions up to order $\ell-1$ satisfies Condition~\ref{LPSreg}. However, it remains an interesting open question to prove a similar result for fractional Sobolev spaces.
\end{remark}

\section{Path-by-path uniqueness for sCE and sTE}\label{pbp_section}

The aim of this section is to prove a path-by-path uniqueness result for both the stochastic transport equation~\eqref{stoch transport} and the stochastic continuity equation~\eqref{stoch cont}. Since we deal with weak solutions, where an integration by parts is necessary at the level of the definition, the general stochastic equation~\eqref{SPDE 1} is not the most convenient one. Let us consider a similar equation in divergence form
\begin{equation}
du+ (  \diverg (bu)  +cu )  dt+\sigma
\diverg  (  u\circ dW_{t})  =0,\qquad u|_{t=0}=u_{0}
\label{SPDE 2}
\end{equation}
for vector fields $b \colon [0,T] \times \mathbb{R}^d \to \mathbb{R}^d$ and $c \colon   [0,T] \times \mathbb{R}^d \to \mathbb{R}$. We observe that\begin{enumerate}[font=\normalfont, label=(\roman{*}), ref=(\roman{*})]
  \item for regular coefficients, the equations~\eqref{SPDE 2} and~\eqref{SPDE 1} are equivalent (renaming~$b$ and~$c$);
  \item the sCE is included in~\eqref{SPDE 2}, with~$u$ as density of the measure $\mu_t$ with respect to the Lebesgue measure;
  \item the sTE is included in~\eqref{SPDE 2}, by formally setting $c=-\diverg b$ (which then gives rise to a restriction  on $\diverg b$ for this equation).
\end{enumerate}

We recall from the introduction that all path-by-path uniqueness results rely heavily on the regularity results achieved in the previous section. For this reason we will always assume Condition~\ref{LPSreg} of Section~\ref{subsection regularity assumptions} to be in force, which allows us to decompose the vector fields~$b$ and~$c$ into rough parts $b^{(1)}$ and $c^{(1)}$ under a LPS-condition and more regular parts $b^{(2)}$ and $c^{(2)}$ under an integrability condition in time (only here the $L^2$-integrability in time is required, cp.~Remark~\ref{L2time}). Concerning the LPS-condition, we still denote the exponents by $p,q \geq 2$ and the conjugate exponent of $p$ by $p^{\prime}$. We will consider the purely stochastic case $\sigma\neq0$ throughout this section.

We can now introduce the concept of a weak solution of the stochastic equation~\eqref{SPDE 2}, in analogy to Definition~\ref{Defin regular solutions} (in particular, it is easily verified that all integrals are well-defined by the integrability assumptions on the vector fields~$b$ and~$c$ and on the weak solution). We recall that $(\mathcal{G}_t)_{t\in[0,T]}$ is a filtration satisfying the standard assumptions and that~$W$ denotes a Brownian motion with respect to $(\mathcal{G}_t)_t$.

\begin{definition}
\label{Def weak sol SPDE}
Let $m\geq2$ be given. A \emph{weak solution of equation~\eqref{SPDE 2} of class $L^{m}(L^{m}_\loc)$} is a random field $u \colon \Omega\times [0,T]
\times\mathbb{R}^{d}\rightarrow\mathbb{R}$ with the following properties:\begin{enumerate}[font=\normalfont, label=(\roman{*}), ref=(\roman{*})]
  \item[(o)] it is weakly progressively measurable with respect to $(\mathcal{G}_{t})_t$;
  \item it is in $L^{m}([0,T] \times B_{R}\times \Omega)$ for every $R>0$;
  \item $t\mapsto\left\langle u(t)  ,\varphi\right\rangle $ has a modification which 
  is a continuous semimartingale, for every~$\varphi$ in $C_{c}^{\infty}(\mathbb{R}^{d})$;
  \item for every~$\varphi$ in $C_{c}^{\infty}(\mathbb{R}^{d})$, for this continuous modification \textup{(}still denoted by $\lan u(t),\varphi\ran$\textup{)} it holds, with probability one, for all $t \in [0,T]$
\[\label{formulation_SPDE2}
\langle u(t)  ,\varphi\rangle=\langle u_{0},\varphi\rangle 
+\int_{0}^{t}\left\langle u(s)  ,b(s)  \cdot
\nabla\varphi-c(s)  \varphi\right\rangle ds + \sigma \int_{0}^{t}\left\langle u(s)  ,\nabla\varphi\right\rangle \circ dW_{s} .
\]
\end{enumerate}
\end{definition}

\begin{remark}
The previous definition can be given with different degrees of integrability in time and space, namely for solutions of class $L^{\theta}(L_{\loc}^{m})$ with $\theta\geq2$ and $m\geq p^{\prime}$ \textup{(}cp.~Definition~\ref{Defin regular solutions}\textup{)}. We take $\theta=m$ only to minimize the notations.
\end{remark}

Since our aim is to establish the stronger results of path-by-path uniqueness, we first give a path-by-path formulation of~\eqref{SPDE 2}. Let us recall that we started with a probability space $(\Omega,\mathcal{A},P)$, a filtration $(\mathcal{G}_{t})_{t\geq0}$ (satisfying the standard assumptions), and a Brownian motion $(W_{t})_{t\geq0}$. We now choose, without restriction, a version of $W_{t}$ which is continuous for every $\omega\in\Omega$. Given $\omega\in\Omega$, considered here as a parameter, we define
\begin{align*}
\widetilde{b}(\omega,t,x) &  \coloneqq b(t,x+\sigma W_{t}(\omega)) \\
\widetilde{c}(\omega,t,x) &  \coloneqq c(t,x+\sigma W_{t}(\omega)) .
\end{align*}
We shall sometimes write $\widetilde{b}^{\omega}$ and $\widetilde{c}^{\omega}$ for $\widetilde{b}(\omega,\cdot,\cdot)$ and $\widetilde{c}(\omega,\cdot,\cdot)$, respectively, in order to stress the parameter character of~$\omega$. With this new notation we now consider the following deterministic PDE, parametrized by $\omega\in\Omega$, in the unknown $\widetilde{u}^{\omega} \colon [0,T]\times\mathbb{R}^{d}\rightarrow\mathbb{R}$:
\begin{equation}
\partial_t \widetilde{u}^{\omega} + \diverg (\widetilde{b}^{\omega}\widetilde{u}^{\omega})  +\widetilde{c}^{\omega}\widetilde{u}^{\omega}=0,\qquad\widetilde{u}^{\omega}|_{t=0}=u_{0}.  
\label{random PDE 1}
\end{equation}

\begin{definition}
\label{Def weak sol random PDE}
Let $m\geq2$. Given $\omega\in\Omega$, we say that $\widetilde{u}^{\omega} \colon [0,T] \times\mathbb{R}^{d}\rightarrow\mathbb{R}$ is a \emph{weak solution to equation~\eqref{random PDE 1} of class $L^{m}(L^{m}_\loc)$} if 
\begin{enumerate}[font=\normalfont, label=(\roman{*}), ref=(\roman{*})]
 \item $\widetilde{u}^{\omega}\in L^{m}([0,T] \times B_{R})$, for every $R>0$;  \label{rand_Def_1}
 \item for each~$\varphi$ in $C^1([0,T];C_{c}^{\infty}(\mathbb{R}^{d}))$, $t\mapsto\left\langle \widetilde{u}^{\omega}(t),\varphi(t)\right\rangle $ is continuous; precisely, this map has a continuous representative, where by representative we mean a function which coincides with $t\mapsto\left\langle \widetilde{u}^{\omega}(t),\varphi(t)\right\rangle $ for $\mc{L}^1$-a.e.~$t\in [0,T]$;
 \item for all $\varphi\in C^1([0,T];C_{c}^{\infty}(\mathbb{R}^{d}))$, for this continuous representative it holds for all $t \in [0,T]$
 \begin{equation}\label{formulation_randomPDE}
 \left\langle \widetilde{u}^{\omega}(t)  ,\varphi(t)\right\rangle =\left\langle u_{0},\varphi(0)\right\rangle +\int^t_0\left\langle \widetilde{u}^\omega(s),\partial_t\varphi(s) + \widetilde{b}^{\omega}(s)  \cdot\nabla\varphi(s)-\widetilde{c}^{\omega}(s)  \varphi(s) \right\rangle ds .
 \end{equation}
\end{enumerate}
\end{definition}

Notice that we have employed here time-dependent test functions. This is only for a technical convenience (we will use such functions in the following), and the definition with autonomous test functions could be shown to be equivalent to Definition~\ref{Def weak sol random PDE}. Equation~\eqref{random PDE 1} will be considered as the path-by-path
formulation of~\eqref{SPDE 2}. The reason is:

\begin{proposition}
\label{Prop equivalence}
If~$u$ is a weak solution of the stochastic equation~\eqref{SPDE 2} of class $L^{m}(L^{m}_\loc)$ in the sense of Definition~\ref{Def weak sol SPDE}, then $\widetilde{u}^{\omega}$ defined as
\begin{equation*}
\widetilde{u}^{\omega}(t,x)  \coloneqq u(t,x+\sigma W_{t}(\omega))
\end{equation*}
is, for a.e.~$\omega\in\Omega$, a weak solution of the deterministic equation~\eqref{random PDE 1} of class $L^{m}(L^{m}_\loc)$ in the sense of Definition~\ref{Def weak sol random PDE}.
\end{proposition}

\begin{remark}\label{techrmk}
The following proof, simple in the idea, becomes tedious because of a technical detail which we will encounter also in the following: equation~\eqref{SPDE 2} resp.~\eqref{random PDE 1}, in its weak formulation, is satisfied by $\lan u,\varphi\ran$ resp.~$\lan\td{u},\varphi\ran$ only for a.e.~$t \in [0,T]$, and the exceptional set in $[0,T]$, where this formulation does not hold, could depend on~$\varphi$,~$\omega$ and the initial datum. This problem can be overcome essentially in every case, but with some small work \textup{(}see also Lemma~\ref{technicality}\textup{)}.
\end{remark}

\begin{proof}[Proof of Proposition~\ref{Prop equivalence}]
The idea of the proof is given by the following formal computation, using the It\^o formula (in Stratonovich form):
\begin{align}
\partial_t\td{u}(t,x)& =\partial_t u(t,x+\sigma W_t)+ \sigma \nabla u(t,x+\sigma W_t)\circ\dot{W_t} \nonumber  \\
 & =-\diverg \big(\td{b}(t,x) \td{u}(t,x)\big)-\td{c}(t,x)\td{u}(t,x) .\label{formalequiv}
\end{align}
Since this does not work rigorously when~$u$ is not regular, one could try to apply the change of variable formula on the test function rather than on~$u$ itself, i.e.~taking $\td{\varphi}(t,x)=\varphi(t,x-\sigma W_t)$ (which is smooth) as test function in equation~\eqref{SPDE 2} for~$u$ and then use a change of variable to get equation~\eqref{random PDE 1} for $\td{u}$, with~$\varphi$ as test function. The problem is that $\td{\varphi}$, besides being time-dependent, is not deterministic (but Definition~\ref{Def weak sol SPDE} only allows deterministic test functions). Thus, we proceed by approximation. The idea is the following: taking a family $(\rho_\eps)_\eps$ of standard symmetric, compactly supported mollifiers, we first use a shifted version of~$\rho_\eps$ as test function, to get an equation for the mollification $u^\eps \coloneqq u*\rho_\eps$ for fixed~$x$; having regularity of~$u^\eps$, we can derive a formula for $u^\eps(t,x) \varphi(t,x- \sigma W_t)$. After integrating in~$x$, taking the limit $\varepsilon \to 0$ and a change of variable, we finally get an equation for~$\td{u}$, still in a weak formulation. 

For simplicity of notation, we set $c=0$ and $\sigma =1$, but all the arguments are valid with immediate extension also in the general case.

\emph{Step 1: For fixed $\varphi \in C^1([0,T];C^\infty_c(\mathbb{R}^d))$, the mollifications~$u^\eps$ satisfy, for a.e.~$(t,x,\omega)$, }
\begin{align}
\lefteqn{ u^\eps(t,x)\varphi(t,x-W_t)} \label{muvarphi}\\
 & = u_0^\eps(x)\varphi(0,x) -\int^t_0(u(s) b(s))*\nabla\rho_\eps(x)\varphi(s,x-W_s) ds \nonumber\\
 & \quad -\int^t_0 u(s)*\nabla\rho_\eps(x)\varphi(s,x-W_s)\cdot d W_s +\frac12 \int^t_0 u(s)*\Delta\rho_\eps(x)\varphi(s,x-W_s) ds \nonumber\\
 & \quad -\int^t_0 u^\eps(s,x)\nabla\varphi(s,x-W_s)\cdot d W_s +\int^t_0 u^\eps(s,x) \big(\partial_t+\frac12\Delta\big)\varphi(s,x-W_s) ds \nonumber\\
 & \quad +\int^t_0 u(s)*\nabla\rho_\eps(x)\cdot\nabla\varphi(s,x-W_s) ds,\nonumber
\end{align}
\emph{and all the addends have modification that are measurable in $(t,x,\omega)$ \textup{(}these are the modifications considered in the equality above\textup{)}.} We fix a measurable map~$u$ (not an equivalence class), so that by Fubini's theorem convolutions of~$u$ are measurable maps in $(t,x,\omega)$. For fixed $x \in \mathbb{R}^d$, we apply Definition~\ref{Def weak sol SPDE} of a weak solution with test function $\varphi=\rho_\eps(x-\cdot) \in C^\infty_0(\mathbb{R}^d)$, getting the following equation for a modification $u(\rho_\eps(x-\cdot))$ of $u^\eps(x) = u*\rho_\eps (x) = \lan u, \rho_\eps(x-\cdot) \ran$ (here the notation $\nabla_\cdot \rho_\eps(x-\cdot)$ means the derivative with respect to the $\cdot$ variable, with $x$ fixed):
\begin{align}
\label{eqn_u_eps_tested_1}
\lefteqn{u(\rho_\eps(x-\cdot))(t)}\\ 
 & = \langle u_{0},\rho_\eps(x-\cdot) \rangle +\int_{0}^{t} \lan u(s),b(s)  \cdot
\nabla_\cdot \rho_\eps(x-\cdot)\ran ds + \int_{0}^{t} \lan u(s),\nabla_\cdot \rho_\eps(x-\cdot) \ran \circ dW_{s} \nonumber \\
 & = u_0^\eps(x) -\int^t_0( u(s) b(s))*\nabla\rho_\eps(x) d s - \int^t_0  u(s)*\nabla\rho_\eps(x)\cdot d W_s +\frac12\int^t_0 u(s)*\Delta\rho_\eps(x) ds, \nonumber
\end{align}
where we also have passed from Stratonovich to It\^o stochastic integral. Applying It\^o's product formula to $u(\rho_\eps(x-\cdot))$ and $\varphi(t,x-W_t)$, we find that $P$-a.s.~it hold for every~$t \in [0,T]$
\begin{align}
\label{eqn_u_eps_tested_2}
\lefteqn{u(\rho_\eps(x-\cdot))(t)\varphi(t,x-W_t)} \\
 & = u_0^\eps(x)\varphi(0,x) -\int^t_0(u(s) b(s))*\nabla\rho_\eps(x)\varphi(s,x-W_s) d s \nonumber \\
 & \quad -\int^t_0 u(s)*\nabla\rho_\eps(x)\varphi(s,x-W_s)\cdot d W_s +\frac12 \int^t_0 u(s)*\Delta\rho_\eps(x)\varphi(s,x-W_s) ds \nonumber \\
 & \quad - \int^t_0 u(\rho_\eps(x-\cdot))(s)\nabla\varphi(s,x-W_s)\cdot d W_s 
 + \int^t_0 u(\rho_\eps(x-\cdot))(s) \big(\partial_t+\frac12\Delta\big)\varphi(s,x-W_s) d s \nonumber \\
 & \quad +\int^t_0 u(s)*\nabla\rho_\eps(x)\cdot\nabla\varphi(s,x-W_s) ds. \nonumber
\end{align}
Since, for fixed $x$, $u(\rho^\eps(x-\cdot))=u^\eps(x)$ for a.e.~$(t,\omega)$, we can replace $u(\rho^\eps(x-\cdot))$ with $u^\eps(x)$ inside the integrals, which implies~\eqref{muvarphi} for all $(t,\omega)$ in a full-measure set $A_x$, possibly depending on~$x$. Note that, up to this point, we have not used any measurability in~$x$.

Now let us justify that all the addends in~\eqref{muvarphi} have modifications which are measurable in $(t,x,\omega)$. By the classical Fubini theorem, the mollifications of~$u$ and thus all the addends but the stochastic integrals are measurable in $(t,x,\omega)$. Concerning the stochastic integrals, their integrands are, in view of the weak progressive measurability of~$u$, measurable in $(t,x,\omega)$ with respect to $\mc{P}\otimes\mc{B}(\mathbb{R}^d)$, where $\mc{P}$ is the progressive $\sigma$-algebra. Thus, the stochastic Fubini theorem (see e.g.~\cite[Theorem~2.2]{VERAAR12} applies and gives the existence of measurable modifications in $(t,x,\omega)$. For such modifications,~\eqref{muvarphi} must holds for a.e.~$(t,x,\omega)$: if this were not the case, then there would exist a positive measure set $B$ in $\mathbb{R}^d$, such that, for every~$x \in B$, there would exist a positive measure set $C_x$ in $[0,T]\times\Omega$ where equality~\eqref{muvarphi} would not hold. Since the addends of~\eqref{muvarphi} are modifications of the addends of those of~\eqref{eqn_u_eps_tested_2}, also~\eqref{eqn_u_eps_tested_2} would not hold on this set, which is a contradiction, cf.~Remark~\ref{simple_rmk}.

\emph{Step 2: For fixed $\varphi \in C^1([0,T];C^\infty_c(\mathbb{R}^d))$, $\td{u}$ has the solution property~\eqref{formulation_randomPDE} a.s..} We may now integrate, for a.e.~$(t,\omega)$, the identity~\eqref{muvarphi} with respect to~$x$, obtaining 
\begin{align}
\label{eqn_u_eps_tested_3}
\lefteqn{\int_{\mathbb{R}^d} u^\eps(t,x)\varphi(t,x-W_t) dx} \\
 & = \int_{\mathbb{R}^d} u_0^\eps(x)\varphi(0,x) d x -\int^t_0\int_{\mathbb{R}^d}( u(s)b(s))*\nabla\rho_\eps(x)\varphi(s,x-W_s) dx ds \nonumber\\
 & \quad -\int^t_0\int_{\mathbb{R}^d} u(s)*\nabla\rho_\eps(x)\varphi(s,x-W_s) d x\cdot d W_s +\frac12\int^t_0\int_{\mathbb{R}^d} u(s)*\Delta\rho_\eps(x)\varphi(s,x-W_s) dx ds \nonumber\\
 & \quad -\int^t_0\int_{\mathbb{R}^d} u^\eps(s,x)\nabla\varphi(s,x-W_s) d x\cdot d W_s +\int^t_0\int_{\mathbb{R}^d} u^\eps(s,x)(\partial_t+\frac12\Delta)\varphi(s,x-W_s) d x d s \nonumber\\
 & \quad +\int^t_0\int_{\mathbb{R}^d} u(s)*\nabla\rho_\eps(x)\cdot\nabla\varphi(s,x-W_s) dx ds,\nonumber
\end{align}
where we have also used the classcial Fubini as well as the stochastic Fubini theorem to exchange the order of integration. Employing once again Fubini's theorem to bring the convolution on $\varphi(t,\cdot-W_t)$, we get, for a.e.~$(t,\omega)$,
\begin{align}
\lefteqn{\lan u(t),\varphi^\eps(t,\cdot-W_t)\ran} \nonumber\\
 &= \lan u_0,\varphi^\eps(0)\ran + \int^t_0\lan u(s),b(s)\cdot\nabla\varphi^\eps(s,\cdot-W_s)\ran ds \nonumber\\
 & \quad +\int^t_0\lan u(s),\nabla\varphi^\eps(s,\cdot-W_s)\ran\cdot d W_s +\frac12\int^t_0\lan u(s),\Delta\varphi^\eps(s,\cdot-W_s)\ran ds \nonumber\\
 & \quad -\int^t_0\lan u(s),\nabla\varphi^\eps(s,\cdot-W_s)\ran\cdot d W_s + \int^t_0\lan u(s),(\partial_t+\frac12\Delta)\varphi^\eps(s,\cdot-W_s)\ran ds \nonumber\\
 & \quad -\int^t_0\lan u(s),\Delta\varphi^\eps(s,\cdot-W_s)\ran ds  \nonumber\\
 & = \lan u_0,\varphi^\eps(0) \ran + \int^t_0\lan u(s),(\partial_t+b(s)\cdot\nabla)\varphi^\eps(s,\cdot-W_s)\ran d s.\nonumber
\end{align}
Letting $\eps\rightarrow0$, since $ u$, $b u$ are in $L^1([0,T];L^1_\loc(\mathbb{R}^d))$ for a.e.~$\omega$, we have for a.e.~$(t,\omega)$,
\begin{equation*}
\lan u(t),\varphi(t,\cdot-W_t)\ran = \lan u_0,\varphi(0)\ran +\int^t_0\lan u(s),(\partial_t+b(s)\cdot\nabla)\varphi(s,\cdot-W_s)\ran ds.
\end{equation*}
By the change of variable $\tilde{x}=x-W_t$, we therefore end up with the claimed solution property
\begin{equation}
\lan\td{u}(t),\varphi(t)\ran = \lan u_0,\varphi(0)\ran +\int^t_0\lan\td{u}(s),(\partial_t+\td{b}(s)\cdot\nabla)\varphi(s)\ran ds\label{almost_rCE}
\end{equation}
for fixed test function $\varphi \in C^1([0,T];C^\infty_c(\mathbb{R}^d))$, for every $(t,\omega)$ in a full measure set~$F_\varphi$, which may still depend on~$\varphi$.

{\emph{Step 3: Removal of the dependency on the test function~$\varphi$.}} In order to conclude the proof of the proposition, we need to make the ``good'' full measure set, where $\td{u}$ satisfies the solution property, independent of~$\varphi$. For this purpose, we use a density argument. Let~$\mathcal{D}$ be a countable set in $C^1([0,T];C^\infty_c(\mathbb{R}^d))$, which is dense in $C^1([0,T]; C^2_b(\mathbb{R}^d))$, and set $F=\cap_{\varphi\in D}F_\varphi$. Then~$F$ is a full measure set and the identity~\eqref{almost_rCE} holds for every $(t,\omega) \in F$ and~$\varphi \in \mathcal{D}$; after possibly passing to a smaller full-measure set $F$ we can also assume $\td{u}^\omega \in L^m([0,T];L^m_\loc(\mathbb{R}^d))$ (thus, fulfilling Definition~\ref{Def weak sol random PDE}~\ref{rand_Def_1}). Now, for a generic test function~$\varphi \in C^1([0,T];C^\infty_c(\mathbb{R}^d))$, we take a sequence $(\varphi^n)_{n \in \mathbb{N}}$ in~$\mathcal{D}$, satisfying equation~\eqref{almost_rCE} and converging to~$\varphi$ in $C^1([0,T]; C^2_b(\mathbb{R}^d))$; by the dominated convergence theorem, we can pass to the limit in the equation, for $(t,\omega) \in F$, getting~\eqref{almost_rCE} for~$\varphi$. Hence, for a.e.~$(t,\omega)$,~\eqref{formulation_randomPDE} holds and the right-hand side defines the continuous representative.
\end{proof}

Since some technical measurability arguments are delicate in the above proof (based mostly on classcial Fubini and stochastic Fubini theorems), we want to give alternative proofs of Step~1 and formula~\eqref{eqn_u_eps_tested_3} at the beginning of Step~2, which rely on a direct exchange of integral formula obtained by continuity of approximations.

\begin{proof}[Alternative proof of Step~1 and~\eqref{eqn_u_eps_tested_3}]\emph{Step 0: Exchange of integrals formula by approximation.} Let $f \colon [0,T]\times\mathbb{R}^d\times\Omega\rightarrow\mathbb{R}$ be a function such that: \begin{itemize}
 \item $f$ is measurable in $(t,x,\omega)$,
 \item for every~$x$, $(t,\omega)\mapsto f(t,x,\omega)$ is progressively measurable, 
 \item $f \in L^2([0,T]\times\Omega;C^\alpha_{\loc}(\mathbb{R}^d))$ for some $\alpha>0$. 
\end{itemize}
Then the family of stochastic integrals $\int^t_0 f(r,x) d W_r$, parametrized by~$x$, admits a modification which is measurable in $(t,x,\omega)$,  for every~$x$ progressively measurable in $(t,\omega)$, and for a.e.~$\omega$ locally H\"older continuous in $(t,x)$. This can be proven by Kolmogorov's continuity criterion in $(t,x)$ for the stochastic integrals (joint measurability is a consequence of progressive measurability and continuity in $(t,x)$). Moreover, for such modification, we have for a.e.~$\omega \in \Omega$: for every~$t \in [0,T]$,
\begin{align*}
\int_{\mathbb{R}^d}\int^t_0 f(r,x) dW_r dx = \int^t_0\int_{\mathbb{R}^d} f(r,x) dx dW_r,
\end{align*}
provided the integrals are well-defined (for example, if $f$ is compactly supported). This is a consequence of the stochastic Fubini theorem but can be proved without it:

For this purpose, we first observe that by continuity of the stochastic integrals in $(t,x)$, we can approximate, for fixed~$t$, for a.e.~$\omega \in \Omega$, the left-hand side $\int_{\mathbb{R}^d}\int^t_0 f(r,x) dW_r dx$ with a finite Riemann sum (in~$x$) of stochastic integrals. We then notice that we can approximate the inner integral $\int_{\mathbb{R}^d} f(r,x) dx$ in $L^2([0,T]\times\Omega)$ with a finite Riemann sum (in~$x$), and as a consequence, we can approximate, for fixed~$t$, the right-hand side $\int^t_0\int_{\mathbb{R}^d} f(r,x) dx dW_r$  in $L^2(\Omega)$ with the stochastic integral of a finite Riemann sum (in~$x$). At the level of these approximations sums we can finally exchanging sum and stochastic integral, and passing to the limit we get the equality above.

\emph{Alternative proof of Step 1 above.} As before, we fix a measurable map~$u$ (not an equivalence class), so that, by Fubini's theorem, all the convolutions with~$u$ are measurable maps in $(t,x,\omega)$, regular in~$x$ for a.e.~$(t,\omega)$ fixed. Then, for fixed $x \in \mathbb{R}^d$, our starting point is the modification $u(\rho_\eps(x-\cdot))$ of $u^\eps(x) = u*\rho_\eps (x) = \lan u, \rho_\eps(x-\cdot) \ran$ satisfying~\eqref{eqn_u_eps_tested_1} and~\eqref{eqn_u_eps_tested_2}. Replacing $u(\rho^\eps(x-\cdot))$ with $u^\eps(x)$ inside the integrals of~\eqref{eqn_u_eps_tested_2} (as before), we get for a.e.~$\omega$, for every~$t$, 
\begin{align*}
\lefteqn{ u(\rho_\eps(x-\cdot))(t)\varphi(t,x-W_t)} \\
 & = u_0^\eps(x)\varphi(0,x) -\int^t_0(u(s) b(s))*\nabla\rho_\eps(x)\varphi(s,x-W_s) ds \\
 & \quad -\int^t_0 u(s)*\nabla\rho_\eps(x)\varphi(s,x-W_s)\cdot d W_s +\frac12 \int^t_0 u(s)*\Delta\rho_\eps(x)\varphi(s,x-W_s) ds \\
 & \quad -\int^t_0 u^\eps(s,x)\nabla\varphi(s,x-W_s)\cdot d W_s +\int^t_0 u^\eps(s,x) \big(\partial_t+\frac12\Delta\big)\varphi(s,x-W_s) ds \\
 & \quad +\int^t_0 u(s)*\nabla\rho_\eps(x)\cdot\nabla\varphi(s,x-W_s) ds.\end{align*}
For the stochastic integrals, the integrands $u(s)*\nabla\rho_\eps(x)\varphi(s,x-W_s)$ and $u^\eps(s,x)\nabla\varphi(s,x-W_s)$ are measurable in $(t,x,\omega)$, progressively measurable for every fixed~$x$, and they also belong to $L^2([0,T]\times\Omega;C^1_{\loc}(\mathbb{R}^d))$. Therefore, by Step~0, there exist ``nice'' modifications of the stochastic integrals. Using these modifications, we get for every~$x$, for a.e.~$(t,\omega)$ (where the exceptional set possibly depends on~$x$) precisely the formula~\eqref{muvarphi}. Moreover, since all the addends are measurable in $(t,x,\omega)$ by construction, this equality is true for a.e.~$(t,x,\omega)$ (otherwise we would find positive measure sets $A_x$ in $[0,T]\times\Omega$, for some~$x$, where the equality above would not hold).

\emph{Alternative justification of~\eqref{eqn_u_eps_tested_3}.} As before we again integrate~\eqref{muvarphi} in~$x$, for a.e.~$(t,\omega)$, but at this stage we may then use Fubini's theorem to exchange the integrals in~$ds$ and~$dx$, while we may use Step~0 to exchange the integral in~$dW_s$ and~$dx$.
\end{proof}

\begin{remark}
One can ask why such a change of variable works and if this is simply a trick. Actually this is not the case: as we will see in Section~\ref{section flow}, this change of variable corresponds to looking at the random ODE
\begin{equation*}
d\td{X}^\omega=\td{b}^\omega(t,\td{X}^\omega)dt.
\end{equation*}
A similar change of variable can be done also for more general diffusion coefficients, see the discussion in the Introduction, Subsection~\ref{general}.
\end{remark}

\subsection{The duality approach in the deterministic
case\label{subsection determ duality}}

To prove uniqueness for equation~\eqref{random PDE 1}, we shall follow a duality approach. It is convenient to recall the idea in a deterministic case first, especially in view of condition~\eqref{tool deterministic} further below. For the sake of illustration, we give here a Hilbert space description, even though the duality approach will be developed later in the stochastic case in a more general set-up. 

Assume we have a Hilbert space~$H$ with inner product $\langle \cdot,\cdot \rangle _{H}$ and two Hilbert spaces $D_{A}, D_{A^{\ast}}$ which are continuously embedded in~$H$, $D_{A}\subset H$ and $D_{A^{\ast}}\subset H$. Furthermore, let $A(t) \colon D_{A}\rightarrow H$ and $A(t)^{\ast} \colon D_{A^{\ast}} \rightarrow H$ be two families of bounded linear operators such that
\begin{equation*}
\left\langle A(t)  x,y\right\rangle _{H}=\left\langle x,A(t)^{\ast}y\right\rangle _{H}
\end{equation*}
for all $x\in$ $D_{A}$, $y\in D_{A^{\ast}}$. Consider the linear evolution
equation in $H$ 
\begin{equation}
\partial_t u(t) -A(t)  u(t)   = 0 \quad \text{for }t\in [0,T], \label{first evol eq} \qquad
u|_{t=0}  =u_{0}
\end{equation}
and suppose that we want to study uniqueness of weak solutions, defined as those
functions $u \colon [0,T] \rightarrow H$, bounded and weakly
continuous, such that
\begin{equation*}
\left\langle u(t)  ,\varphi\right\rangle _{H}=\left\langle u_{0},\varphi\right\rangle _{H}+\int_{0}^{t}\left\langle u(s) ,A(s)^{\ast}\varphi\right\rangle_{H}ds
\end{equation*}
for all $\varphi\in D_{A^{\ast}}$ and all $t \in [0,T]$. Assume we can prove that this weak formulation implies
\begin{equation}
\label{equation_phi_cont}
\left\langle u(t),\varphi(t) \right\rangle _{H}=\left\langle u_{0},\varphi(0) \right\rangle_{H}+\int_{0}^{t} \Big\langle u(s),A(s)^{\ast}
\varphi(s)  + \partial_t \varphi(s) \Big\rangle_{H} ds
\end{equation}
for all $\varphi\in C([0,T];D_{A^{\ast}}) \cap C^{1}([0,T];H)$ and all $t \in [0,T]$. In order to identify~$u$ at any time $t_{f}\in [0,T]$, we need to consider the dual problem on
$[0,t_{f}]$ with final condition at time $t_{f}$. Thus, given any $t_{f}\in [0,T]$, we consider the equation
\begin{equation}
\partial_t v(t) +A(t)^{\ast}v(t)  =0 \quad \text{for } t \in [0,t_{f}], \label{dual evol eq} 
\qquad v|_{t=t_{f}} = v_{0}
\end{equation}
and assume that, for every $v_{0}$ in a dense set $\mathcal{D}$ of $H$, it has a regular solution $v\in C([0,t_{f}];D_{A^{\ast}}) \cap C^{1}([0,t_{f}];H)$. Then by the previous assumption~\eqref{equation_phi_cont} we obtain with the choice $\varphi =v$ that
\begin{equation*}
\left\langle u(t_{f}) ,v_{0}\right\rangle_{H}=\left\langle
u_{0},v(0) \right\rangle_{H}.
\end{equation*}
If $u_{0}=0$, then $\left\langle u(t_{f}) ,v_{0}\right\rangle_{H}=0$ for every $v_{0}\in\mathcal{D}$, hence $u(t_{f}) =0$. This implies uniqueness for equation~\eqref{first evol eq} by linearity.

Let us repeat this scheme (still considering the case $u_0=0$), when a regularized version of the dual equation is used. Assume we have a sequence of (smooth) approximations of equation~\eqref{dual evol eq}
\begin{equation*}
\partial_t v_{\eps} (t) +A_{\eps}(t)^{\ast} v_{\eps}(t) =0 \quad \text{for } t\in[0,t_{f}], 
\qquad v_{\eps}|_{t=t_{f}} = v_{0} ,
\end{equation*}
where $A_{\eps}(t)^{\ast} \colon D_{A^{\ast}}\rightarrow H$. If, for every final datum $v_{0}\in\mathcal{D}$, we have regular solutions $v_{\eps} \in C([0,t_{f}];D_{A^{\ast}}) \cap C^{1}([0,t_{f}];H)$, then, if $u_0=0$, we find
\begin{equation}
\left\langle u(t_{f}),v_{0}\right\rangle _{H}=\int_{0}^{t_{f} }\left\langle u(s)  ,(  A{}^{\ast}(s) - A_{\eps}{}^{\ast}(s) )  v_{\eps}(s)
\right\rangle _{H}ds  \label{duality formula determ}
\end{equation}
again from~\eqref{equation_phi_cont}, and hence
\begin{equation*}
\vert \left\langle u(t_{f}),v_{0}\right\rangle_{H} \vert \leq \Vert u\Vert_{L^{\infty}(0,T;H)} \int_{0}^{t_{f}} \Vert ( A^{\ast}(s) - A_{\eps}^{\ast}(s)) v_{\eps}(s) \Vert_{H}ds .
\end{equation*}
If, for every $t_{f}\in [0,T]$ and $v_{0}\in\mathcal{D}$, we have
\begin{equation}
\lim_{\eps\rightarrow0}\int_{0}^{t_{f}} \Vert ( A^{\ast}(s)  -A_{\eps}^{\ast}(s))
v_{\eps} (s) \Vert_{H}ds=0 \label{tool deterministic} ,
\end{equation}
then we again conclude with $u=0$, which proves uniqueness for equation~\eqref{first evol eq}. A property of the form~\eqref{tool deterministic} will be a basic tool in the sequel.

The problem to apply this method rigorously is the regularity of~$v$ (or a uniform control of the regularity of~$v_{\eps}$). For deterministic transport and continuity equations with rough drift, one cannot solve the dual equation~\eqref{dual evol eq} in a sufficiently regular space. Thus, the regularity results of Section~\ref{section regularity} are the key point of this approach, specific to the stochastic case.

\subsection{Dual sPDE and random PDE}

Let us recall that we started with a probability space $(\Omega,\mathcal{A},P)$, a (complete and right-continuous) filtration $(\mathcal{G}_{t})_{t\geq0}$ and a Brownian motion $(W_{t})_{t\geq0}$. Given $t_{f}\in [0,T]$ (which will be the final time), we consider the process
\begin{equation}
B_{t}\coloneqq W_{t}-W_{t_{f}},\qquad t\in [0,t_{f}]
\end{equation}
and the family of $\sigma$-fields, for $t \in [0,t_{f}]$,
\begin{equation}
\mathcal{F}^{t}=\sigma (\{  B_{s},s\in [t,t_{f}] \} \cup \mathcal{N} ),
\end{equation}
where $\mathcal{N}$ is the set of $P$-null sets in $\mathcal{A}$. The family $(\mathcal{F}^{t})_{t\in [0,t_{f}]}$ is a backward filtration, in the sense that $\mathcal{F}^{t_{1}} \subset\mathcal{F}^{t_{2}}$ if $t_{1}>t_{2}$. The process~$B$ is a ``backward Brownian motion'', or a ``Brownian motion in the reversed direction of time'', with respect to the filtration $(\mathcal{F}^{t})_t$:

\begin{itemize}
  \item $B_{t_{f}}=0$ a.s., $t\mapsto B_{t}$ is a.s.~continuous (in fact, for all $\omega \in \Omega$ by our choice of $W_t$), 
  \item $B_{t-h}-B_{t}$ is ${\cal N}(0,h)$ and independent of $\mathcal{F}^{t}$, for every $t\in [0,t_{f}]$ and $h \in (0,t]$,
  \item $(\mathcal{F}^{t})_{t\in [0,t_{f}]}$ is complete and right-continuous (see e.g.~\cite[Proposition 2.5]{BASS11}).
\end{itemize}

Stochastic calculus in the backward direction of time can be developed without any difference (except notational) compared to the common forward stochastic calculus, see \cite[Chapter~3]{KUNITA84}. Thus, we may consider the backward version of the sPDE~\eqref{SPDE 1} in Stratonovich form
\begin{equation*}
dv+ ( b\cdot\nabla v  -cv)  dt+\sigma\nabla v\circ
dB  =0 \quad \text{for } t\in [0,t_{f}], \qquad
v|_{t=t_{f}}  =v_{0} ,
\end{equation*}
and define weak or $W^{1,r}$ solutions in the same way as in the forward case. In fact, instead of this equation, we shall use its regularized version
\begin{equation}
\label{backward SPDE 1}
dv_{\eps} + ( b_{\eps}\cdot\nabla v_{\eps} -c_{\eps}v_{\eps})  dt+\sigma\nabla v_{\eps}\circ dB 
=0 \quad \text{for } t\in [0,t_{f}], \qquad v_{\eps}|_{t=t_{f}}  =v_{0} ,
\end{equation}
where $b_{\eps}$, $c_{\eps}$, for $\eps>0$, and~$v_{0}$ are $C_{c}^{\infty}$ functions. First, for every $\eps>0$, this equations has a smooth solution with the properties described in Lemma~\ref{lemma preliminare sul caso smooth}. Second, we have the analog of Theorem~\ref{theorem a priori estimate} and Corollary~\ref{Corollary a priori estimate}:

\begin{corollary}
\label{corollary regularity backward} Let $m$ be an even integer and let $s$ be in $\mathbb{R}^d$. Assume that $b$, $c$ satisfy Condition~\ref{LPSreg}, and let $b_{\eps}$, $c_{\eps}$ be $C_{c}^{\infty}([0,T]\times\mathbb{R}^d)$ functions satisfying Condition~\ref{LPSappprox}. Then there exists a constant~$C$ independent of~$\eps$ such that, for every  $v_{0}\in C_{c}^{\infty}(\mathbb{R}^d)$, the smooth solution~$v_{\eps}$ of equation~\eqref{backward SPDE 1} verifies for all $\eps\in (0,1)$
\begin{equation*}
\sup_{t\in [0,T]} E\Big[ \Vert v_{\eps}(t,\cdot)\Vert_{W_{(1+ \vert \cdot \vert)^{s}}^{1,m}(\mathbb{R}^{d})}^{m}\Big]  \leq C \Vert v_{0}\Vert _{W_{(  1+ \vert \cdot \vert )
^{2s+d+1}}^{1,2m}(\mathbb{R}^{d})}^{m}.
\end{equation*}
\end{corollary}

Moreover, the analog of Proposition~\ref{Prop equivalence} holds. But, about this, let us pay attention to the notations. The result here is:

\begin{corollary}

With the notations $\widetilde{v}_{\eps}^{B}(t,x) \coloneqq v_{\eps}(t,x+\sigma B_{t})$, $\widetilde{b}_{\eps}^{B}(t,x) \coloneqq b_{\eps}(t,x+\sigma B_{t})$, $\widetilde{c}_{\eps}^{B}(t,x) \coloneqq c_{\eps}(t,x+\sigma B_{t})$ we have for a.e.~$\omega\in\Omega$ that $\widetilde{v}_{\eps}^{B}$ has paths in $C^{1}([0,t_{f}];C_{c}^{\infty}(\mathbb{R}^{d}))$ and that there holds
\begin{equation}
\partial_t \widetilde{v}_{\eps}^{B}+\widetilde{b}_{\eps}^{B}\cdot\nabla\widetilde{v}_{\eps}^{B}-\widetilde{c}_{\eps}^{B}\widetilde{v}_{\eps}^{B}=0,\qquad\widetilde{v}_{\eps}^{B}|_{t=t_{f}}=v_{0} . \label{pre-dual equation}
\end{equation}
\end{corollary}

To check that the substitution $x+\sigma B_{t}$ is correct, we should repeat step by step the proof of Proposition~\ref{Prop equivalence} in the backward case, but, since this is lengthy, let us only convince ourselves with a formal computation, similar to~\eqref{formalequiv}, which would be rigorous if $W$ (hence $B$) and $v$ were smooth:
\begin{multline*}
\Big( \partial_t \widetilde{v}_{\eps}^{B} +\widetilde{b}_{\eps}^{B}\cdot\nabla\widetilde{v}_{\eps}^{B}-\widetilde{c}_{\eps}^{B}\widetilde{v}_{\eps}^{B}\Big) (t,x) \\
  = \Big( \partial_t v_{\eps} +b_{\eps}\cdot\nabla v_{\eps}-c_{\eps}v_{\eps}\Big) (t,x+\sigma B_{t}) +\sigma\nabla v_{\eps}(t,x+\sigma B_{t}) \circ \frac{dB}{dt}=0 .
\end{multline*}

Unfortunately, equation~\eqref{pre-dual equation} is not dual to equation~\eqref{random PDE 1} (up to the fact that coefficients are smoothed) because, in the coefficients,~$x$ is translated by $W$ in~\eqref{random PDE 1} and by~$B$ in~\eqref{pre-dual equation}. If we introduce $\widetilde{v}_{\eps}(t,x) \coloneqq v_{\eps}(t,x+\sigma W_{t}) $, then we have
$\widetilde{v}_{\eps}(t,x)  =\widetilde{v}_{\eps}^{B}(t,x+\sigma W_{t_{f}})$ and therefore:

\begin{corollary}\label{SPDE PDE backward}
With the notations $\widetilde{v}_{\eps}( t,x)\coloneqq v_{\eps}(t,x+\sigma W_{t})$, $\widetilde{b}_{\eps}(t,x)\coloneqq b_{\eps}(t,x+\sigma W_{t})$, $\widetilde{c}_{\eps}(t,x)\coloneqq c_{\eps}(t,x+\sigma W_{t})$, we have for a.e.~$\omega\in\Omega$
\begin{equation}
\partial_t \widetilde{v}_{\eps} +\widetilde{b}_{\eps}\cdot\nabla\widetilde{v}_{\eps}-\widetilde{c}_{\eps}\widetilde{v}_{\eps}=0,\qquad\widetilde{v}_{\eps} (t_{f},x)
=v_{0}(x+\sigma W_{t_{f}}) . \label{dual equation}
\end{equation}
\end{corollary}

So $\widetilde{v}_{\eps}$ solves the dual equation to~\eqref{random PDE 1} (more precisely, the regularized version of the dual equation), but with a randomized final condition at time $t_{f}$. Having in mind the scheme of the previous section, we have found the operator $A^*_\eps(t)$.

\subsection{Duality formula}

The aim of this section is to prove the duality formula~\eqref{duality formula}, in order to repeat the ideas described in Section~\ref{subsection determ duality}. Notice that, by the explicit formula~\eqref{explicit formula}, smooth solutions of equation~\eqref{SPDE 1} with smooth, compactly supported initial data, and therefore also the smooth solution $v_{\eps}(\omega,t,x)$ of the backward stochastic equation~\eqref{backward SPDE 1} with smooth, compactly supported final data, are compactly supported in space, with support depending on $(t,\omega)$. The same is true for the function $\widetilde{v}_{\eps}(\omega,t,x)  \coloneqq v_{\eps}(\omega,t,x+\sigma W_{t}(\omega))$ (since we have assumed that $W_{t}$ is continuous for every~$\omega \in \Omega$). We shall write $v_{\eps}^\omega$ and $\widetilde{v}_{\eps}^\omega$ for these functions, respectively, for a given $\omega\in\Omega$.

Before going on, we need to give a meaning to equation~\eqref{formulation_randomPDE} for every~$t$, for a certain fixed (i.e.~independent of~$\varphi$) modification of $\td{u}^\omega$ (see Remark~\ref{techrmk}). To this end, we establish the next lemma, in which we denote by $\mc{B}_b$ the set of bounded Borel functions and $H^{-1}(B_R)\coloneqq (W^{1,2}_0(B_R))^*$.

\begin{lemma}\label{technicality}
Suppose that $\td{u}^\omega$ is a weak solution to equation~\eqref{random PDE 1} of class $L^{m}(L^{m}_\loc)$ according to Definition~\ref{Def weak sol random PDE}, for some fixed $\omega \in \Omega$. Then there exists a representative of $\td{u}^\omega$ \textup{(}that is, a measurable map $[0,T]\rightarrow\mc{D}'(\mathbb{R}^d)$ which coincides with $\td{u}^\omega$ up to negligible sets in $[0,T]$\textup{)}, still denoted by $\td{u}^\omega$, which belongs to $\cap_{R>0}\mc{B}_b([0,T];H^{-1}(B_R))$ and satisfies formula~\eqref{formulation_randomPDE} for every~$t \in [0,T]$.
\end{lemma}

Unless otherwise stated, we will use this representative~$\td{u}^\omega$ (and the validity of formula~\eqref{formulation_randomPDE} for every~$t \in [0,T]$).

\begin{proof}
We will omit the superscript~$\omega$ in the following. In order to construct the representative, we fix~$t \in [0,T]$ and define $F_t \colon C^\infty_c(\mathbb{R}^d) \rightarrow\mathbb{R}$ via
\begin{equation*}
\lan F_t,\varphi\ran=\lan u_0,\varphi\ran +\int^t_0\lan \td{u}(s),\td{b}(s)\cdot\nabla\varphi-\td{c}(s)\varphi\ran ds .
\end{equation*}
By our integrability assumption on~$b$ and~$c$, if~$\varphi$ has support in $B_R$, then $|\lan F_t,\varphi\ran|$ is bounded by $C_R\|\varphi\|_{W^{1,2}}$, with a constant~$C_R$ which is independent of~$t$. Therefore, for any $R>0$, $F_t$ can be extended to a linear continuous functional on $W^{1,2}_0(B_R)$, with norms uniformly bounded in~$t$.

Let us verify that $F$ is a representative of $\td{u}$. By equation~\eqref{formulation_randomPDE}, for every time-independent test function~$\varphi$ in $C^\infty_c(\mathbb{R}^d)$, there exists a full $\mc{L}^1$-measure set $A_\varphi$ in $[0,T]$ such that, for all~$t$ in $A_\varphi$, $\lan F_t,\varphi\ran$ coincides with $\lan \td{u}_t,\varphi\ran$. Hence, for a countable dense set~$\mc{D}$ in $C^\infty_c(\mathbb{R}^d)$, $F_t$ and $\td{u}_t$ must coincide for all~$t$ in $\cap_{\varphi\in\mc{D}}A_{\varphi}$, which is still a full measure set in $[0,T]$.

It remains to prove that $F$ satisfies the identity~\eqref{formulation_randomPDE} for time-dependent test functions~$\varphi$ in $C^1([0,T];C^\infty_c(\mathbb{R}^d))$. To this end we notice that, since $F$ is a representative of $\td{u}$, it must verify~\eqref{formulation_randomPDE} for a.e.~$t \in [0,T]$. Moreover, $t \mapsto \lan F_t,\varphi(t)\ran$ is continuous, which follows from the uniform (in time) bound of the $H^{-1}$ norm of $F$: in fact, we have
\begin{equation*}
|\lan F_t,\varphi(t)\ran-\lan F_s,\varphi(s)\ran|\le |\lan F_t-F_s,\varphi(t)\ran|+|\lan F_s,\varphi(t)-\varphi(s)\ran|; 
\end{equation*}
hence, when $s\rightarrow t$, then $|\lan F_t-F_s,\varphi(t)\ran| \to 0$, as a consequence of the definition of $F$, and also $|\lan F_s,\varphi(t)-\varphi(s)\ran| \to 0$, since $\varphi(s)\rightarrow\varphi(t)$ and $\sup_{s\in[0,T]}\|F_s\|_{H^{-1}(B_R)} \leq C_R$ for every $R>0$. Since the right-hand side of~\eqref{formulation_randomPDE} is continuous in time as well, we conclude that~\eqref{formulation_randomPDE} holds in fact for every $t \in [0,T]$, and 
the proof of the lemma is complete.
\end{proof}

\begin{remark}\label{weak_cont_progr}
The map $(t,\omega)\mapsto \td{u}^\omega(t)$ is actually $H^{-1}$-weakly-$*$ progressively measurable: indeed, for every test function~$\varphi$, the weak-$*$ continuity of $\td{u}(t)$ implies that the map $(t,\omega)\mapsto \lan\td{u}^\omega(t),\varphi\ran$ can be approximated by simple progressively measurable functions: we can take for instance
\begin{equation*}
\td{u}^n(t) = \sum_{j=1}^{\lfloor 2^n T \rfloor} 2^n \int_{t_{j-1} }^{t_{j}} \td{u}(s) ds \, \1_{[t_j,t_{j+1})}(t),
\end{equation*}
where $t_j = 2^{-n}j$, for $j\in\mathbb{N}$, is a dyadic partition of the interval $[0,T]$. Analogously, the $H^{-1}$-weakly-$*$ continuous version of~$u$, defined from the weakly-$*$ continuous representative of $\td{u}$ via $u(t)=\td{u}(t,\cdot-\sigma W)$, is $H^{-1}$-weakly-$*$ progressively measurable. Keep in mind that these continuous, distribution-valued versions do not need to be functions in $(t,x,\omega)$. 
\end{remark}

With this ``weakly continuous'' representative, we can now state the duality formula for approximations.

\begin{proposition}
Given $t_{f}\in [0,T]$, $v_{0}\in C_{c}^{\infty}(\mathbb{R}^{d})$ and $\varepsilon>0$, let $v_{\eps}$ be the smooth solution of the backward stochastic equation~\eqref{backward SPDE 1}. For some $\omega\in\Omega$ assume that $v_{\eps}^\omega\in C^{1}([0,T];C_{c}^{\infty}(\mathbb{R}^{d}))$ and that identity~\eqref{dual equation} holds for $\widetilde{v}_{\eps}^\omega$. If  $\widetilde{u}^\omega$ is any weak solution of equation~\eqref{random PDE 1} of class $L^{m}(L^{m}_\loc)$ corresponding to that~$\omega$, then we have 
\begin{multline}
\left\langle \widetilde{u}^\omega(t_{f}) ,v_{0}(\cdot+\sigma W_{t_{f}}(\omega))  \right\rangle
\label{duality formula} \\
 = \left\langle u_{0},\widetilde{v}_{\eps}^\omega(0) \right\rangle +\int_{0}^{t_{f}}\left\langle \widetilde{u}^\omega(s),(\widetilde{b}^\omega(s)  -\widetilde{b}_{\eps}^\omega(s))  \cdot\nabla\widetilde{v}_{\eps}^\omega(s)  - (  \widetilde{c}^\omega(s) -\widetilde{c}_{\eps}^\omega(s))  \widetilde{v}_{\eps}^\omega(s)  \right\rangle ds .
\end{multline}
\end{proposition}

\begin{proof}
This follows directly from~\eqref{formulation_randomPDE} (which can be stated for $t=t_f$ fixed, by the previous Lemma~\ref{technicality}) for $\varphi = \widetilde{v}_{\eps}^\omega$ and identity~\eqref{dual equation}.
\end{proof}

\subsection{Path-by-path uniqueness\label{subsection path by path uniq}}

By linearity of equation~\eqref{SPDE 1}, in order to prove uniqueness it is sufficient to prove that $u_{0}=0$ implies $u=0$. To this aim, we will combine identity~\eqref{duality formula} and Corollary~\ref{corollary regularity backward}, similarly to the idea explained in Section~\ref{subsection determ duality} for the deterministic case. The problem in the stochastic case, however, is that we have regularity control in average and we want to deduce path-by-path uniqueness. Let us first state the analog of~\eqref{tool deterministic}. Here, we need to impose $m>2$ (while the Definition~\ref{Def weak sol random PDE} of weak $L^m$-solutions requires only $m\geq2$). 

\begin{lemma}
\label{lemma tool stoch} 
Let $m>2$, $\beta>0$ and assume Condition~\ref{LPSreg} on~$b$ and~$c$ and Condition~\ref{LPSappprox} on the families $(b_\eps)_\eps$ and $(c_\eps)_\eps$. Given $t_{f}\in [0,T]$ and $v_{0}\in C_{c}^{\infty}(\mathbb{R}^{d})$, let $(v_{\eps})_\eps$ be the family of smooth solutions of the backward stochastic equation~\eqref{backward SPDE 1}. Then there holds
\begin{equation*}
\lim_{\eps\rightarrow0} E \bigg[ \int_{0}^{t_{f}}\int_{\mathbb{R}^{d}}(
1+\vert x \vert )^{\beta} \big( \vert (  b-b_{\eps})  \cdot\nabla v_{\eps} \vert + \vert (c-c_{\eps})  v_{\eps} \vert \big)^{m^{\prime}}dxds \bigg] =0 .
\end{equation*}
\end{lemma}

\begin{proof}
We only prove the convergence in the case $c=0$, since the terms with~$c$ are similar to or easier than those with~$b$. We will prove the assertion for every~$b$ and family $(b_\eps)_\eps$ which satisfy
\begin{equation}
\int^T_0 \Big( \int_{\mathbb{R}^d}(1+|x|)^{-\alpha}|b-b_\eps|^{\td{p}}dx \Big)^{m'/\td{p}}ds\rightarrow0\label{formula cond b}
\end{equation}
for some $\alpha\ge0$ and $\td{p}>m'$. This condition is more general than the LPS condition and includes the cases of
\begin{itemize}
 \item $b^{(1)}$ as in Condition~\ref{LPSreg} 1a), for $p<\infty$, or $b^{(1)}$ as in 1b) or in 1c), with $\alpha=0$, $\td{p}=p$;
 \item $b^{(1)}$ as in Condition~\ref{LPSreg} 1a) for $(p,q)=(\infty,2)$, with $\alpha>d$, any $\td{p}$ with $m'<\td{p}<2$: indeed, $|b-b_\eps|$ converges a.e.~in $[0,T] \times \mathbb{R}^d$ to~$0$ and $\int^T_0\int_{\mathbb{R}^d}(1+|x|)^{-\alpha}|b-b_\eps|^2dxds$ is uniformly (in~$\eps$) bounded by H\"older's inequality. Thus,~\eqref{formula cond b} follows from Vitali's convergence theorem, in this form: if $\nu$ is a finite measure (here $(1+|x|)^{-\alpha}dxds$), $f_\eps$ converges to $0$ $\nu$-a.e.~(here $|b-b_\eps|^{\td{p}}$) and $\int|f_\eps|^ad\nu$ is uniformly bounded for some $a>1$ (here $a=2/\td{p}$), then $f_\eps$ converges to $0$ in $L^1$;
 \item $b^{(2)}$ as in Condition~\ref{LPSreg} 2), with $\alpha>d+2$, $\td{p}=2$ (here we need $L^2$ integrability in time instead of $L^1$ for $b^{(2)}$, see Remark~\ref{L2time}).
\end{itemize}

Assuming~\eqref{formula cond b}, we write $\beta=(\beta+\alpha m'/\td{p})-\alpha m'/\td{p}$ and apply H\"older's inequality, first in~$x$ and~$\omega$ with exponent $\td{p}/m'>1$, then in time with exponent $1$. In this way, we find
\begin{align*}
\lefteqn{E \bigg[ \int^{t_f}_0\int_{\mathbb{R}^d}(1+|x|)^\beta|(b-b_\eps)\cdot\nabla v_\eps|^{m'}dxds \bigg] } \\
& \le \bigg(\int^{t_f}_0\Big(\int_{\mathbb{R}^d}(1+|x|)^{-\alpha}|b-b_\eps|^{\td{p}}dx\Big)^{m'/\td{p}}ds \bigg)\\
& \quad {}\times \sup_{t\in[0,t_f]} E\bigg[ \Big(\int_{\mathbb{R}^d}(1+|x|)^{(\beta\td{p}+\alpha m')/(\td{p}-m')} |\nabla v_\eps|^{m'\td{p}/(\td{p}-m')} dx \Big)^{1-m'/\td{p}} \bigg] .
\end{align*}
Now Corollary~\ref{corollary regularity backward} gives that the second term is uniformly bounded and we get the claim of the lemma in view of~\eqref{formula cond b}.
\end{proof}

For the following uniqueness statement we have to restrict the behavior at infinity of weak $L^{m}$-solutions (note that in the definition they are just $L_{\loc}^{m}(\mathbb{R}^{d})$). The restriction is not severe: we just need at most polynomial growth at infinity. To be precise, we need that for some $\alpha>0$ we have
\begin{equation}
\int_{0}^{T}\int_{\mathbb{R}^{d}}\frac{1}{1+ \vert x \vert^{\alpha}} \vert \widetilde{u}^\omega(t,x) \vert^{m}dx dt<\infty .\label{additional for uniqueness}
\end{equation}

\begin{theorem}\label{pathbypath uniq SPDE}
Let $m>2$. There exists a full measure set $\Omega_{0}\subset\Omega$ such that for all $\omega\in\Omega_{0}$ the following property holds: for every $u_{0} \colon \mathbb{R}^{d}\rightarrow\mathbb{R}$ with $\int_{\mathbb{R}^{d}} (1+ \vert x \vert^{\alpha})^{-1} \vert u_{0}(x) \vert^{m}dx<\infty$ for some
$\alpha>0$, equation~\eqref{random PDE 1} has at most one weak solution $\widetilde{u}^\omega \colon [0,T] \times \mathbb{R}^{d}\rightarrow\mathbb{R}$ of class $L^{m}(L^{m}_\loc)$ which satisfies the additional condition~\eqref{additional for uniqueness}.
\end{theorem}

\begin{proof}
\emph{Step 1: Identification of $\Omega_{0}$.} From Lemma~\ref{lemma tool stoch}, given $t_{f}\in [0,T]$ and $v_{0}\in C_{c}^{\infty}(\mathbb{R}^{d})$, there exist a full measure set $\Omega_{t_{f},v_{0}} \subset\Omega$ and a sequence $(\eps_{n})_{n \in \mathbb{N}}$ with $\eps_n \to 0$ as $n \to \infty$ such that 
\begin{gather}
\widetilde{v}_{\eps_{n}}^\omega \text{belongs to $C^{1}([0,T];C_{c}^{\infty}(\mathbb{R}^{d}))$ and satisfies~\eqref{dual equation}, for all } n \in \mathbb{N},  \label{omegawise 1} \\
\lim_{n \to \infty} \big\Vert (1+\vert \cdot \vert^{\alpha/m}) (b^\omega -b_{\eps_{n}}^\omega ) \cdot \nabla v_{\eps_{n}}^\omega
\big\Vert _{L^{m^{\prime}}([0,t_{f}] \times \mathbb{R}^{d})  }=0 \nonumber
\end{gather}
for all $\omega\in\Omega_{t_{f},v_{0}}$. Hence, we also have
\begin{equation}
\lim_{n\rightarrow\infty} \big\Vert ( 1+ \vert \cdot+\sigma W(\omega) \vert^{\alpha/m}) (\widetilde{b}^\omega  -\widetilde{b}_{\eps_{n}}^\omega ) \cdot\nabla\widetilde{v}_{\eps_{n}}^\omega  \big\Vert_{L^{m^{\prime}}([0,t_{f}] \times \mathbb{R}^{d})}=0\label{omegawise 2} 
\end{equation}
for all $\omega\in\Omega_{t_{f},v_{0}}$. 
By a diagonal procedure, given a
countable set~$\mathcal{D}$ in $C_{c}^{\infty}(\mathbb{R}^{d})$ which is dense in $L^{2}(\mathbb{R}^{d})$, there exist a full measure set $\Omega_{b}\subset\Omega$ and a sequence $(\eps_{n})_{n \in \mathbb{N}}$ with $\eps_n \to 0$ as $n \to \infty$ such that properties~\eqref{omegawise 1} and~\eqref{omegawise 2} hold for all $t_{f}\in [0,T]  \cap\mathbb{Q}$,
$v_{0}\in\mathcal{D}$ and $\omega\in\Omega_{b}$. Since, for a given~$\omega$, there exists a constant $C_\omega>0$ such that
\[
( 1+ \vert x \vert^{\alpha/m} )\leq C_{\omega} (1+ \vert x+\sigma W_{t}(\omega) \vert ^{\alpha/m} )
\]
for all $x\in\mathbb{R}^{d}$ and $t\in [0,T]  $, we may replace~\eqref{omegawise 2} by
\begin{equation}
\lim_{n\rightarrow\infty}\big\Vert ( 1+ \vert \cdot \vert
^{\alpha/m}) ( \widetilde{b}^\omega -\widetilde{b}_{\eps_{n}}^\omega)  \cdot\nabla\widetilde{v}_{\eps_{n}}^\omega \big\Vert _{L^{m^{\prime}}(
[0,t_{f}] \times \mathbb{R}^{d}) } = 0\label{omegawise 2 bis} .
\end{equation}
An analogous selection is possible for $( 1+ \vert \cdot \vert^{\alpha/m}) ( \widetilde{c}^\omega -\widetilde{c}_{\eps_{n}}^\omega ) \widetilde{v}_{\eps_{n}}^\omega$, which provides another full measure set $\Omega_c$, and $\Omega_0$ is then defined as the intersection $\Omega_b \cap \Omega_c$.

\emph{Step 2: Path-by-path uniqueness for equation~\eqref{random PDE 1} on $\Omega_{0}$.} Given $\omega\in\Omega_{0}$ and a weak solution~$\widetilde{u}^\omega$ to~\eqref{random PDE 1} of class $L^{m}(L^{m}_\loc)$ with $u_{0}=0$, by identity~\eqref{duality formula} and property~\eqref{omegawise 1} we have 
\begin{equation*}
\left\langle \widetilde{u}^\omega (t_{f}),v_{0}(\cdot+\sigma W_{t_{f}}(\omega)) \right\rangle =\int_{0}^{t_{f}}\left\langle \widetilde{u}^\omega,( \widetilde{b}^\omega  -\widetilde{b}_{\eps_{n}}^\omega ) \cdot\nabla\widetilde{v}_{\eps_{n}}^\omega  -(\widetilde{c}^\omega -\widetilde{c}_{\eps_{n}}^\omega )  \widetilde{v}_{\eps_{n}}^\omega  \right\rangle ds
\end{equation*}
for all $t_{f}\in [0,T] \cap\mathbb{Q}$, $v_{0}\in\mathcal{D}$ and $n\in\mathbb{N}$. Now we pass to the limit. By H\"older's inequality we get
\begin{align*}
\lefteqn{\big\vert \left\langle \widetilde{u}^\omega (t_{f})
,v_{0}(\cdot+\sigma W_{t_{f}}(\omega))\right\rangle \big\vert } \\
&  \leq \Big( \int_{0}^{t_{f}}\int_{\mathbb{R}^{d}}\frac{1}{(1+\vert x \vert^{\alpha/m} )^{m}} \vert \widetilde{u}^\omega(s,x) \vert^{m}dxds \Big)  ^{1/m}\\
& \quad {} \times \Big( \int_{0}^{t_{f}} \int_{\mathbb{R}^{d}} (1+\vert x\vert^{\alpha/m})^{m^{\prime}} \big\vert (\widetilde{b}^\omega  -\widetilde{b}_{\eps_{n}}^\omega ) \cdot \nabla\widetilde{v}_{\eps_{n}}^\omega  -(  \widetilde{c}^\omega -\widetilde{c}_{\eps_{n}}^\omega )  \widetilde{v}_{\eps_{n}}^\omega \big\vert^{m^{\prime}}dxds \Big)^{1/m^{\prime}}
\end{align*}
and thus, we have $\left\langle \widetilde{u}^\omega(t_{f}),v_{0}(\cdot+\sigma W_{t_{f}}(\omega))\right\rangle =0$ by~\eqref{omegawise 2 bis} (and the analogous identity for $( 1+ \vert \cdot \vert^{\alpha/m}) ( \widetilde{c}^\omega -\widetilde{c}_{\eps_{n}}^\omega ) \widetilde{v}_{\eps_{n}}^\omega$). This is equivalent to $\left\langle \widetilde{u}^\omega(t_{f},\cdot-\sigma W_{t_{f}}(\omega) )  ,v_{0}\right\rangle =0$, which implies $\widetilde{u}^\omega(t_{f},\cdot-\sigma W_{t_{f}}(\omega))  =0$ by the density of $\mathcal{D}$ in $L^{2}(\mathbb{R}^{d})$ and thus $\widetilde{u}^\omega(t_{f},\cdot)  =0$. This holds true for every $t_{f}\in [0,T] \cap\mathbb{Q}$; since $t\mapsto \widetilde{u}^\omega(t)$ is continuous in the sense of distributions, we get $\widetilde{u}^\omega(t,\cdot) =0$ for every $t \in [0,T]$. The proof of the theorem is complete.
\end{proof}

\subsection{Existence for~\eqref{SPDE 2}}

So far we have proved that path-by-path uniqueness holds for the stochastic equation~\eqref{SPDE 2}. It remains to prove the existence of a (distributional) solution. The proof is based on a priori estimates and is somehow similar to that of Theorem~\ref{theorem a priori estimate} and Theorem~\ref{Theo existence regular sol}, without the difficulty of taking derivatives. Thus, we will state the result and only sketch the proof.

\begin{proposition}
Let $p$, $q$ be in $(2,\infty)$ satisfying $\frac2q+\frac{d}{p}\le1$ or $(p,q)=(\infty,2)$. Assume that~$b$ and~$c$ are a vector field and a scalar field, respectively, such that $b=b^{(1)}+b^{(2)}$, $c=c^{(1)}+c^{(2)}$, with $b^{(i)}$, $c^{(i)}$ in $C^\infty_c([0,T]\times\mathbb{R}^d)$ for $i=1,2$. Let $\chi$ be a function satisfying~\eqref{property of phi}. Then there exists a constant $C$ such that, for every $u_0$ in $C^\infty_c(\mathbb{R}^d)$, the smooth solution~$u$ of equation~\eqref{SPDE 2} starting from $u_0$, given by Lemma~\ref{lemma preliminare sul caso smooth}, verifies
\begin{equation*}
\sup_{t\in [0,T]  } \int_{\mathbb{R}^{d}}E\left[(u(t,x))^{m}\right]^{2}\chi(x) dx\leq C\left\Vert u_{0}\right\Vert_{L_{\chi}^{2m}(\mathbb{R}^{d})}^{2m}.
\end{equation*}
Moreover, the constant $C$ can be chosen to have continuous dependence on $m,d,\sigma,\chi,p,q$ and on the $L^q([0,T];L^p(\mathbb{R}^d))$ norms of $b^{(1)}$ and $c^{(1)}$, on the $L^1([0,T];C^1_{\lin}(\mathbb{R}^d))$ norm of~$b^{(2)}$, and on the $L^1([0,T];C^1_b(\mathbb{R}^d))$ norm of $c^{(2)}$.

The result holds also for $(p,q)=(d,\infty)$ under the additional hypothesis that the $L^\infty([0,T];L^d(\mathbb{R}^d))$ norms of $b^{(1)}$ and $c^{(1)}$ are smaller than $\delta$, see Condition~\ref{LPSreg}, \textup{1c)} \textup{(}in this case the continuous dependence of $C$ on these norms is up to $\delta$\textup{)}.
\end{proposition}

\begin{proof}
We proceed similarly as in the the proof of Theorem~\ref{theorem a priori estimate}, but aiming for a priori estimates for~$u$ and not for its derivatives. To this end, we consider the equation for~$E[u^m]$, which is a parabolic closed equation. The same method of proof as in Theorem~\ref{theorem a priori estimate} can then be applied (without the difficulty of having a system with many indices), which then shows the claim..
\end{proof}

\begin{theorem}\label{existence CE}
Let $m \geq 4$ be an even integer and let $s$ be a real number. Assume that $b$, $c$ satisfy Condition~\ref{LPSreg} and let $u_0 \in L^{2m}_{(1+|\cdot|)^{2s+d+1}}(\mathbb{R}^d)$. There exists a weak solution~$u$ to equation~\eqref{SPDE 2} of class $L^{m}(L^{m}_\loc)$. Moreover, there holds
\begin{equation*}
\esssup_{t\in[0,T]}E\left[\|u(t,\cdot)\|_{L^m_{(1+|\cdot|)^s}(\mathbb{R}^d)}^m\right] <\infty .
\end{equation*}
Finally, pathwise uniqueness holds among such solutions and actually among all solutions such that $\td{u}$ satisfies~\eqref{additional for uniqueness} a.s..
\end{theorem}

\begin{proof}
The existence of a weak solution to equation~\eqref{SPDE 2} of class $L^{m}(L^{m}_\loc)$ follows by the same arguments as in the proof of Theorem~\ref{Theo existence regular sol}. The main differences are that weak-$*$ convergence holds in $L^\infty([0,T];L^m(\Omega;L^m(B_R)))$ instead of $L^\infty([0,T];L^m(\Omega;W^{1,m}(B_R)))$ and that all the derivatives must be carried over to the test function~$\varphi$.

Pathwise uniqueness follows from Theorem~\ref{pathbypath uniq SPDE}: Let $u$, $u_1$ be two solutions to~\eqref{SPDE 2} satisfying~\eqref{additional for uniqueness} on the same filtered probability space $(\Omega,(\mc{G}_t)_t,P)$, such that $W$ is a Brownian motion with respect to $(\mc{G}_t)_t$. Then, according to Lemma~\ref{Prop equivalence}, the function~$\td{u}_1$ given by $\td{u}_1(t,x) =u_1(t,x+\sigma W_t)$ solves the deterministic PDE~\eqref{random PDE 1} for a.e.~$\omega \in \Omega$, so $\td{u}_1$ must coincide with~$\td{u}$ for a.e.~$(t,x,\omega)$, which implies the claim $u_1=u$. 
\end{proof}

\begin{remark}\label{positivity}
For solutions to equation~\eqref{SPDE 2}, non-negativity of initial values is preserved, i.e.~if $u_0\ge0$, then $u(t,x,\omega)\ge0$ for a.e.~$(t,x,\omega)$: this is true in the regular case \textup{(}for $u$, $b$ and~$c$ smooth and compactly supported\textup{)}, thanks to the representation formula~\eqref{explicit formula} \textup{(}where, for the application to equation~\eqref{SPDE 2}, $c$ is replaced by $c+ \diverg b$\textup{)}. This carries over to the general case, since $u$ is constructed as weak-$*$ limit in $L^\infty([0,T];L^m(\Omega;W^{1,m}(B_R)))$ of solutions with regularized coefficients and initial condition.
\end{remark}

\section{Path-by-path uniqueness and regularity of the flow solving the sDE}\label{section flow}

In this section we want to apply the previous results to study the stochastic differential equation~\eqref{SDE}. We will get existence, strong (even path-by-path) uniqueness and regularity for the stochastic flow solving the sDE, where~$b$ is in the LPS class and $\sigma\neq0$. Once again, we recall that no such result holds in the deterministic case ($\sigma=0$), which means that for the stochastic case ($\sigma \neq 0$) the evolution of the finite-dimensional system gets better due to the additional stochastic term.

In order to state the result, we need to make the formal links between~\eqref{SDE} and~\eqref{stoch cont} precise. This will be done for the deterministic case, in the first subsection, using Ambrosio's theory of Lagrangian flows. Then we will use this link (read in a proper way in the stochastic case) combined with uniqueness and regularity for the stochastic equations to arrive at our result. 

\subsection{The deterministic case}\label{deterministic case subsection}

Consider the ODE
\begin{equation}
\frac{d}{dt}X =f(t,X) \label{ODEf}
\end{equation}
on $\mathbb{R}^d$. If $f$ is a regular field (e.g.~$C^1_c([0,T]\times\mathbb{R}^d)$), there exists a unique flow~$\Phi \colon [0,T] \to \mathrm{Diff}(\mathbb{R}^d)$ of diffeomorphisms on~$\mathbb{R}^d$ solving the ODE, i.e.~for every~$x$ in $\mathbb{R}^d$, $t\mapsto\Phi(t,x)$ is of class $C^1$ and solves the ODE starting from $\Phi(0,x)=x$. If~$\varphi$ is a test function in $C^\infty_c(\mathbb{R}^d)$, then the chain rule gives the following equation for $\varphi(\Phi)$:
\begin{equation*}
\frac{d}{dt}\varphi(\Phi_t)=\nabla\varphi(\Phi_t)\cdot f(t,\Phi_t) .
\end{equation*}
If we integrate this equation with respect to a finite signed measure $\mu_0$ on $\mathbb{R}^d$, we get
\begin{equation}\label{CEdistrib}
\lan\mu_t,\varphi\ran=\lan\mu_0,\varphi\ran+\int^t_0\lan\mu_s,f(s,\cdot)\cdot\nabla\varphi\ran ds ,
\end{equation}
where $\mu_t=(\Phi_t)_\#\mu_0$ is the image measure on $\mathbb{R}^d$ of $\mu_0$ under $\Phi_t$, i.e.\ $\int gd\mu_t=\int g(\Phi_t)d\mu_0$ for every measurable bounded function $g$ on $\mathbb{R}^d$. Equation~\eqref{CEdistrib} is the continuity equation (CE) for~$\mu$ (starting from~$\mu_0$), which we have written so far in compact form as
\begin{equation}\label{CE}\tag{CE}
\partial_t\mu+\diverg(f\mu)=0 .
\end{equation}

It is easy to see that the previous passages still hold when $f$ is not regular. Starting from this remark, DiPerna--Lions' and Ambrosio's theory extends the above link between the ODE and the CE (in some generalized sense) to the irregular case, so that one can study the CE in order to study the ODE. We will follow Ambrosio's theory of Lagrangian flows, which allows to transfer a well-posedness result for the CE to a well-posedness result for the ODE.

In the general theory, one considers a convex set~$\mc{L}_f$ of solutions $\mu=(\mu_t)_t$ to the equation~\eqref{CE}, with values in the set $\mc{M}_+(\mathbb{R}^d)$ of finite positive measures on $\mathbb{R}^d$, which satisfies $\int^T_0\int_{\mathbb{R}^d} (1+|x|)^{-1} |f(t,x)| \mu_t(dx)dt< \infty$ for every~$\mu$ in~$\mc{L}_f$ and
\begin{equation}
0\le\mu'_t\le\mu_t,\ \mu\in\mc{L}_f,\ \mu'\mbox{ solves~\eqref{CE} }\Rightarrow\mu'\in\mc{L}_f\label{condLb}
\end{equation}
(``solution of~\eqref{CE}'' is here intended in the sense of distributions). For our purposes,~$\mc{L}_f$ will be, for some $m$ fixed a priori, the set
\begin{multline*}
\mc{L}_f=\bigg\{\mu=(\mu_t)_t \colon \mu\mbox{ solves~\eqref{CE}}, \mu_t=u_t\mc{L}^d \text{ for some non-negative } u\in L^1\cap L^m([0,T]\times\mathbb{R}^d), \\ \int^T_0\int_{\mathbb{R}^d}\frac{|f(t,x)|}{1+|x|}\mu_t(dx)dt< \infty \bigg\} .
\end{multline*}
Sometimes, we will use~$u$ to indicate also~$\mu$ and vice versa.

\begin{definition}
A~\emph{$\mc{L}_f$ \textup{(}local\textup{)} Lagrangian flow, starting from some fixed \textup{(}non-negligible\textup{)} Borel set~$S$ in $\mathbb{R}^d$}, is a Borel map $\Phi \colon [0,T]\times\mathbb{R}^d\rightarrow\mathbb{R}^d$ such that
\begin{itemize}
\item for $\mc{L}^d$-a.e.~$x$ in~$S$, for every~$t$, $\Phi_t(x)=x+\int^t_0f(s,\Phi_s(x))ds$;
\item $\mu_t\coloneqq (\Phi_t)_\#( \1_{S} \mc{L}^d)$ is in~$\mc{L}_f$.
\end{itemize}
A~\emph{$\mc{L}_f$ global Lagrangian flow} is a Borel map $[0,T]\times\mathbb{R}^d\rightarrow\mathbb{R}^d$ which is a local Lagrangian flow from every \textup{(}non-negligible\textup{)} Borel set~$S$.
\end{definition}

\begin{theorem}\label{Lagrthm}
Suppose that uniqueness holds for~\eqref{CE}, starting from every $ \1_S \mc{L}^d$, in the class~$\mc{L}_f$. Then local uniqueness \textup{(}i.e.~uniqueness from every~$S$\textup{)} holds among~$\mc{L}_f$ Lagrangian flows \textup{(}that is, if $\Phi^1$ and $\Phi^2$ are two such flows, then, for a.e.~$x$ in~$S$: for every~$t$, $\Phi_t^1(x)=\Phi_t^2(x)$\textup{)}. If, in addition, existence holds for~\eqref{CE} in the class~$\mc{L}_f$ starting from $ \1_{B_N} \mc{L}^d$, for every positive integer $N$, then there exists a global Lagrangian flow.
\end{theorem}

This theorem is stated and proved in~\cite[Theorem 18]{AMBCRI08}. We give here a concise proof (similar to the one in~\cite{AMBCRI08}) only of uniqueness since the existence part is more technical and long (though not difficult). The idea for uniqueness is again to use the link between the ODE and the CE: whenever one has two flows $\Phi^1$ and $\Phi^2$, then $(\Phi^1_t)_\# \1_S \mc{L}^d$ and $(\Phi^2_t)_\# \1_S \mc{L}^d$ are solutions to~\eqref{CE} in the class~$\mc{L}_f$, so, by uniqueness, they must coincide, so that $\Phi^1$ and $\Phi^2$ coincide on~$S$.

\begin{proof}
By continuity in time of the Lagrangian flow (for a.e.~fixed~$x$), it is enough to show that, given two~$\mc{L}_f$ Lagrangian flows $\Phi^1$, $\Phi^2$ starting from the same (non-negligible) set~$S$, then, for every~$t \in [0,T]$, we have $\Phi^1_t(x)=\Phi^2_t(x)$ for a.e.~$x \in S$. Suppose by contradiction that this is not the case. Then there exist a time $t \in [0,T]$, two disjoint Borel sets $E^1$, $E^2$ in~$\mathbb{R}^d$ and a Borel set~$S'$ in~$S$ with $0<\mc{L}^d(S')<\infty$ such that $\Phi_t^i(x)$ is in~$E^i$ for every~$x \in S'$, for $i=1,2$. Define $\mu^i_t \coloneqq (\Phi^i_t)_\# \1_{S'} \mc{L}^d$ for $i=1,2$. Then $\mu^1$ and $\mu^2$ are maps from $[0,T]$ to $\mc{M}_+(\mathbb{R}^d)$, which are weakly continuous solutions to~\eqref{CE}, still in the class~$\mc{L}_f$ (as they are restrictions of $(\Phi^1_t)_\# \1_S \mc{L}^d$ and $(\Phi^2_t)_\# \ \1_S \mc{L}^d$), and differ at least in one point~$t$. This is a contradiction and uniqueness is proved.
\end{proof}

\begin{remark}\label{lagr_repr}
If existence and uniqueness \textup{(}starting from every $ \1_S \mc{L}^d$ in the class~$\mc{L}_f$\textup{)} hold for~\eqref{CE} and if $f$ is in $L^2_{\loc}([0,T]\times\mathbb{R}^d)$, then for every~$t \in [0,T]$ there holds $\lan u_0,\varphi(\Phi_t)\ran = \lan u_t,\varphi\ran$, where we have denoted by~$\Phi$ the Lagrangian flow and by~$u$ the $H^{-1}$-weakly-$*$ continuous version of the solution to~\eqref{CE}. Indeed, the map $t\mapsto (\Phi_t)_\#u_0$ is also a $H^{-1}$-weakly-$*$ continuous solution to~\eqref{CE}, hence, it must coincide with~$u$ at all times.
\end{remark}

Having Theorem~\ref{Lagrthm} at our disposal, we can employ existence and uniqueness for the CE in the class~$\mc{L}_f$ in order to prove well-posedness for the Lagrangian flow associated with the ODE. This is what DiPerna--Lions and Ambrosio have done (using mainly the transport equation instead of the continuity equation and with a different~$\mc{L}_f$) for weakly differentiable functions~$f$ (see \cite{DIPELI89,AMBROSIO04}). We will follow this strategy, but for~$f$ in LPS class and with noise, using our well-posedness result for~\eqref{stoch cont}. 

\subsection{Stochastic Lagrangian flow: existence, uniqueness and regularity}

We consider the equation~\eqref{SDE} on $\mathbb{R}^d$. Since we use also here the results of the previous sections, we again assume the same LPS Condition~\ref{LPSreg} on the drift~$b$. As before, we consider the purely stochastic case $\sigma\neq0$, and $W$ is a standard $d$-dimensional Brownian motion, endowed with its natural completed filtration $(\mc{F}_t)_t$ (the smallest among all the possible filtrations), which is also right-continuous (see \cite[Proposition 2.5]{BASS11}).

With the change of variable $\td{X}_t=X_t-\sigma W_t$, this sDE becomes a family of (random) ODEs, parametrized by~$\omega\in \Omega$:
\begin{equation}
\frac{d}{dt}\td{X} =\td{b}^{\omega}(t,\td{X} ),\label{ODE}
\end{equation}
where, as usual, $\td{b}^\omega(t,x)=b(t,x+\sigma W_t(\omega))$. More precisely, if~$X$ is a progressively measurable process, then~$X$ solves~\eqref{SDE} if and only if $\td{X}$ solves the ODE~\eqref{ODE} for a.e.~$\omega$. For this family of ODEs, the concepts of Lagrangian flow and CE (at~$\omega$ fixed) make sense and the CE associated with this ODE is precisely the random PDE~\eqref{random PDE 1} with $c=0$, that is the random CE
\begin{equation}
\partial_t\td{u}+\diverg(\td{b}\td{u})=0 .\label{random CE}
\end{equation}
Thus, we can hope to apply our existence and uniqueness result for~\eqref{stoch cont} (remembering that, by Lemma~\ref{Prop equivalence}, a solution to the random CE~\eqref{random CE} is given by $\td{u}(t,x)=u(t,x+\sigma W_t)$, when~$u$ solves~\eqref{stoch cont}).

\begin{definition}
A \emph{stochastic \textup{(}global\textup{)} Lagrangian flow solving the equation~\eqref{SDE}} is a measurable map $\Phi \colon [0,T]\times\mathbb{R}^d\times\Omega\rightarrow\mathbb{R}^d$ with the following properties:  
\begin{itemize}
\item for a.e.~$\omega \in \Omega$, $(t,x) \mapsto \td{\Phi}^\omega_t(x)\coloneqq \Phi_t^\omega(x) -\sigma W_t(\omega) \coloneqq \Phi(t,x,\omega)-\sigma W_t(\omega)$ is a $\mc{L}_{\td{b}^\omega}$ \textup{(}global\textup{)} Lagrangian flow \textup{(}solving the ODE~\eqref{ODE} with that~$\omega$ fixed\textup{)};
\item $\Phi$ is progressively measurable, i.e.~it is $\mc{P}\otimes\mc{B}(\mathbb{R}^d)$-measurable, where $\mc{P}$ is the progressive $\sigma$-algebra.
\end{itemize}
Given a certain class $A$ of functions from $\mathbb{R}^d$ to $\mathbb{R}^d$, e.g.~$W^{1,m}_{\loc}(\mathbb{R}^d)$, the flow is said to be of class~$A$ if, for every $t \in [0,T]$, $\Phi_t$ is in class~$A$ with probability one.
\end{definition}

Let us now state the main result of this section:

\begin{theorem}\label{main thm flows}
Let $m \geq 4$ be an even integer and assume that~$b$ verifies Condition~\ref{LPSreg}. Then  
\begin{enumerate}[font=\normalfont]
\item local path-by-path uniqueness holds among Lagrangian flows solving~\eqref{SDE}, i.e., for a.e.~$\omega \in \Omega$, local uniqueness holds among $\mc{L}_{\td{b}^\omega}$ Lagrangian flows solving the ODE~\eqref{ODE} with that~$\omega$ fixed; 
\item there exists a global stochastic Lagrangian flow solving~\eqref{SDE};
\item this flow is of class $W^{1,m}_{\loc}(\mathbb{R}^d)$.
\end{enumerate}
\end{theorem}

Before proceeding to the proof, which is essentially an application of our well-posedness result for the deterministic PDE~\eqref{random PDE 1}, we make some comments on this result. 

\begin{remark}
If $m>d$, we deduce by Sobolev embedding that, for every~$t \in [0,T]$, there exists a representative of $\Phi_t$ which is of class $C^{0,\alpha}_{\loc}(\mathbb{R}^d)$ for $\alpha=1-d/m$. However, we are not able to show that this representative is jointly continuous in $(t,x)$ \textup{(}we actually do not even show joint measurability\textup{)}, though such joint continuity is known to be true in the subcritical case \textup{(}see~\textup{\cite{FEDFLA11}}\textup{)}.
\end{remark}

\begin{remark}
The existence part gives essentially a family a flows $\Phi^\omega$, parametrized by~$\omega \in \Omega$, such that $\Phi(x)$ solves~\eqref{SDE} for a.e.~$x$, while the regularity part gives local weak differentiability of the flow. The uniqueness part implies pathwise uniqueness among stochastic Lagrangian flows: given two stochastic Lagrangian flows $\Phi^1$ and $\Phi^2$ solving~\eqref{SDE} \textup{(}even adapted to some filtration larger than $(\mc{F}_t)_t$\textup{)} and starting from the same initial datum of the form $ \1_{S} \mc{L}^d$, they necessarily coincide. Indeed, for a.e.~$\omega \in \Omega$, $\td{\Phi}^1(\omega)$ and $\td{\Phi}^2(\omega)$ are Lagrangian flows solving the random ODE~\eqref{ODE} \textup{(}with that~$\omega$\textup{)}, so $\Phi^1=\Phi^2$ a.e.. 

Let us emphasize that path-by-path uniqueness is stronger than pathwise uniqueness: it says that, for each fixed~$\omega$ 
in a full $P$-measure set, any two Lagrangian flows, solving~\eqref{SDE} \textup{(}interpreted as the random ODE~\eqref{ODE}\textup{)} with~$\omega$ fixed, must coincide, without any need to have adapted flows. On the other hands, while we can manage flows, we are not able to compare two solutions to the ODE~\eqref{ODE}, at~$\omega$ fixed, starting from a fixed~$x$, so we have no uniqueness result for~\eqref{SDE} with~$x$ as initial datum. Let us remind, however, that pathwise uniqueness holds for~\eqref{SDE} \textup{(}with~$x$ fixed\textup{)} under Krylov--R\"ockner conditions, see~\textup{\cite{KRYROE05}}.
\end{remark}

We also wish to recall a basic argument in measure theory, that we will use quite often:

\begin{remark}\label{simple_rmk}
Let $(E,\mathcal{E},\mu)$, $(F,\mathcal{F},\nu)$ be two $\sigma$-finite measure spaces and let $f \colon E\times F\rightarrow \mathbb{R}$ be a map such that, for $\nu$-a.e.~$z \in F$, the map $y \mapsto f(y,z)$ is $\mathcal{E}$-measurable. Assume that~$f$ has a $\mathcal{E} \otimes \mathcal{F}$-measurable version $g \colon E\times F\rightarrow \mathbb{R}$, i.e.~there exist a full measure set $F_0$ and, for every $z \in F_0$, a full measure set $E_0^z$, such that $f(y,z)=g(y,z)$ for all $z \in F_0$ and $y \in E_0^z$. Let $BP=BP(a)$ be a Borel property defined for $a\in\mathbb{R}$ \textup{(}in the sense that the subset where $BP$ is true is a Borel set\textup{)}, for example $\varphi(a)=0$ for some Borel function~$\varphi$. Assume that, for $\nu$-a.e.~$z \in F$, it holds: $BP(f(y,z))$ for $\mu$-a.e.~$y \in E$. Then $BP(g(y,z))$ holds for $(\mu \times \nu)$-a.e.~$(y,z) \in E \times F$. A similar property also holds for more than two variables.
\end{remark}

\begin{proof}
If this were not true, then the set $A=\{(y,z) \in E \times F \colon \neg BP(g(y,z))\}$ is $\mathcal{E} \otimes \mathcal{F}$-measurable (by measurability of~$BP$ and~$g$) and of positive measure. Therefore, by Fubini's theorem, there exists a positive measure set $F_{\neg} \subset F$ such that, for every $z \in F_{\neg}$, the set $E_\neg^{z} \coloneqq \{y \in E\colon \neg BP(g(y,z))\}$ is $\mathcal{E}$-measurable and of positive measure. But~$g$ is by assumption a version of~$f$. Therefore, for every~$z$ in the positive measure set $F_{\neg}\cap F_0$, the set $\{y \in E \colon\neg BP(f(y,z))\}$ contains the positive measure set $E_\neg^z\cap E_0^z$, which is in contradiction with the assumption on $BP(f)$.
\end{proof}

\begin{proof}[Proof of Theorem~\ref{main thm flows}]
\emph{Part 1: Uniqueness of Lagrangian flows solving the ODE~\eqref{ODE}.} Theorem~\ref{pathbypath uniq SPDE}, applied to the random CE~\eqref{random CE}, gives a full $P$-measure set $\Omega_0$ in~$\Omega$ such that, for every $\omega \in \Omega_0$, for every Borel set~$S$, there exists at most one solution~$\td{u}^\omega$ to the CE in the class $\mc{L}_{\td{b}^\omega}$, which starts from~$\1_S$ (note that~$\Omega_0$ is independent of the initial datum). Thus, for every~$\omega \in \Omega_0$, the first part of Theorem~\ref{Lagrthm} gives local uniqueness among Lagrangian flows solving~\eqref{ODE} at~$\omega$ fixed. 

\emph{Part 2: Existence of a global stochastic Lagrangian flow.} The idea is to proceed in three steps and use Ambrosio's theory for the random ODE~\eqref{ODE} to get the existence, at~$\omega$ fixed, of a Lagrangian flow, then to use the progressive measurability of the solution to~\eqref{stoch cont} to show progressive measurability of (a version of) the Lagrangian flow and to conclude. As we are going to take various modifications of the same function, we keep the following convention: we use the notation $\bar{\Phi}$ for a solution of the sDE which is continuous in time (at~$x$ fixed), but not necessarily measurable in~$\omega$, the notation $\td{\Phi}$ for a solution of the random ODE and the notation $\td{\bar{\Phi}}$ for a solution of the random ODE which is also continuous in time (again at~$x$ fixed). For the solution to the (s)CE, we do not use the ``bar'' since we consider, unless otherwise stated, versions that are both weakly-$*$ measurable and weakly-$*$ continuous.

In the \emph{first step}, we get the existence, at~$\omega$ fixed, of a Lagrangian flow $\td{\bar{\Phi}}^\omega$ solving the ODE~\eqref{ODE}. We take $S=B_N$ for an arbitrary positive integer~$N$. By Theorem~\ref{existence CE} (applied with $c=0$), Remark \ref{positivity} and Proposition~\ref{Prop equivalence}, we find a full $P$-measure set $\Omega_0$ in~$\Omega$, independently of~$N$ (by a diagonal procedure), such that, for every~$\omega \in \Omega_0$ and~$N$, there exists a (unique) solution $\td{u}^{\omega,N}$ to the CE~\eqref{random CE} in the class $\mc{L}_{\td{b}^\omega}$, starting from $ \1_{B_N}$. Thus, the second part of Theorem~\ref{Lagrthm} gives the claimed existence of a global Lagrangian flow $\td{\bar{\Phi}}^\omega$ solving the ODE~\eqref{ODE}.

Now we define $\bar{\Phi}=\td{\bar{\Phi}}+\sigma W$, which seems at first the natural candidate for the stochastic Lagrangian flow solution to~\eqref{SDE}. The main problem is that $\bar{\Phi}$ does not have any measurability property in~$\omega$. Therefore, in the \emph{second step}, we find a progressively measurable map $\Phi \colon [0,T]\times\mathbb{R}^d\times\Omega\rightarrow \mathbb{R}^d$ version on $\bar{\Phi}$, that is $P(\bar{\Phi}(t,x,\omega)=\Phi(t,x,\omega)\text{ for a.e. }(t,x))=1$ (keep in mind that this set is not a priori measurable in~$\omega$). To this end, we shall use the link between ODE and CE (at the deterministic level) and the progressive measurability of the solution to~\eqref{stoch cont}.

In what follows, we denote by~$\varphi_n$ functions in $C^\infty_c(\mathbb{R}^d)$ with $\varphi_n(x)=x$ for $|x|\le n$. By the deterministic theory (Theorem~\ref{Lagrthm} and Remark~\ref{lagr_repr}), we know that, for every~$n \in \mathbb{N}$ and $u_0 \in C^\infty_c(\mathbb{R}^d)$, we have for every~$\omega$ in a full measure set $\Omega_{u_0,n}$: for every~$t \in [0,T]$ there holds $\lan u_0,\varphi_n(\td{\bar{\Phi}}_t^\omega)\ran = \lan \td{u}^\omega(t),\varphi_n\ran$ and so $\lan u_0,\varphi_n(\bar{\Phi}_t^\omega)\ran = \lan u^\omega(t),\varphi_n\ran$, where $u^\omega(t)$ is the $H^{-1}$-weakly-$*$ continuous version, as in Remark~\ref{weak_cont_progr}, of the solution to~\eqref{stoch cont} starting from $u_0$. In particular, the map $(t,\omega)\mapsto \lan u_0,\varphi_n(\bar{\Phi}_t^\omega)\ran$ coincides with a progressively measurable map for every $t \in [0,T]$, for a.e.~$\omega$ (with the exceptional set independent of~$t$) for every $u_0 \in C^\infty_c(\mathbb{R}^d)$. Hence, up to redefining $\bar{\Phi}_t^\omega$ on a $P$-null set independent of~$t$, $\lan u_0,\varphi_n(\bar{\Phi}_t^\omega)\ran$ is progressively measurable for every $\varphi$ in $C^\infty_c(\mathbb{R}^d)$ and thus, by density, also for every $u_0 \in L^2(B_R)$. Therefore, $\varphi_n(\bar{\Phi}_t^\omega)$ is weakly-$*$ progressively measurable in $L^2(B_R)$. Since $L^2(B_R)$ is a separable reflexive space, Pettis measurability theorem applies and gives that $\varphi_n(\bar{\Phi}_t^\omega)$ is strongly progressively measurable with values in $L^2(B_R)$; in particular, there exists $\Phi_n \colon [0,T]\times B_R\times \Omega\rightarrow\mathbb{R}^d$, $\mc{P}\otimes\mc{B}(\mathbb{R}^d)$-measurable, version of $\varphi_n(\bar{\Phi})$, that is, for a.e.~$(t,\omega)$ there holds $\Phi_n(t,x,\omega)=\varphi_n(\bar{\Phi}_t^\omega(x))$ for a.e.~$x\in B_R$ (cf.~\cite[Proposition~A.6]{MAURELLIPHD}). Using Remark~\ref{simple_rmk} and the analogous properties for~$\bar{\Phi}$, one can check that~$\Phi_n$ does not depend on~$R$ and is definitively constant in~$n$ (for a.e.~$(t,x,\omega)$), so we get~$\Phi$, which is $\mc{P}\times\mc{B}(\mathbb{R}^d)$-measurable and a version of $\bar{\Phi}$. The second step is complete.

To conclude the proof of existence, we have to prove that $\Phi -\sigma W$ is a (global) Lagrangian flow solving~\eqref{ODE}. However, since~$\Phi$ coincides with $\bar{\Phi}$ only for a.e. $(t,x)$ (for fixed~$\omega$), $\Phi(\cdot,x,\omega) -\sigma W(\omega)$ does not need to be continuous in time and satisfies~\eqref{ODE} only for a.e.~$t \in [0,T]$. In the \emph{third step}, we prove that there exists a measurable version of~$\Phi$, and so of $\bar{\Phi}$, which is continuous in~$t$ for a.e.~$(x,\omega)$, and use this version to conclude. The conceptual idea is that, given a path $\gamma$ which has a continuous version, its continuous version can be constructed from $\gamma$ in a measurable way, so that this version is measurable (with respect to some other variable) if $\gamma$ is measurable.

For any $N \in \mathbb{N}$, we choose a dyadic partition $t^N_j=2^{-N}j$, we set $I^N_j \coloneqq [t^N_j,t^N_{j+2})$ and, for~$t \in [0,T]$, we define $I^N(t)$ as $I^N_j$ for the minimal~$j$ with~$t \in I^N_j$. Note that, for a.e.~$\omega$, it holds: for a.e.~$(t,x)$, $\Phi(t,x,\omega)=\bar{\Phi}(t,x,\omega)$ (as $\Phi$ is a modification of~$\bar{\Phi}$ and both are measurable in $(t,x)$ for~$\omega$ fixed). In particular, for a.e.~$\omega$, it holds: for a.e.~$x$,
\begin{align*}
\max_j \big[\esssup_{t\in I^N_j}\Phi(t,x,\omega) - \essinf_{t\in I^N_j}\Phi(t,x,\omega)\big] = \max_j\big[\esssup_{t\in I^N_j}\bar{\Phi}(t,x,\omega) - \essinf_{t\in I^N_j}\bar{\Phi}(t,x,\omega)\big].
\end{align*}
The continuity property of $\bar{\Phi}$ implies that for a.e.~$\omega$ the following is true: for a.e.~$x$ and every $m \in \mathbb{N}$, there exists~$N\in \mathbb{N}$ with $\max_j[\esssup_{t\in I^N_j}\bar{\Phi}(t,x,\omega) - \essinf_{t\in I^N_j}\bar{\Phi}(t,x,\omega)]<1/m$. Therefore, by Remark~\ref{simple_rmk}, the set
\begin{align*}
\bigcap_{m \in \mathbb{N}} \bigcup_{N \in \mathbb{N}} \big\{(x,\omega) \colon \max_j \big[\esssup_{t\in I^N_j}\Phi(t,x,\omega) - \essinf_{t\in I^N_j}\Phi(t,x,\omega)\big] < 1/m \big\}
\end{align*}
has full measure. Then, for a.e.~$(x,\omega)$, the limit
\begin{align*}
A(t,x,\omega)=\lim_{N \to \infty} \esssup_{s\in I^N(t)}\Phi(s,x,\omega)\end{align*}
is well-defined and finite for every~$t$. Moreover, the map~$A$ (defined zero on the exceptional set where the above limit does not exist) is measurable in $(t,x,\omega)$, and continuous in~$t$ for a.e.~$(x,\omega)$. For a.e.~$\omega$ it holds: $A(t,x,\omega)=\bar{\Phi}(t,x,\omega)$ for a.e.~$(t,x)$ (since $\bar{\Phi}(t,x,\omega)=\lim_{N \to \infty} \esssup_{s\in I^N(t)}\bar{\Phi}(s,x,\omega)$ for a.e.~$(t,x)$). So, again by Remark~\ref{simple_rmk}, $A=\Phi$ for a.e.~$(t,x,\omega)$. With a little abuse of notation, we will use now~$\Phi$ also for its modification which is continuous in~$t$.

It remains to show that $\td{\Phi} = \Phi-\sigma W$ is a Lagrangian flow solving~\eqref{ODE}. The integrand $b(\td{\Phi})$ is in $L^1(0,T)$ for a.e.~$(x,\omega)$ and the ODE~\eqref{ODE} is satisfied for a.e.~$(t,x,\omega)$: otherwise, since~$\Phi$ is a version of $\bar{\Phi}$, reasoning as in Remark~\ref{simple_rmk}, for some~$\omega$ in a positive measure set, the ODE would not be satisfied even by $\td{\bar{\Phi}}$. The continuity in time implies that, for a.e.~$(x,\omega)$, the ODE~\eqref{ODE} is satisfied for every~$t$. Therefore, this~$\Phi$ is the desired stochastic Lagrangian flow. 

\emph{Part 3: $W^{1,m}_{\loc}(\mathbb{R}^d)$-regularity of~$\Phi$.} We prove a stability result, which is interesting in itself.

\begin{lemma}\label{stability}
Let $m \geq 4$ be an even integer and assume that $(b_\eps)_\eps$ verifies Condition~\ref{LPSappprox}. If $\Phi^\eps$ are the associated regular stochastic flows, then, for every~$t \in [0,T]$, $(\Phi^\eps_t)_\eps$ converges to $\Phi_t$ weakly in $L^m(\Omega;W^{1,m}_{(1+|\cdot|)^{-d-1-m}}(\mathbb{R}^d))$.
\end{lemma}

This weak convergence result yields in particular that, for every $t \in [0,T]$, $\Phi_t$ belongs to $L^m(\Omega;W^{1,m}_{\loc}(\mathbb{R}^d))$ and, if $m>d$, to $L^m(\Omega;C^{0,\alpha}_{\loc}(\mathbb{R}^d))$ (for $\alpha=1-d/m$) by Sobolev immersion. The proof is complete.
\end{proof}

\begin{proof}[Proof of Lemma~\ref{stability}]
\emph{Step 1: Representation formula for fixed time}. For every $u_0 \in C^\infty_c(\mathbb{R}^d)$, for every~$\varphi \in C^\infty_c(\mathbb{R}^d)$, for a.e.~$(t,x,\omega)$, we have
\begin{align}
\lan u(t),\varphi\ran = \lan u_0,\varphi(\Phi_t)\ran,\label{repr_formula_CE}
\end{align}
as a consequence of the analogous property for $\bar{\Phi}$ and of Remark~\ref{simple_rmk}. In particular, taking the $H^{-1}$-valued weakly-$*$ time continuous version for~$u$ (cp.~Remark~\ref{weak_cont_progr}), we get the above formula for every~$t$, for every~$\omega$ (in a full measure set independent of~$t$). Moreover, by Theorem~\ref{existence CE} extended to every time by weak-$*$ continuity, $\sup_t E[\|u(t,\cdot)\|_{L^m_{(1+|\cdot|)^\alpha}(\mathbb{R}^d)}^m]$ is finite for every real $\alpha$ (since $u_0$ is bounded compactly supported); therefore, calling again $\varphi_n$ functions in $C^\infty_c(\mathbb{R}^d)$ with $\varphi_n(x)=x$ for $|x|\le n$, we have that, for every~$t$ fixed: $\lan u_0,\Phi_t\ran$ is in $L^m(\Omega)$ and 
\begin{align*}
E\big[\left| \lan u(t),id-\varphi_n\ran \right|^m\big] = E\big[\left|\lan u_0,\Phi_t-\varphi_n(\Phi_t)\ran \right|^m\big] \rightarrow 0 \quad \text{as } n \to \infty.
\end{align*}

\emph{Step 2: Approximation and conclusion}. Fix~$t \in [0,T]$ and $u_0 \in C^\infty_c(\mathbb{R}^d)$. Note that, for every~$\varphi \in C^\infty_c(\mathbb{R}^d)$, $\varphi(\Phi^\eps_t)$ is the solution $v^\eps$, at time~$0$, to the backward approximated stochastic transport equation, with final time~$t$ and final datum~$\varphi$. The approximated duality formula~\eqref{duality formula} (for $c=0$ and with a change of variable to avoid the ``tilde''), the approximation Condition~\ref{LPSappprox} on $(b_\eps)_\eps$ and equation~\eqref{repr_formula_CE} then give
\begin{equation}
\label{eqn_stability_eps}
E\left[\left|\lan u_0,\varphi(\Phi_t)-\varphi(\Phi^\eps_t)\ran \right|^m\right] = E\left[\left| \lan u(t),\varphi\ran -\lan u_0,v^\eps\ran \right|^m\right] \rightarrow 0 \quad \text{as } \eps \to 0.
\end{equation}
On the other hand, Corollary~\ref{corollary regularity backward} ensures that $v^\eps(0) = \varphi(\Phi^\eps_t)$ is bounded, uniformly in~$\eps$, in the space $L^m(\Omega;W^{1,m}_{(1+|\cdot|)^{-(d+1+m)}}(\mathbb{R}^d))$. Since this space is reflexive (see Remark~\ref{rem_weihted_spaces_reflexive}), $\varphi(\Phi^\eps_t)$ converges weakly, as $\eps\rightarrow0$ and up to the choice of a subsequence, to an element~$\Psi^\varphi_t$ with
\begin{align*}
\|\Psi^\varphi_t\|_{L^m(\Omega;W^{1,m}_{(1+|\cdot|)^{-(d+1+m)}}(\mathbb{R}^d))} \le C\|\varphi\|_{W^{1,2m}_{(1+|\cdot|)^{-(d+1+2m)}}}
\end{align*}
for a constant $C$ which is independent of~$t$. In particular, for every $u_0 \in C^\infty_c(\mathbb{R}^d)$ and $F \in L^\infty(\Omega)$, we get
\begin{align*}
E\big[\lan u_0,\Psi^\varphi_t-\varphi(\Phi^\eps_t)\ran F\big]\rightarrow 0 \quad \text{as } \eps \to 0.
\end{align*}
Therefore, by~\eqref{eqn_stability_eps} we find $\varphi(\Phi_t)=\Psi^\varphi_t$ for a.e.~$(x,\omega)$. Now, taking $\varphi=\varphi_n$ (with bounded $W^{1,2m}_{(1+|\cdot|)^{-(d+1+2m)}}$ norm), we have
\begin{align*}
\|\varphi_n(\Phi_t)\|_{L^m(\Omega;W^{1,m}_{(1+|\cdot|)^{-(d+1+m)}}(\mathbb{R}^d))} \le C,
\end{align*}
where the constant~$C$ is independent of~$n$ and~$t$. As a consequence, $\varphi_n(\Phi_t)$ converges weakly in $L^m(\Omega;W^{1,m}_{(1+|\cdot|)^{-(d+1+m)}}(\mathbb{R}^d))$, as $n\rightarrow\infty$ and up to the choice of a subsequence. On the other hand, by Step 1 we know that $\lan u_0,\Phi_t-\varphi_n(\Phi_t)\ran \to 0$ in $L^m(\Omega)$, as $n \to \infty$. So, by a similar argument to the one for $\eps\rightarrow0$, any weak limit of $\varphi_n(\Phi_t)$ has to be $\Phi_t$, and hence
\begin{equation*}
\|\Phi_t\|_{L^m(\Omega;W^{1,m}_{(1+|\cdot|)^{-(d+1+m)}}(\mathbb{R}^d))} \le C. \qedhere
\end{equation*}
\end{proof}

\section{Towards classical pathwise uniqueness}

So far we have investigated the problem of path-by-path uniqueness for the equations~\eqref{stoch cont} and~\eqref{SDE}. In some sense, this is the strongest type of uniqueness we know. Indeed, we can come back heuristically to pathwise uniqueness for~\eqref{SDE} in this way: given two processes~$X$ and~$Y$ which are solutions to~\eqref{SDE} with the same initial datum, then, for a.e.~$\omega$, $X(\omega)$ and $Y(\omega)$ solve the sDE at fixed~$\omega$ (more precisely, $\td{X}^\omega$ and $\td{Y}^\omega$ solve the random ODE~\eqref{ODE}), so that, by path-by-path uniqueness, they must coincide. However, since we only deal with flows, we are not able to give a ``classical'' pathwise uniqueness result (among processes instead of flows), as a direct consequence of Theorem~\ref{main thm flows}. Thus, we will now see how to modify the duality argument to get a more classical pathwise uniqueness, though still the initial datum cannot be a single point $x \in \mathbb{R}^d$, but it has to be a suitable diffused random variable.

\subsection{The first result}\label{firstpathwise}

The easiest consequence of Theorem~\ref{pathbypath uniq SPDE} (applied to the continuity equation) is pathwise uniqueness among solutions with conditional laws (given the Brownian motion) in $L^m([0,T],L^m_{(1+|\cdot|)^{-\alpha}}(\mathbb{R}^d))$.

The relevant concept of solution and the result are shown below, but let us explain the idea. As already mentioned, we need the initial datum~$X_0$ to be diffuse. We could take e.g.~the probability space $(C([0,T];\mathbb{R}^d)\times B_R(y_0),Q\otimes\mc{L}^d)$ (with the suitable $\sigma$-algebra), with~$Q$ as Wiener measure, for some $R>0$, $y_0 \in \mathbb{R}^d$, and $X_0(\gamma,x)=x$, $W_t(\gamma,x)=\gamma_t$; the filtration must be any filtration~$(\mc{G}_t)_t$ (satisfying the standard assumptions) such that $\mc{G}_t$ contains $\sigma\{X_0,W_s|s\le t\}$. The solution $X=X(\gamma,x)$ should be thought of as a flow, for fixed~$\gamma \in C([0,T];\mathbb{R}^d)$, solving the sDE at this fixed~$\gamma$. Now we ask: among which class of processes path-by-path uniqueness applies, implying pathwise uniqueness? We have to require (again heuristically) that, for $Q$-a.e.~Brownian trajectory $W=\gamma$, $(X_t(\gamma,\cdot))_\#\mc{L}^d$ is a diffuse measure. This is true in the case above, while for the general case (of a general probability space and general initial datum~$X_0$), we must require that~$X_0$ has a diffuse law and that ``the law of~$X_t$ for fixed~$\gamma$ is diffuse'' too. This law of~$X_t$ for fixed~$\gamma$ is the conditional law of~$X_t$ given the Brownian motion~$W$; see e.g.~\cite[Chapter~1]{STRVAR79} for a reference on conditional law.

\begin{definition}\label{1defpathw}
Let $m \geq 1$, $\alpha\in\mathbb{R}$; let $W$ be a Brownian motion \textup{(}on a probability space~$(\Omega,\mc{A},P)$\textup{)}, let $(\mc{F}_t)_t$ be its natural completed filtration. An $\mathbb{R}^d$-valued process~$X$ on~$\Omega$ is said to have \emph{conditional \textup{(}marginal\textup{)} laws} \textup{(}given the Brownian motion $W$\textup{)} \emph{in $L^m([0,T],L^m_{(1+|\cdot|)^{\alpha}}(\mathbb{R}^d))$} if, for a.e.~$t \in [0,T]$, the conditional law of~$X_t$ given $\mc{F}_t$ has a density \textup{(}with respect to Lebesgue measure\textup{)} $\rho(t,x,\omega)$ and, for a.e.~$\omega \in \Omega$, $\rho(\cdot,\cdot,\omega)$ belongs to $L^m([0,T],L^m_{(1+|\cdot|)^{\alpha}}(\mathbb{R}^d))$.
\end{definition}

\begin{theorem}\label{1pathuniq}
Let $m\ge4$, $s\in\mathbb{R}$. Let $W$, $(\mc{F}_t)_t$ be as above and let~$X_0$ be a random variable on~$\Omega$, independent of $W$, such that the law of~$X_0$ has a density \textup{(}with respect to the Lebesgue measure\textup{)} in $L^m_{(1+|\cdot|)^{2s+d+1}}(\mathbb{R}^d)$. Assume Condition~\ref{LPSreg}. Then, for every $\alpha\le s$, strong existence and pathwise uniqueness hold for~\eqref{SDE} with initial datum~$X_0$, among solutions with conditional laws in $L^m([0,T],L^m_{(1+|\cdot|)^{\alpha}}(\mathbb{R}^d))$. More precisely, if $(\mc{G}_t)_t$ is an admissible filtration \textup{(}satisfying the standard assumptions\textup{)} on~$\Omega$ \textup{(}i.e.~$X_0$ is $\mc{G}_0$-measurable and $W$ is a Brownian motion with respect to $(\mc{G}_t)_t$\textup{)}, then there exists a unique $\mc{G}$-adapted process solving~\eqref{SDE}, starting from~$X_0$ and with conditional laws in $L^m([0,T],L^m_{(1+|\cdot|)^{\alpha}}(\mathbb{R}^d))$.
\end{theorem}

We will not give all the details of the proof, also because the proof is similar to the one of the next Theorem~\ref{2pathwise}.

\begin{proof}
The proof of uniqueness is similar to the one of the first part of Theorem~\ref{Lagrthm}. Suppose by contradiction that there exist two different solutions~$X$ and~$Y$ with the properties above. Then it is possible to find a time $t_0$, two disjoint Borel sets $E$ and $F$ in $\mathbb{R}^d$ and a measurable set $\Omega'$ in~$\Omega$ with $P(\Omega')>0$ such that $X_{t_0}(\omega)$ belongs to $E$ and $Y_{t_0}(\omega)$ belongs to $F$ for every~$\omega$ in $\Omega'$.

On $C([0,T];\mathbb{R}^d)$ we denote by $Q$ the Wiener measure and by $\Gamma$ the essential image of $\Omega'$ under the map $W$, i.e.\ $\Gamma \coloneqq \{\gamma \in C([0,T];\mathbb{R}^d) \colon P(\Omega' | W=\gamma)>0\}$ (this definition makes sense up to $Q$-negligible sets). Since $Q$ is the image measure of $P$ under $W$, we have $Q(\Gamma)>0$. For every~$t$, we define $\td{\mu}_t$ as the conditional law on $\mathbb{R}^d$ of $\td{X}_t$, restricted to $\Omega'$, given $W$, i.e., for every~$\varphi$ in $C_b(\mathbb{R}^d)$,
\begin{equation*}
\lan\td{\mu}_t^{\gamma},\varphi\ran= E[\varphi(\td{X}_t) \1_{\Omega'} | W=\gamma] ,\quad \mbox{for $Q$-a.e.~}\gamma\in\Gamma.
\end{equation*}
We analogously define $\td{\nu}_t$ for $\td{Y}$ instead of $\td{X}$. Then one can show that:
\begin{itemize}
\item for $Q$-a.e.~$\gamma \in \Gamma$, $\td{\mu}^{\gamma}$ and $\td{\nu}^{\gamma}$ are weakly continuous (in time) solutions to the random CE~\eqref{random CE} at $\gamma$ fixed;
\item $\td{\mu}$ and $\td{\nu}$ belong to the $\mc{L}_{\td{b}}$ class;
\item $\td{\mu}$ and $\td{\nu}$ differ at time $t_0$.
\end{itemize}
So we have found two different $\mc{L}_{\td{b}^{\gamma}}$ solutions to the random CE at $\gamma$ fixed, for a non-negligible set of $\gamma$. This is a contradiction, and thus, the proof of uniqueness is complete.

Strong existence is a consequence of the existence of a stochastic Lagrangian flows~$\Phi$ solving~\eqref{SDE}. Indeed, defining $X_t(\omega)\coloneqq \Phi_t^\omega(X_0(\omega))$, we observe the following facts: 
\begin{itemize}
\item Since $\Phi(x)$ solves the sDE with initial datum~$x$, for a.e.~$x$, and~$X_0$ is absolutely continuous (with respect to the Lebesgue measure), $X$ verifies for a.e.~$\omega$
\begin{equation*}
X_t=X_0+\int^t_0 b(s,X_s)ds+W_t.
\end{equation*}
\item $X$ is obviously $\mc{H}$-adapted, where $\mc{H}_t=\sigma(\{X_0,W_s|s\le t\}\cup\mathcal{N})$ is the minimal admissible filtration ($\mathcal{N}$ are the $P$-null sets).
\item Let $u_0$ be the density of the law of~$X_0$ and let~$u$ be the solution to the sCE, with initial datum $u_0$, in $L^\infty([0,T];L^m(\Omega;L^m_{(1+|\cdot|)^s}(\mathbb{R}^d)))$  as in Theorem~\ref{existence CE}. Then the law of~$X_t$ has $u(t)$ as conditional density, given $W$. To prove this, notice that $W$ and~$\Phi$ are adapted to the Brownian (completed) filtration~$\mc{F}$ and that~$X_0$ is independent of $\mc{F}_T$, so, for any test function~$\varphi \in C_c^\infty(\mathbb{R}^d)$ and any $\psi \in C_b(C([0,T];\mathbb{R}^d))$, we have\begin{align*}
E[\varphi(X_t)\psi(W)] & =E[\varphi(\Phi_t(X_0))\psi(W)]  \\
  & = E\Big[\int_{\mathbb{R}^d}\varphi(\Phi_t(x))u_0(x)\psi(W)dx\Big] =E\Big[\int_{\mathbb{R}^d}\varphi(x)u(t,x)\psi(W)dx\Big],
\end{align*}
where in the second passage we used independence (precisely, in the form of Lemma~\ref{indep} in the following paragraph) and the last passage is a consequence of $u(t)=(\Phi_t)_\#u_0$.
\end{itemize}
Thus,~$X$ is the desired solution and also existence is proved.
\end{proof}

\subsection{The second result}\label{secondpathwise}

The previous result is somehow limited, at least for uniqueness, by our hypothesis on conditional laws. In this paragraph we prove that actually pathwise uniqueness holds among processes whose marginal laws are diffuse (the precise hypothesis is stated below), with no need to control conditional laws.

To understand the relation with the previous Theorem~\ref{1pathuniq}, consider again the case discussed at the beginning of the previous paragraph and notice that, given a process~$X$ on $(C([0,T];\mathbb{R}^d)\times B_R(y_0),Q\otimes\mc{L}^d)$, the law $\rho_t$ of~$X_t$ is the $Q$-average, on $C([0,T];\mathbb{R}^d)$, of the conditional laws $\rho^\gamma_t$ of $X_t(\gamma,x)$, given the Brownian trajectory~$\gamma$. So the fact that the law (that is, the mean of the conditional laws) is diffuse is a weaker condition than the hypothesis on $Q$-a.e.~conditional law. Hence the class of processes whose marginal laws are diffuse is larger that the class used in Theorem~\ref{1pathuniq}, and in particular, the uniqueness result in the following Theorem~\ref{2pathwise} is morally stronger. Actually no implication holds between the two uniqueness results (for a technicality on the bounds on the densities, see the next definition), but still the idea is that uniqueness is stronger in Theorem~\ref{2pathwise}.

\begin{definition}
Let $m \geq 1$, $\alpha\in\mathbb{R}$. An $\mathbb{R}^d$-valued process~$X$ is said to have \emph{\textup{(}marginal\textup{)} laws in $L^\infty([0,T],L^m_{(1+|\cdot|)^{\alpha}}(\mathbb{R}^d))$} if, for a.e.~$t \in [0,T]$, the law of~$X_t$ has a density \textup{(}with respect to the Lebesgue measure\textup{)} $\rho(t,x)$, which belongs to $L^\infty([0,T],L^m_{(1+|\cdot|)^{\alpha}}(\mathbb{R}^d))$.
\end{definition}

As previously mentioned, this class seems to be larger than that of Definition~\ref{1defpathw}. Rigorously speaking, it is not: to deduce $\mu_t\in L^m([0,T]\times\mathbb{R}^d)$ from $\mu^\omega_t\in L^m([0,T]\times\mathbb{R}^d)$ for a.e.~$\omega$, we need the additional condition that $\int\|\mu_t^\omega\|_{L^m([0,T]\times\mathbb{R}^d)}P(d\omega)$ is finite.

Here is the main pathwise uniqueness result:

\begin{theorem}\label{2pathwise}
Let $m\ge4$, $s\in\mathbb{R}$. Let~$W$ be a Brownian motion \textup{(}on a probability space $(\Omega,\mc{A},P)$\textup{)}. Let~$X_0$ be a random variable on~$\Omega$, independent of $W$, such that the law of~$X_0$ has a density \textup{(}with respect to the Lebesgue measure\textup{)} in $L^m_{(1+|\cdot|)^{2s+d+1}}(\mathbb{R}^d)$. Assume Condition~\ref{LPSreg}. Then, for every $\alpha\le s$, strong existence and pathwise uniqueness hold for~\eqref{SDE} with initial datum~$X_0$, among solutions with laws in $L^\infty([0,T],L^m_{(1+|\cdot|)^{\alpha}}(\mathbb{R}^d))$. More precisely, if $(\mc{G}_t)_t$ is an admissible filtration \textup{(}satisfying the standard assumptions\textup{)} on~$\Omega$ \textup{(}i.e.~$X_0$ is $\mc{G}_0$-measurable and $W$ is a Brownian motion with respect to $(\mc{G}_t)_t$\textup{)}, then there exists a unique $\mc{G}$-adapted process solving~\eqref{SDE}, starting from~$X_0$ and with laws in $L^\infty([0,T],L^m_{(1+|\cdot|)^{\alpha}}(\mathbb{R}^d))$. 
\end{theorem}

\begin{proof}[Proof of uniqueness]
First we give the idea of the proof. Let $X$, $Y$ be two solutions to~\eqref{SDE} which are adapted to an admissible filtration $(\mc{G}_t)_t$. Set $\mu_t\coloneqq \delta_{X_t}-\delta_{Y_t}$; then~$\mu$ is a random distribution which solves the sCE
\begin{equation*}
\partial_t\mu+\diverg(b\mu)+\sum^d_{k=1}\partial_k\mu\circ\dot{W}=0
\end{equation*}
in the sense of distributions, with $\mu_0=0$. We have to prove that $\mu\equiv0$. We again want to use duality: if $v$ solves the backward sTE
\begin{equation*}
\partial_t v +b\cdot\nabla v +\sum^d_{k=1}\partial_kv\circ\dot{W} =0
\end{equation*}
with final time $t_f$ and final condition $v(t_f)=\varphi$ fixed, then formally it holds $\lan \mu_t,\varphi\ran= \lan \mu_0,v_0\ran =0$. But now we must be careful: expressions like
\begin{equation}
\lan \mu_s,b(s) \cdot\nabla v(s)\ran,\label{exprcontr}
\end{equation}
which appear naturally in the rigorous proof of the duality formula, are no more under control: $\mu_s$ is only a measure, while~$b(s)$ and $\nabla v(s)$ are not continuous (not even bounded). There are two key facts. The first one is where the integrability hypothesis plays a role: if we replace $\mu_s$ by its average $\rho_s=E[\mu_s]=(X_s)_\#P-(Y_s)_\#P$, we can estimate~\eqref{exprcontr} since the density of $\rho$ is in the correct integrability class for H\"older's inequality. However, taking the expectation, we have to deal with $E[\mu_s\nabla v(s)]$. Here enters the second key fact, namely that $\mu_s$ and $\nabla v(s)$ are independent, since $\mu_s$ is $\mc{G}_s$-measurable, while $v(t)$ (as backward solution) is adapted to the Brownian backward (completed) filtration $\mc{F}^{s}$, which is independent of~$\mc{G}_s$. Having this in mind, we come to the rigorous proof of the result.

Take $t_f \in [0,T]$ and~$\varphi \in C^\infty_c(\mathbb{R}^d)$. Let $b_\eps$ be as in Condition~\ref{LPSappprox}, let $v_\eps$ be the solution to the approximated backward transport equation
\begin{equation*}
\partial_t v_\eps +b_\eps\cdot\nabla v_\eps +\sum^d_{k=1}\partial_k v_\eps\circ\dot{W} =0
\end{equation*}
with final time $t_f$ and final datum $v_\eps(t_f)=\varphi$. With the usual notation with tilde ($\td{v}_\eps(s,x)=v_\eps(s,x+\sigma W_s)$, $\td{X}_s=X_s-\sigma W_s$), the chain rule gives
\begin{align*}
dv_\eps(t,X_t)= d\td{v}_\eps(t,\td{X}_t) & = \td{b}(t,\td{X}_t)\cdot\nabla \td{v}_\eps(t,\td{X}_t)dt -\td{b}_\eps(t,\td{X}_t)\cdot\nabla \td{v}_\eps(t,\td{X}_t)dt \\	
  & = [(b-b_\eps)\cdot\nabla v_\eps](t,X_t)dt
\end{align*}
and similarly for~$Y$. Subtracting the expression for~$Y$ from that for~$X$, we get
\begin{equation*}
\varphi(X_{t_f})-\varphi(Y_{t_f})=\int^{t_f}_0[(b-b_\eps)\cdot\nabla v_\eps](s,X_s)ds-\int^{t_f}_0[(b-b_\eps)\cdot\nabla v_\eps](s,Y_s)ds .
\end{equation*}
We now claim that
\begin{equation}
\lim_{\eps\rightarrow0}\int^{t_f}_0 E \big[ |(b-b_\eps)\cdot\nabla v_\eps|(s,X_s) \big] ds=0\label{claim}
\end{equation}
and similarly for~$Y$. Assuming this, we obtain $\varphi(X_{t_f})=\varphi(Y_{t_f})$ and then, by the arbitrariness of~$\varphi$ and~$t_f$, also $X\equiv Y$.

For proving~\eqref{claim}, we want to exploit the independence of $\nabla v_\eps(s)$ and $X_s$, for fixed $s \in [0,T]$. To this end, we need the following elementary lemma:

\begin{lemma}\label{indep}
Consider two measurable spaces $(F_1,\mc{F}_1)$, $(F_2,\mc{F}_2)$ and a probability measure~$P$ on $(F_2,\mc{F}_2)$. Let $f \colon F_1 \times F_2 \rightarrow\mathbb{R}$, $Z \colon F_2\rightarrow F_1$ be two measurable functions and denote by~$\rho$ the law of~$Z$ on~$F_1$. Suppose that there exists a $\sigma$-algebra $\mc{A} \subset \mc{F}_2$ such that~$f$ is $\mc{F}_1\otimes\mc{A}$-measurable and~$Z$ is independent of~$\mc{A}$. Assume also $\int_{F_1} \int_{F_2}|f(y,\omega)|P(d\omega)\rho(dy)<\infty$. Then it holds
\begin{equation*}
\int_{F_2} f(Z(\omega),\omega) P(d\omega)=\int_{F_1} \int_{F_2} f(y,\omega) P(d\omega) \rho(dy) .
\end{equation*}
\end{lemma}

\begin{proof}
The lemma is clear for $f(y,\omega)=g(y)h(\omega)$, when $g$ is $\mc{F}_1$-measurable and integrable (with respect to~$\eta$), and~$h$ is $\mc{A}$-measurable and integrable (with respect to~$P$). The general case is obtained by approximating~$f$ with sums of functions as above.
\end{proof}

Applying this lemma with $F_1 = \mathbb{R}^d$, $F_2 = \Omega$, $f=|(b-b_\eps) \cdot \nabla v_\eps|$ and $Z=X$ with law~$\rho$, for fixed $s \in [0,t_f]$, and then integrating over $s \in [0,t_f]$, we obtain 
\begin{equation*}
 \int^{t_f}_0 E \big[ |(b-b_\eps)\cdot\nabla v_\eps|(s,X_s) \big] ds
 = \int^{t_f}_0 \lan E[|(b(s)-b_\eps(s))\cdot \nabla v_\eps(s)|],\rho_s\ran ds .
\end{equation*}
We would like to use H\"older's inequality to conclude with~\eqref{claim}. Since the density of $\rho$ belongs to $L^\infty([0,T];L^m_{(1+|\cdot|)^{-\alpha}}(\mathbb{R}^d))$ by assumption, it is enough to prove that
\begin{equation}
\int^{t_f}_0\Big(\int_{\mathbb{R}^d}|b-b_\eps|^{m'}|E[|\nabla v_\eps|^{m'}](1+|x|)^{\alpha m'/m} dx\Big)^{1/m'}ds < \infty .\label{finalest}
\end{equation}
The proof of~\eqref{finalest} is almost the same of that of Lemma~\ref{lemma tool stoch}. The only change is the exponent $1/m'$ in the time integral. For this reason we need even less, namely it suffices that
\begin{equation*}
\int^T_0 \Big(\int_{\mathbb{R}^d}(1+|x|)^{-\beta}|b-b_\eps|^{\td{p}}dx\Big)^{1/\td{p}}ds< \infty
\end{equation*}
for some $\beta\ge0$ and $\td{p}>m'$, which holds under Condition~\ref{LPSreg}. The rest of the proof requires only obvious changes compared to the proof of Lemma~\ref{lemma tool stoch}. 
\end{proof}

\begin{proof}[Proof of existence]
As for Theorem~\ref{1pathuniq}, strong existence is an easy consequence of the existence of a stochastic Lagrangian flows~$\Phi$ solving~\eqref{SDE}. Indeed, defining $X_t(\omega)\coloneqq \Phi_t^\omega(X_0(\omega))$, we observe the following facts.
\begin{itemize}
\item Since $\Phi(x)$ solves the sDE with initial datum~$x$, for a.e.~$x$, and~$X_0$ is absolutely continuous (with respect to the Lebesgue measure),~$X$ verifies for a.e.~$\omega$
\begin{equation*}
X_t=X_0+\int^t_0 b(s,X_s)ds+W_t.
\end{equation*}
\item $X$ is obviously $\mc{H}$-adapted, where $\mc{H}_t=\sigma(\{X_0,W_s|s\le t\}\cup\mathcal{N})$ is the minimal admissible filtration ($\mathcal{N}$ are the $P$-null sets).
\item Let $u_0$ be the density of the law of~$X_0$ and let~$u$ be the solution to the sCE, with initial datum~$u_0$, in $L^\infty([0,T];L^m(\Omega;L^m_{(1+|\cdot|)^s}(\mathbb{R}^d)))$ as in Theorem~\ref{existence CE}. Then the law of~$X_t$ has density given by $\mu_t=E[u(t)]$. Indeed~$\Phi$ and~$X_0$ are independent (which allows to use Lemma~\ref{indep}), so, for any test function~$\varphi \in C_c^\infty(\mathbb{R}^d)$, we have
\begin{equation*}
E[\varphi(X_t)]=E[\varphi(\Phi_t(X_0))]=E\Big[\int_{\mathbb{R}^d}\varphi(\Phi_t(x))u_0(x)dx\Big] =E\Big[\int_{\mathbb{R}^d}\varphi(x)u(t,x)dx\Big],
\end{equation*}
where the last passage is a consequence of $u(t)=(\Phi_t)_\#u_0$. We further have for~$\mu$
\begin{equation*}
\sup_{t\in[0,T]}\int_{\mathbb{R}^d}|\mu_t|^m(1+|x|)^{-\alpha}dx\le \sup_{t\in[0,T]}\int_{\mathbb{R}^d}E[|u_t|^m](1+|x|)^{-\alpha}dx <\infty ,
\end{equation*}
so~$X$ has law in $L^\infty([0,T],L^m_{(1+|\cdot|)^s}(\mathbb{R}^d))$.
\end{itemize}
Thus,~$X$ is the desired solution and also existence is proved.
\end{proof}

\section{Path-by-path results for sDE}

\subsection{Path-by-path uniqueness of individual trajectories}

We next consider equation~\eqref{SDE}. Its integral formulation is 
\begin{equation*}
X_{t}(\omega)  =x+\int_{0}^{t}b(s,X_{s}(\omega)) ds+\sigma W_{t}(\omega)
\end{equation*}
and therefore, we may give a path-by-path meaning to it. Assume for some constant $C>0$ that $b \colon [0,T] \times \mathbb{R}^{d} \rightarrow \mathbb{R}^{d}$ is a measurable locally bounded function (defined for all $(t,x)$, not only a.e.). As before, let us assume that $W$ has continuous trajectories (everywhere). Given $\omega\in\Omega$, hence given the continuous function $t\mapsto W_{t}(\omega)$, consider all continuous functions $y \colon [0,T]  \rightarrow \mathbb{R}^{d}$ which satisfy the identity 
\begin{equation*}
y(t)  =x+\int_{0}^{t}b(  s,y(s)  ) ds+\sigma W_{t}(\omega)
\end{equation*}
and call $C(\omega,x)$ the set of all such functions. Denote by $Card(C(\omega,x))$ the cardinality of the set~$C(\omega,x)$.

\begin{remark}
If~$b$ is continuous with $\vert b(  t,x)  \vert \leq C(  1+\vert x\vert )$ for all $(t,x) \in [0,T] \times\mathbb{R}^{d}$, then by classical deterministic arguments $C(\omega,x)$ is non empty.
\end{remark}

\begin{definition}
We say that the sDE satisfies \emph{path-by-path uniqueness} if
\begin{equation*}
P\big(  Card(  C(\omega,x) ) \leq 1 \text{ for all } x\in\mathbb{R}^d \big)  =1,
\end{equation*}
namely if for a.e.~$\omega \in \Omega$, $C(\omega,x)$ is at most a singleton for every~$x$ in $\mathbb{R}^d$.
\end{definition}

To our knowledge, the only two results on path-by-path uniqueness are~\cite{DAVIE07} and~\cite{CATGUB13}. We present here a new strategy for this kind of results.

Let $y\in C(\omega,x)$ be a solution. Set
\[
z^\omega(t)  \coloneqq y(t)  -\sigma W_{t}(
\omega)
\]
solution of
\begin{equation*}
z^\omega(t)  =x+\int_{0}^{t}b(s,z^\omega(s)  +\sigma W_{s}(\omega)  )  ds =x+\int_{0}^{t}\widetilde{b}^\omega(s,z^\omega(s))  ds ,
\end{equation*}
where, as usual, $\td{b}^\omega(t,x)=b(t,x+\sigma W_t(\omega))$. Consider the time-dependent Dirac measure $\widetilde{\mu}_{\omega}(t)  =\delta_{z^\omega(t)  }$ on $\mathbb{R}^{d} 
$. For every $\varphi\in C_{c}^{\infty}(\mathbb{R}^{d})$, we write $\left\langle \widetilde{\mu}^\omega(t)  ,\varphi\right\rangle
$ for $\int_{\mathbb{R}^{d}}\varphi d\widetilde{\mu}^\omega(t)$, which, in this particular case, is simply $\varphi(z^\omega(t))$.

\begin{lemma}
\label{lemma_equation_mu_tilde}
For all $t\in [0,T]$ and $\varphi\in C^{1}([0,T];C_{c}^{\infty}(\mathbb{R}^{d}))$, we have
\begin{equation*}
\left\langle \widetilde{\mu}^\omega(t)  ,\varphi(t) \right\rangle =\left\langle \widetilde{\mu}^\omega(0)  ,\varphi(  0)  \right\rangle +\int_{0}^{t}\left\langle
\widetilde{\mu}^\omega(s)  ,\widetilde{b}^\omega(s)  \cdot\nabla\varphi(s)  + \partial_t \varphi(s) \right\rangle ds .
\end{equation*}
\end{lemma}

\begin{proof}
We have to prove that
\begin{equation*}
\varphi(t,z^\omega(t))  =\varphi(0,z^\omega(0))  +\int_{0}^{t} \big( \widetilde{b}^\omega(s,z^\omega(s))  \cdot\nabla \varphi(s,z^\omega(s)) + \partial_t \varphi (s,z^\omega(s) )  \big)  ds ,
\end{equation*}
which is true by ordinary calculus.
\end{proof}

We can now prove a central fact. For a bounded function $f$ and a Borel set $E$, denote with $\|f\|_{0,E}$ the supremum of $f$ over $E$; in general, this is not the essential supremum, unless $f$ is continuous.

\begin{theorem}
\label{Theo sufficient for sDE}
Let $(b_{\eps})_{\eps\in (0,1)}$ be a family in $C_{c}^{\infty}([0,T] \times\mathbb{R}^{d})$. Assume that, for
every $t_{f}\in [0,T]  $ and $v_{0}\in C_{c}^{\infty} (\mathbb{R}^{d})$, we have 
\begin{equation}
P-\lim_{\eps\rightarrow0}\int_{0}^{t_{f}} \Vert (  b-b_{\eps}) \cdot\nabla v_{\eps}^\omega \Vert
_{0,B_R}ds=0 \label{main condition for sDE} 
\end{equation}
for every positive $R$, and where, for every $\eps\in(0,1)$, $v_{\eps}$ is the smooth solution of the backward sPDEs~\eqref{backward SPDE 1} corresponding to
$b_{\eps}$ and $v_{0}$, with $c_{\eps}=0$ \textup{(}$v_\eps^\omega$ denotes $v_\eps(\cdot,\cdot,\omega)$ as before\textup{)}. Then path-by-path uniqueness holds for~\eqref{SDE}.
\end{theorem}

\begin{proof}
\emph{Step 1: Identification of $\Omega_0$, independently of~$x$.}
By assumption~\eqref{main condition for sDE}, given $t_{f}\in [0,T]  $ and $v_{0}\in C_{c}^{\infty}(\mathbb{R}^{d})$, there exist a full measure set $\Omega_{t_{f},v_{0}} \subset\Omega$ and a sequence $\eps_{n}\rightarrow0$ such that
\begin{gather}
\widetilde{v}_{\eps_{n}}^\omega \text{ belongs to } C^{1}([
0,T]  ;C_{c}^{\infty}(\mathbb{R}^{d}))  \text{ and satisfies~\eqref{dual equation} (with $c=0$), for all $n \in \mathbb{N}$,}
\label{omegawise 1 sDE} \\
\lim_{n\rightarrow\infty}\int_{0}^{t_{f}} \big\|(  b^\omega  -b_{\eps_{n}}^\omega)  \cdot\nabla
v_{\eps_{n}}^\omega \big\|_{0,B_R}ds=0  \nonumber
\end{gather}
for all $\omega\in\Omega_{t_{f},v_{0}}$. Hence also
\begin{equation}
\lim_{n\rightarrow\infty}\int_{0}^{t_{f}} \big\| ( \widetilde{b}^\omega  -\widetilde{b}_{\eps_{n}}^\omega )  \cdot\nabla\widetilde{v}_{\eps_{n}}^\omega \big\|_{0,B_{R}-\sigma W_{t} (\omega)} ds=0 \label{omegawise 2 sDE}
\end{equation}
for all $\omega\in\Omega_{t_{f},v_{0}}$. Let $\mathcal{D}\subset C_{c}^{\infty}(\mathbb{R}^{d})$ be a countable set which separates points, i.e.~for all $a\neq b\in\mathbb{R}^{d}$, there exists $v_{0}\in\mathcal{D}$ with $v_{0}(a) \neq v_{0}(b)$. By a diagonal procedure, there exist a full measure set $\Omega_{0}\subset\Omega$ and a sequence $\eps_{n}\rightarrow0$ such that properties~\eqref{omegawise 1 sDE} and~\eqref{omegawise 2 sDE} hold for all $t_{f}\in [0,T]  \cap\mathbb{Q}$, $v_{0}\in\mathcal{D}$, $n,R\in\mathbb{N}$ and $\omega\in \Omega_{0}$. Since, given $\omega\in \Omega_{0}$ and $R\in\mathbb{N}$, there exists $R_{\omega}^{\prime}\in\mathbb{N}$ such that $B_{R} \subset B_{R_{\omega }^{\prime}} -\sigma W_{t}(\omega)$ for all $t\in [0,T]  $, we may replace~\eqref{omegawise 2 sDE} by
\begin{equation}
\lim_{n\rightarrow\infty}\int_{0}^{t_{f}} \big\| (\widetilde{b}^\omega  -\widetilde{b}_{\eps_{n}}^\omega ) \cdot\nabla\widetilde{v}_{\eps_{n}}^\omega \big\|_{0,B_R}ds=0
\label{omegawise 2 bis sDE}
\end{equation}
for all $t_{f}\in [0,T]  \cap\mathbb{Q}$, $v_{0}\in\mathcal{D}$,
$R\in\mathbb{N}$ and $\omega\in \Omega_0$.

\emph{Step 2: $C(\omega,x) $ is a singleton for every~$x \in \mathbb{R}^d$ and $\omega \in \Omega_0$, i.e.~path-by-path uniqueness holds.} Given $\omega\in \Omega_0$ and $y^{(i)}\in C(\omega,x)$, $i=1,2$, we define the (signed) measure
\begin{equation*}
\widetilde{\rho}^\omega(t)  \coloneqq \delta_{y^{(1)}(t) - \sigma W_t(\omega)}  - \delta_{y^{(2)}(t) - \sigma W_t(\omega)},
\end{equation*}
which satisfies
\begin{equation*}
\left\langle \widetilde{\rho}^\omega(t_{f})  ,\varphi(
t_{f})  \right\rangle =\int_{0}^{t_{f}}\left\langle \widetilde{\rho
}^\omega(s)  ,\widetilde{b}^\omega(s)
\cdot\nabla\varphi(s)  + \partial_t \varphi(s) \right\rangle ds
\end{equation*}
for all $t_{f}\in [0,T]$ and $\varphi\in C^{1}([0,T];C_{c}^{\infty}(\mathbb{R}^{d}))$, due to Lemma~\ref{lemma_equation_mu_tilde}. In particular, this holds for $\varphi=\widetilde{v}_{\eps_{n}}^\omega $ and thus, by~\eqref{dual equation}, we get 
\begin{equation*}
\left\langle \widetilde{\rho}^\omega(t_{f})  ,v_{0} (\cdot+\sigma W_{t_{f}}(\omega))  \right\rangle =\int_{0}^{t_{f}}\left\langle \widetilde{\rho}^\omega(s) ,(
\widetilde{b}^\omega(s)  -\widetilde{b}_{\eps_{n}}^\omega(s) )  \cdot\nabla\widetilde{v}_{\eps_{n}}^\omega(s)  \right\rangle ds .
\end{equation*}
Then, if $R>0$ is such that $\vert y^{(i)}(t)\vert \leq R$ for $t\in [0,T]$ and $i=1,2$, we find
\begin{equation*}
\big\vert \left\langle \widetilde{\rho}^\omega(t_{f})
,v_{0} (\cdot+\sigma W_{t_{f}}(\omega))
\right\rangle \big\vert \leq2\int_{0}^{t_{f}} \big\Vert (
\widetilde{b}^\omega(s)  -\widetilde{b}_{\eps_{n}}^\omega(s) )  \cdot\nabla\widetilde{v}_{\eps_{n}}^\omega(s) \big\Vert _{0,B_R}ds ,
\end{equation*}
and thus $\left\langle \widetilde{\rho}^\omega(t_{f}),v_{0}(\cdot+\sigma W_{t_{f}}(  \omega)  )\right\rangle =0$ is satisfied due to the identity~\eqref{omegawise 2 bis sDE}. This is equivalent to $\left\langle \widetilde{\rho}^\omega (t_{f},\cdot-\sigma W_{t_{f}}(\omega))  ,v_{0}\right\rangle =0$, which implies $\widetilde{\rho}^\omega(t_{f},\cdot-\sigma W_{t_{f}}(\omega))=0$ since $v_0 \in \mathcal{D}$ was arbitrary and~$\mathcal{D}$ separates points. Consequently, $y^{(1)}(t_{f})  =y^{(2)}(t_{f})$ follows. This holds true for every $t_{f}\in [0,T] \cap\mathbb{Q}$, and since $t\mapsto y^{(1)}(t)$ is continuous, we get $y^{(1)}(t) =y^{(2)}(t)$ for every~$t \in [0,T]$. This finishes the proof of the theorem.
\end{proof}

Theorem~\ref{Theo sufficient for sDE} is, in a sense, our main result on path-by-path uniqueness, although assumption~\eqref{main condition for sDE} is not explicit in terms of~$b$. Roughly speaking, this condition is true when we have a uniform bound (in some probabilistic sense) for $\Vert \nabla\widetilde{v}_{\eps_{n}}^\omega \Vert_{0,B_R}$. It introduces a new approach to the very difficult question of path-by-path uniqueness, which may be easily generalized, for instance, to sDEs in infinite dimensions (which will be treated in separate works). A simple consequence is:

\begin{corollary}
\label{Coroll sufficient for sDE}
Let $(b_{\eps})_{\eps\in(0,1)}$ be a family in $C_{c}^{\infty}([0,T] \times\mathbb{R}^{d})$ which converges uniformly to~$b$ on compact sets $[0,T]  \times B_{R}$,
for every $R>0$. Assume that, for every $t_{f}\in [0,T]  $, $R>0$ and $v_{0}\in C_{c}^{\infty}(\mathbb{R}^{d})$, we have
\begin{equation}
\sup_{\eps\in(  0,1)  }E\int_{0}^{t_{f}}\left\Vert \nabla
v_{\eps}\right\Vert _{0,B_R}ds<\infty , \label{sufficient condition for sDE}
\end{equation}
where $v_{\eps}$ is the smooth solution of the backward sPDEs~\eqref{backward SPDE 1} corresponding to~$b_{\eps}$ and~$v_{0}$, with $c_{\eps}=0$. Then path-by-path uniqueness holds for~\eqref{SDE}.
\end{corollary}

In Section~\ref{Section W^2} we have proved (reformulated for the backward sPDE) that, for every $t_{f}\in [0,T]  $, $R>0$, $m$ an even positive integer and $v_{0}\in C_{c}^{\infty}(\mathbb{R}^{d})$,
\begin{equation*}
\sup_{\eps\in (0,1)  }\sup_{[0,T]  }E \Big[\left\Vert v_{\eps}\right\Vert _{W^{2,m}( B_{R})} 
^{m}\Big]  <\infty .
\end{equation*}
Therefore, by Sobolev embedding, we obtain for $m >d$ 
\begin{equation}
\sup_{\eps\in (0,1)}\sup_{[0,T]  }E \Big[ \left\Vert \nabla v_{\eps}\right\Vert _{0,B_R}^{m}\Big]  <\infty, \label{uniform estimate grad u} 
\end{equation}
which implies condition~\eqref{sufficient condition for sDE} of Corollary~\ref{Coroll sufficient for sDE}. Hence, we have:

\begin{corollary}
Under the conditions of Section~\ref{Section W^2} \textup{(}$Db$ of class LPS\textup{)} we have path-by-path uniqueness for~\eqref{SDE}.
\end{corollary}

Notice that the conditions of Section~\ref{Section W^2} with $m>d$ imply $b\in C_{\loc}^{\varepsilon}(\mathbb{R}^{d},\mathbb{R}^{d})$ for some $\varepsilon>0$. Thus, in the case $m>d$, this result is included in Corollary~\ref{Corollary sDE Holder} below and already in~\cite{DAVIE07}. However, also the limit case $m=d\geq3$ is included in our statement. Otherwise, we may take estimate~\eqref{uniform estimate grad u} from
\cite{FLAGUBPRI10} in the case
\begin{equation}
b\in L^{\infty}(0,T;C_{b}^{\alpha}(\mathbb{R}^{d}));  \label{assumption C^alpha} 
\end{equation}
(essential boundedness in time, with values in $C_{b}^{\alpha}(\mathbb{R}^{d})$, is actually enough since the measure solutions to the continuity equation are only space-valued). Precisely, the following result is proved in~\cite{FLAGUBPRI10}. We give here an independent proof for the sake of completeness.

\begin{lemma}
\label{lemma_exp_flow}
Let~$b$ satisfy~\eqref{assumption C^alpha} and take a family $b_{\eps}\in C_{c}^{\infty}([0,T] \times\mathbb{R}^{d})$ which converges uniformly to~$b$ on compact sets. Then the flows $\Phi^{\eps}_t$ associated to~\eqref{SDE} with coefficients $b_{\eps}$ satisfy for every $m\geq 1$
\begin{equation}
\sup_{\eps\in (0,1)  }\sup_{[0,T]  }E\Big[ \Vert D\Phi^{\eps}_t \Vert_{\infty,B_R
}^{m} \Big]  <\infty .\label{flow estimate} 
\end{equation}
\end{lemma}

\begin{proof}
\emph{Step 1: Formula for $D \Phi^{\eps}_t(x)$ via It\^o--Tanaka trick}. Let us introduce the vector field $U_{\eps}(t,x)  $, for $t\in [0,T]$, $x\in\mathbb{R}^{d}$, and with components $U_{\eps}^{i}(t,x)$, for $i=1,\ldots,d$, satisfying the backward parabolic equation (where $b_{\eps}^{i}$ is the $i$-component of
$b_{\eps}$)
\begin{equation}
\partial_t U_{\eps}^{i} + b_{\eps}\cdot DU_{\eps
}^{i}+\frac{\sigma^2}{2}\Delta U_{\eps}^{i}=-b_{\eps}^{i}+\lambda
U_{\eps}^{i},\qquad U_{\eps}^{i}(  T,x)
=0 \label{PDE for U} 
\end{equation}
for some $\lambda >0$. As explained in~\cite[Section~2]{FLAGUBPRI10} based on classical results of~\cite{KRYLOV96} (see also a partial probabilistic proof in~\cite{FLALNM}), this equation has a unique solution~$U_{\eps}^{i}$ of class $C^{1}( [0,T] ;C_{b}^{\alpha}(\mathbb{R}^{d}))  \cap C([0,T];C_{b}^{2,\alpha}(\mathbb{R}^{d}))$, and there is a uniform constant $C>0$ such that
\begin{equation}
\sup_{\eps \in (0,1)} \sup_{[0,T]  }\left\Vert U_{\eps}\right\Vert _{C_{b} 
^{2,\alpha}(  \mathbb{R}^{d})  }\leq
C .\label{uniform estimate on U} 
\end{equation}
Moreover, given any $\delta>0$, there exists $\lambda>0$ large enough such that
\begin{equation}
\left\Vert DU_{\eps}\right\Vert _{\infty}\leq\delta .\label{uniform estimate on U 1/2} 
\end{equation}
Here and below we denote by $\left\Vert \cdot\right\Vert_{\infty}$ the $L^{\infty}$ norm both in time and space. We may apply the It\^{o} formula to $U_{\eps}^{i}(t,\Phi^{\eps}_t(x)  )$ and use~\eqref{PDE for U} to get \begin{equation*}
U_{\eps}^{i} ( t,\Phi^{\eps}_t(x)) =U_{\eps}^{i}(  0,x)  +\int_{0}^{t}(  -b^{i} + \lambda U_{\eps}^{i})  (  s,\Phi^{\eps}_s(x)) ds+\sigma\int_{0}^{t}\nabla U_{\eps}^{i}(  s,\Phi
^{\eps}_s(x) ) \cdot dW_{s} .
\end{equation*}
This allows us to rewrite the equation
\[
\Phi_{t}^{\eps,i}(x)  =x^{i}+\int_{0}^{t}b_{\eps}^{i}(s,\Phi^{\eps}_s(x))  ds+\sigma W_{t}^{i} 
\]
in the form
\begin{equation*}
\Phi^{\eps,i}_{t}(x) =x^{i}+U_{\eps}^{i}(0,x)  -U_{\eps}^{i}(  t,\Phi^{\eps}_t(x) ) + \int_{0}^{t}\lambda U_{\eps}^{i}(s,\Phi^{\eps}_s(x))  ds
+ \sigma\int_{0}^{t}\nabla U_{\eps}^{i}(s,\Phi^{\eps}_s(x)  )  \cdot dW_{s}+\sigma W_{t}^{i} .
\end{equation*}
Since $b_{\eps}$ is smooth and compactly supported, we a priori know from~\cite{KUNITA84} that $\Phi^{\eps}_t$ is differentiable; hence we may use the differentiability properties of $U_{\eps}$ and the result of differentiation under stochastic integral of~\cite{KUNITA84} to have
\begin{align}
\label{formula_Dvarphi}
\partial_{k}\Phi^{\eps,i}_{t}(x) & =\delta_{ik}+\partial_{k}U_{\eps}^{i}(0,x)  -\sum_{j=1}^{d} \partial_{j}U_{\eps}^{i}(t,\Phi^{\eps}_t(x)) 
\partial_{k}\Phi^{\eps,j}_{t}(x) \nonumber \\
& \quad +\int_{0}^{t}\lambda\sum_{j=1}^{d}\partial_{j}U_{\eps}^{i}(s,\Phi^{\eps}_s(x)) \partial_{k}\Phi^{\eps,j}_{s}(x) ds \nonumber \\
& \quad +\sigma\int_{0}^{t}\sum_{j,l=1}^{d}\partial_{l}\partial_{j}U_{\eps}^{i}(s,\Phi^{\eps}_s(x))  \partial_{k}\Phi^{\eps,j}_{s}(x)  dW_{s}^{l}. 
\end{align}

\emph{Step 2: Uniform pointwise estimate for $D\Phi^{\eps}_t(x)$}. We first use the previous identity to estimate $E[\vert \partial_{k}\Phi^{\eps,i}_{t}(x) \vert ^{r}]
$ uniformly in $(  t,x)  $ and~$\eps$, for each $r>1$. Denoting by $C_{r}>0$ a generic constant depending only on $r$, we have
\begin{align*}
E[ \vert \partial_{k}\Phi^{\eps}_{t}(x) \vert^{r}] & \leq C_{r}+C_{r} \Vert DU_{\varepsilon} \Vert_{\infty}^{r}+C_{r} \Vert DU_{\varepsilon} \Vert
_{\infty}^{r} E [ \vert \partial_{k}\Phi^{\eps}_{t}(x) \vert^{r} ]  \\
 & \quad +\lambda^{r}C_{r} \Vert DU_{\varepsilon} \Vert _{\infty}^{r}
\int_{0}^{t}E [ \vert \partial_{k}\Phi^{\eps}_{s}(x) \vert^{r}]  ds \\
 & \quad +\sigma^{r}C_{r} \Vert D^{2}U_{\varepsilon} \Vert _{\infty}^{r}
\int_{0}^{t} E[ \vert \partial_{k}\Phi^{\eps}_{s}(x) \vert^{r}]  ds ,
\end{align*}
where we have used the Burkholder--Davis--Gundy inequality in the last term. By choosing~$\lambda$ so large that $C_{r}\left\Vert DU_{\varepsilon}\right\Vert _{\infty}^{r}\leq 1/2$ (possible by~\eqref{uniform estimate on U 1/2}), we find
\begin{equation*}
\frac{1}{2} E[ \vert \partial_{k}\Phi^{\eps}_{t}(x) \vert^{r}] \leq C_{r} 
 +\lambda^{r} \int_{0}^{t}E [ \vert \partial_{k}\Phi^{\eps}_{s}(x) \vert^{r}] ds
 +\sigma^{r} C_{r} \Vert D^{2}U_{\varepsilon}\Vert _{\infty}^{r}
\int_{0}^{t} E[ \vert \partial_{k}\Phi^{\eps}_{s}(x) \vert^{r} ]  ds .
\end{equation*}
Now it is sufficient to apply Gronwall's lemma and the uniform estimate~\eqref{uniform estimate on U} to arrive at
\begin{equation}
\sup_{\eps\in (0,1)  }\sup_{t\in [0,T]  } 
\sup_{x\in\mathbb{R}^{d}}E [ \vert D\Phi^{\eps}_{t}(x) \vert^{r}] <\infty .\label{uniform estimate on phi} 
\end{equation}

\emph{Step 3: Conclusion via Kolmogorov's regularity criterion}. To get the supremum in~$x$ inside the expectation, we want to apply the Kolmogorov regularity criterion. Given $x,y\in\mathbb{R}^{d}$, $r>1$, we derive from~\eqref{formula_Dvarphi} (using suitable vector notations)
\begin{equation*}
E[  \vert \partial_{k}\Phi^{\eps}_{t}(x)
-\partial_{k}\Phi^{\eps}_{t}(  y)  \vert ^{r}]
\leq C_{r}(  I_{1}+I_{21}+I_{22}+\lambda^{r}I_{31}+\lambda^{r} 
I_{32}+\sigma^{r}I_{41}+\sigma^{r}I_{42})
\end{equation*}
with the following abbreviations
\begin{align*}
I_{1} & \coloneqq \vert \partial_{k}U_{\eps}(  0,x)  -\partial
_{k}U_{\eps}(  0,y)  \vert ^{r} \\[0.1cm]
I_{21}  & \coloneqq  E[  \vert DU_{\eps}(  t,\Phi_{\eps
,t}(x)  )  \vert ^{r}\vert \partial_{k} 
\Phi^{\eps}_{t}(x)  -\partial_{k}\Phi^{\eps}_{t}(
y)  \vert ^{r}]  \\[0.1cm]
I_{22}  & \coloneqq  E[  \vert DU_{\eps}(  t,\Phi_{\eps
,t}(x)  )  -DU_{\eps}(  t,\Phi_{\eps
,t}(  y)  )  \vert ^{r}\vert \partial_{k}
\Phi^{\eps}_{t}(  y)  \vert ^{r}] \\
I_{31}  & \coloneqq \int_{0}^{t}E[  \vert DU_{\eps}(  s,\Phi
^{\eps}_{t}(x)  )  \vert ^{r}\vert
\partial_{k}\Phi^{\eps}_{s}(x)  -\partial_{k}\Phi
^{\eps}_{s}(  y)  \vert ^{r}]  ds\\
I_{32}  & \coloneqq  \int_{0}^{t}E[  \vert DU_{\eps}(  s,\Phi
^{\eps}_{s}(x)  )  -DU_{\eps}(  s,\Phi
^{\eps}_{s}(  y)  )  \vert ^{r}\vert
\partial_{k}\Phi^{\eps}_{s}(  y)  \vert ^{r}]  ds \\
I_{41}  & \coloneqq \int_{0}^{t}E[  \vert D^{2}U_{\eps}(
s,\Phi^{\eps}_{t}(x)  )  \vert ^{r}\vert
\partial_{k}\Phi^{\eps}_{s}(x)  -\partial_{k}\Phi
^{\eps}_{s}(  y)  \vert ^{r}]  ds\\
I_{42}  & \coloneqq \int_{0}^{t}E\big[  \vert D^{2}U_{\eps}(
s,\Phi^{\eps}_{s}(x)  )  -D^{2}U_{\eps}(
s,\Phi^{\eps}_{s}(  y)  )  \vert ^{r}\vert
\partial_{k}\Phi^{\eps}_{s}(  y)  \vert ^{r}\big]  ds .
\end{align*}
For the last term we have used again the Burkholder--Davis--Gundy inequality. Let us denote by $C_{U}>0$ (resp.\ $C_{r,\Phi}>0$) a constant independent of $\eps \in (0,1)$, based on the uniform estimate~\eqref{uniform estimate on U} (resp.\ on~\eqref{uniform estimate on phi}) and let us write $\delta>0$ for the constant in~\eqref{uniform estimate on U 1/2}. We have
\begin{align*}
I_{1} & \leq\Vert DU_{\varepsilon}\Vert _{\infty}^{r}\vert
x-y\vert ^{r}\leq C_{U}^{r}\vert x-y\vert ^{r} \\[0.1cm]
I_{21} & \leq\Vert DU_{\varepsilon}\Vert _{\infty}^{r}E[
\vert \partial_{k}\Phi^{\eps}_{t}(x)  -\partial
_{k}\Phi^{\eps}_{t}(  y)  \vert ^{r}]  \leq
\delta^{r}E[  \vert \partial_{k}\Phi^{\eps}_{t}(
x)  -\partial_{k}\Phi^{\eps}_{t}(  y)  \vert
^{r}] \\
I_{22} & \leq\Vert D^{2}U_{\varepsilon}\Vert _{\infty}^{r}E \Big[
\int_{0}^{1}\vert D\Phi^{\eps}_{t}(  \theta x+(  1-\theta)
y)  \vert ^{r}d\theta\vert \partial_{k}\Phi^{\eps}_{t}(
y)  \vert ^{r}\Big]  \vert x-y\vert ^{r}\\
& \leq\Vert D^{2}U_{\varepsilon}\Vert _{\infty}^{r}E \Big[
\vert \partial_{k}\Phi^{\eps}_{t}(  y)  \vert
^{2r}\Big]  ^{1/2} \Big(  \int_{0}^{1}E \big[ \vert D\Phi
^{\eps}_{t}(  ux+(  1-u)  y)  \vert
^{2r} \big]  du\Big)  ^{1/2}\vert x-y\vert ^{r}\\
& \leq C_{U}^{r}C_{2r,\Phi}^{1/2}C_{2r,\Phi}^{1/2}\vert
x-y\vert ^{r}=C_{U}^{r}C_{2r,\Phi}\vert x-y\vert ^{r} .
\end{align*}
Similarly, we get
\begin{align*}
I_{31} & \leq C_{U}^{r}\int_{0}^{t}E[  \vert \partial_{k} 
\Phi^{\eps}_{s}(x)  -\partial_{k}\Phi^{\eps}_{s}(
y)  \vert ^{r}]  ds \\
I_{32} & \leq TC_{U}^{r}C_{2r,\Phi}\vert x-y\vert ^{r} 
\end{align*}
and finally
\begin{align*}
I_{41} & \leq C_{U}^{r}\int_{0}^{t}E \big[  \vert \partial_{k}\Phi
^{\eps}_{s}(x)  -\partial_{k}\Phi^{\eps}_{s}(
y)  \vert ^{r}\big]  ds \\
I_{42}  & \leq\sup_{[0,T]  }\Vert D^{2}U_{\varepsilon
}\Vert _{C^{\alpha}}^{r}\int_{0}^{t}E \Big[  \int_{0}^{1}\vert
D\Phi^{\eps}_{s}(  ux+(  1-u)  y)  \vert
^{\alpha r}du\vert \partial_{k}\Phi^{\eps}_{s}(  y)
\vert ^{r}\Big]  ds\vert x-y\vert ^{\alpha r}\\
& \leq TC_{U}^{r}C_{2\alpha r,\Phi}^{1/2}C_{2r,\Phi}^{1/2}\vert
x-y\vert ^{\alpha r}.
\end{align*}
Taking $\delta$ sufficiently small (and thus~$\lambda$ large enough), from Gronwall's lemma we deduce
\begin{equation*}
E\big[  \vert \partial_{k}\Phi^{\eps}_{t}(x)
-\partial_{k}\Phi^{\eps}_{t}(  y)  \vert ^{r}\big]
\leq C_{r}\vert x-y\vert ^{\alpha r}.
\end{equation*}
Since $r>1$ is arbitrary, we may apply Kolmogorov's regularity criterion (see for instance the quantitative version of~\cite{KUNITA84} for the bound on the moments of supremum norm in~$x$) and entail~\eqref{flow estimate}, which finishes the proof of the lemma.
\end{proof}

As a straight-forward consequence of Theorem~\ref{Theo sufficient for sDE} in combination with Lemma~\ref{lemma_exp_flow}, it follows:

\begin{corollary}
\label{Corollary sDE Holder}Under condition~\eqref{assumption C^alpha} we have path-by-path uniqueness for~\eqref{SDE}.
\end{corollary}

\begin{proof}
In view of the formula $v_{\eps}(t,x)  =v_{0}(\Phi^{\eps}_t(x))$ and~\eqref{flow estimate}, the estimate~\eqref{uniform estimate grad u} holds, which in turn implies~\eqref{main condition for sDE}. Hence, the path-by-path uniqueness follows immediately from Theorem~\ref{Theo sufficient for sDE}.
\end{proof}

\section{Examples and counterexamples}\label{ex_section}

In this final section we present some examples of drifts under LPS conditions, which exhibit regularization by noise phenomena (i.e.~the ODE is ill-posed, while the sDE is well-posed), and an example outside of the LPS conditions where our results do not hold.

\subsection{Examples of regularization by noise}

\begin{example}
Given a real number $\alpha$, we consider on $\mathbb{R}^d$ the autonomous vector field
\begin{equation*}
b(x)\coloneqq  \1_{0<|x|\le1}|x|^{\alpha}\hat{x}+ \1_{|x|>1}x,
\end{equation*}
where $\hat{x}=x/|x|$ for~$x \neq 0$ \textup{(}and $\hat{0} = 0$\textup{)}. The vector field~$b$ is Lipschitz continuous if and only if $\alpha\ge1$ and it satisfies the LPS conditions if and only if $\alpha>-1$. In the deterministic case, if $-1<\alpha<1$, we see that:
\begin{itemize}
\item if $x_0\neq0$, then there exists a unique solution~$Y$ to the ODE $dX_t/dt=b(t,X_t)$ \textup{(}that is~\eqref{SDE} with $\sigma = 0$\textup{)} starting from $x_0$, namely 
\begin{equation*}
Y(t)=(|x_0|^{1-\alpha}+(1-\alpha)t)^{1/(1-\alpha)}\hat{x}_0 \1_{t\le t_1}+e^{t-t_1}\hat{x}_0 \1_{t>t_1}, 
\end{equation*}
where $t_1$ is the first time that $|Y|=1$ \textup{(}$t_1=0$ if $|x_0|>1$\textup{)};
\item if $x_0=0$, then there is an infinite number of solutions to the ODE starting from $0$, namely any function of the form
\begin{equation*}
Y(t)= \1_{t>t_a}((1-\alpha)(t-t_0))^{1/(1-\alpha)}\hat{x}_a  \1_{t\le t_1}+e^{t-t_1}\hat{x}_a  \1_{t>t_1}
\end{equation*}
for some $t_a$ in $[0,\infty]$ \textup{(}for $t_a=\infty$, we find the null solution\textup{)} and some $x_{a}$ in the unit sphere~$\mathbb{S}^{d-1}$ \textup{(}and with $t_1$ as before\textup{)}.
\end{itemize}

\noindent\begin{minipage}{0.97\linewidth}
\begin{wrapfigure}{r}[0.4cm]{0.44\textwidth}
\includegraphics[width=0.44\textwidth]{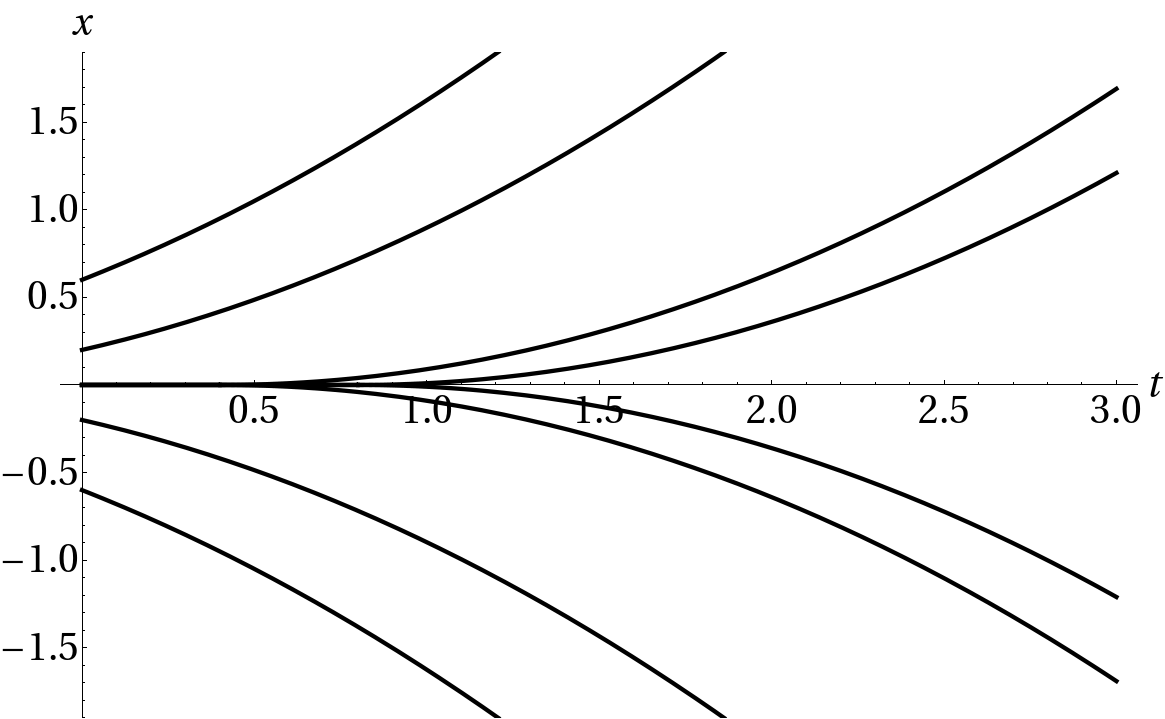}
\caption{Non-uniqueness of trajectories}
\end{wrapfigure}

\hspace{15pt}Consequently, if $-1<\alpha<1$, there cannot exist a continuous flow solving the ODE: continuity fails in $x_0=0$. This also implies that non-uniqueness occurs for the transport equation with discontinuous \textup{(}at $0$\textup{)} initial datum.

\hspace{15pt}On the contrary, for $\alpha>-1$ the drift is in the LPS class; hence, all our results apply to this example and show regularization by noise in various ways. In particular, the discontinuity of the flow in the origin is removed~$\omega$-wise; moreover, for $\alpha>0$ \textup{(}when~$b$ is H\"older continuous\textup{)}, we have proved path-by-path uniqueness starting from $x_0=0$. Let us also remind that pathwise uniqueness holds, starting from~$0$, by~\textup{\cite{KRYROE05}}. 

\end{minipage}
\\[0.1cm]

In this particular example it is possible to get an intuitive idea, ``by hands'', of what happens. If one consider the ODE without noise starting from~$0$, any solution~$Y$ grows near~$0$ no faster than $t^{1/(1-\alpha)}$; on the contrary, the Brownian motion~$W$ near~$0$ grows as~$t^{1/2}$ \textup{(}up to a logarithmic correction, which does not affect the intuition\textup{)}. Heuristically, we could say that the ``speed'' of~$Y$ near~$0$ caused by the drift is like  $t^{\alpha/(1-\alpha)}$, while the one caused by~$W$ is like $t^{-1/2}$. So what we expect to happen is that the Brownian motion moves the particle immediately away from~$0$, faster than the action of the drift, and this prevents the formation of non-uniqueness or singularities. At least in the one-dimensional case, this can be seen also through speed measure and scale function, see~\textup{\cite{BREIMAN}}.

Notice that if $\alpha<-1$ the opposite phenomenon appears \textup{(}$Y$ is faster than $W$\textup{)}, so that we expect ill-posedness. This is also an argument for the optimality of the LPS conditions \textup{(}even if, in this case, we do not look at critical cases in LPS hypotheses\textup{)}, see also Example~\ref{counterex}.
\end{example}

\begin{example}
We consider a similar autonomous vector field as in the previous example, but change sign:
\begin{equation*}
b(x)\coloneqq - \1_{0<|x|\le1}|x|^{\alpha}\hat{x}- \1_{|x|>1}x.
\end{equation*}
In this case, we see that, for every initial $x_0$, there exists a unique solution~$Y$ to the ODE, which reaches~$0$ in finite time and then stays in~$0$. Thus, concentration happens in~$0$, so that there does not exist a Lagrangian flow \textup{(}the image measure of the flow at time~$t$ can have a Dirac delta in~$0$\textup{)}. Moreover the solution to the continuity equation concentrates in~$0$. Again our results apply when $\alpha>-1$, so these concentration phenomena disappear in the stochastic case.
\end{example}

\begin{example}
Take $d=2$ for simplicity. The following field is a combination of the previous two examples:
\begin{equation*}
b(x)= \1_{x\in A} \big[ \1_{0<|x|\le1}|x|^{\alpha}\hat{x}+ \1_{|x|>1}x \big]+ \1_{x\in A^c} \big[- \1_{0<|x|\le1}|x|^{\alpha}\hat{x}- \1_{|x|>1}x\big],
\end{equation*}
where $A=\{x\in\mathbb{R}^2 \colon x_1>0\mbox{ or }(x_1=0,x_2>0)\}$. It is easy to see that, for $\alpha<1$, in the deterministic case one can construct flows with discontinuity, concentration of the mass in $0$ or both; in particular, non-uniqueness holds. Again, for $\alpha>-1$, well-posedness \textup{(}as in Theorem~\ref{main thm flows}\textup{)} is restored.
\end{example}

\subsection{A counterexample in the supercritical case}

Finally let us show that outside the LPS class there are equations and diffuse initial conditions without any solution; in particular the statement of Theorem~\ref{2pathwise} does not hold in this case.

\begin{example}\label{counterex}
 We now consider equation~\eqref{SDE} on $\mathbb{R}^d$, with $\sigma = 1$ and drift~$b$ defined as
\begin{equation*}
b(x) \coloneqq -\beta |x|^{-2} x  \1_{x\neq0} ,
\end{equation*}
with $\beta>d/2$. Notice that, for $d\ge 2$, this drift is just outside the LPS class \textup{(}in the sense that $|x|^{\alpha-1}x$ belongs to that class for any $\alpha>-1$\textup{)}. For this particular sDE, we have: for some $T>0$ and $M>0$, if~$X_0$ is a random variable, independent of $W$ and uniformly distributed on $[-M,M]$, then there does not exist a weak solution, starting from~$X_0$.
\end{example}

\begin{proof}
\emph{Step 1:~\eqref{SDE} does not have a weak solution for $X_0=0$ \textup{(}for any $T>0$\textup{)}.} Notice that this does not prevent from extending Theorem~\ref{2pathwise} to this case (because the initial datum is concentrated on~$0$), but it is a first step. The method is taken from~\cite{CHEENG05}.

Assume, by contradiction, that a weak solution on $[0,T]$ exists, i.e.~there is a filtered probability space $(\Omega,\mathcal{A},\mathcal{G}_{t},P)$ (satisfying standard assumptions), a Brownian motion $W$ in $\mathbb{R}^{d}$ with respect to $(\mathcal{G}_{t})_t$, an $(\mathcal{G}_{t})_t$-adapted continuous process $(X_{t})_{t\geq0}$ in
$\mathbb{R}^{d}$, such that $ \int_{0}^{T} \vert b(X_{t}) \vert dt<\infty$ and, a.s.,
\begin{equation*}
X_{t}=\int_{0}^{t}b(X_{s}) ds+W_{t} .
\end{equation*}
Hence,~$X$ is a continuous semimartingale, with quadratic covariation $\left\langle X^{i},X^{j}\right\rangle _{t}=\delta_{ij}t$. By the It\^{o} formula, we have
\begin{align*}
d\left\vert X_{t}\right\vert ^{2}  & =-2\beta \1_{X_{t}\neq0}dt+2X_{t}\cdot
dW_{t}+d\cdot dt\\
& =(d \1_{X_{t}=0}-(2\beta-d) \1_{X_{t}\neq0})dt+2X_{t}\cdot dW_{t}.
\end{align*}
We now claim that
\begin{equation}
\int_{0}^{T} \1_{X_{s}=0}ds=0\label{key} 
\end{equation}
holds with probability one. This implies
\begin{equation*}
\vert X_{t} \vert ^{2}=-(2\beta-d)\int^t_0 \1_{X_{s}\neq0}ds+\int_{0}^{t}2X_{s}\cdot dW_{s}.
\end{equation*}
Therefore, $\vert X_{t} \vert^{2}$ is a positive local supermartingale, vanishing at $t=0$. This implies $\vert X_{t} \vert^{2}\equiv0$, hence $X_{t}\equiv0$, which contradicts the fact that $\left\langle X^{i},X^{j}\right\rangle _{t}=\delta_{ij}t$. 

It remains to prove the claim~\eqref{key}. Consider the random set $\{ t\in [0,T] \colon X_{t}=0\}$. Since it is a subset of $A_{1}=\{ t\in [0,T] \colon X_{t}^{1}=0 \}$, it is sufficient to prove that $A_{1}$ is of Lebesgue measure zero, $P$-a.s.~and this is equivalent to $P\big(\int_{0}^{T}  \1_{X_{s}^{i}=0}ds=0 \big)=1$. Since~$X$ is a continuous semimartingale with quadratic covariation $\left\langle X^{i},X^{j}\right\rangle _{t}=\delta_{ij}t$, also $X^{1}$ is a continuous semimartingale, with quadratic covariation $\left\langle X^{1},X^{1} \right\rangle _{t}=t$. Hence, by the occupation times formula (see
\cite[Chapter~VI, Corollary~1.6]{REVYOR94})
\begin{equation*}
\int_{0}^{T} \1_{X_{s}^{1}=0}ds=\int_{\mathbb{R}} \1_{a=0}L_{T}^{a}(X^{1})  da
\end{equation*}
where $L_{T}^{a}( X^{1})$ is the local time at $a$ on $[
0,T]  $ of the process $X^{1}$. Hence, a.s., we have $\int_{0}^{T} \1_{X_{s}^{1}=0}ds=0$. 

\emph{Step 2:~\eqref{SDE} does not have a weak solution starting from~$X_0$ uniformly distributed on $[-M,M]$ \textup{(}for some $T>0$ and $M>0$\textup{)}.} Again, we suppose by contradiction that there exists such a solution~$X$ (associated with some filtration $(\mc{G}_t)_t$), on a probability space $(\Omega,\mc{A},P)$. Let~$\tau$ be the first time when~$X$ hits~$0$ (which is a stopping time with respect to~$(\mc{G}_t)_t$), with $\tau=\infty$ when~$X$ does not hit~$0$. We now claim that
\begin{equation}
P(\tau<\infty)>0 .\label{claim'}
\end{equation}
Assuming this, we can construct a new process~$Y$, which is a weak solution to~\eqref{SDE}, starting from $Y_0=0$. This is in contradiction with Step 1. The process~$Y$ is built as follows. Take $\td{\Omega}=\{\tau<\infty\}$, $\td{\mc{A}}=\{A\cap\td{\Omega} \colon A\in\mc{A}\}$, $Q= P(\td{\Omega})^{-1} P|_{\td{\mc{A}}}$, then define $Y_t \coloneqq X_{t+\tau}$, $\td{W}_t \coloneqq W_{t+\tau}-W_\tau$ on $\td{\Omega}$ and $\mc{H}_t=\sigma(\{\td{W}_s,Y_s|s\le t\}\cup\tilde{\mathcal{N}})$ ($\sigma$-algebra on $\td{\Omega}$), where $\tilde{\mathcal{N}}$ is the set of $Q$-null sets of $\tilde{\Omega}$. We observe the following facts:

\begin{itemize}
  \item $\td{W}$ is a natural Brownian motion on the space $(\td{\Omega},\mc{A},Q)$, i.e., for every positive integer $n$, for every $0<t_1<\ldots < t_n$ and for every $f_1,\ldots, f_n$ in $C_b(\mathbb{R}^d)$, there holds 
  \begin{equation}
  E\Big[ \1_{\td{\Omega}}\prod^n_{j=1}f_j(W(t_j+\tau)-W(t_{j-1}+\tau))\Big]= P(\td{\Omega})\prod^n_{j=1}\int_{\mathbb{R}^d}f_jd\mc{N}(0,(t_j-t_{j-1})I),\label{newBM}
  \end{equation}
  where $\mc{N}(m,A)$ is the Gaussian law of mean $m$ and covariance matrix $A$. This can be verified, for a general $\mc{G}$-stopping time, with a standard argument: first one proves~\eqref{newBM} when~$\tau$ is a stopping time with discrete range in $[0,\infty]$, then, for the general case, one uses an approximation of~$\tau$ with stopping times $\tau_k$ with discrete range such that $\tau_k\downarrow\tau$ (as $k\rightarrow\infty$) and $\{\tau=\infty\}=\{\tau_k=\infty\}$ for every~$k$.
  \item $\td{W}$ is a Brownian motion with respect to the filtration $\mc{H}$, i.e., for every $0=t_0<t_1<\ldots < t_n\le s<t$ and for every $f,g_1,\ldots, g_n$ in $C_b(\mathbb{R}^d)$, there holds 
  \begin{multline*}
  E\Big[ \1_{\td{\Omega}}f(W(t+\tau)-W(s+\tau))\prod^n_{j=0}g_j(X(t_j+\tau))\Big] \\
  = \int_{\mathbb{R}^d}fd\mc{N}(0,(t-s)I)E\Big[\prod^n_{j=0}g_j(X(t_j+\tau))\Big].
  \end{multline*}
  Again this can be shown by approximation (with stopping times with discrete range).
  \item $Y$ is a weak solution to~\eqref{SDE}, starting from $Y_0 = 0$. This follows immediately from
  \begin{equation*}
  X_{s'}=X_s+\int^{s'}_sb(X_r)dr+W_{s'}-W_s ,
  \end{equation*}
  by setting $s'=t+\tau$ and $s=\tau$.
\end{itemize}

It remains to prove claim~\eqref{claim'}. We suppose by contradiction that $\tau=\infty$ holds a.s.; this implies that, for every $t \in [0,T]$, we have $P(X_t\neq0)=1$. Then, computing $E[|X|^2]$ by the It\^o formula, we get
\begin{equation*}
\frac{d}{dt}E[|X_t|^2]=-2\beta P(X_t\neq0)+d=-2\beta+d<0,
\end{equation*}
hence, there exists a time $t_0>0$ with $E[|X_{t_0}|^2]<0$, which is a contradiction. This completes the proof.
\end{proof}

\begin{remark}
The restriction $\beta>d/2$ is due to the first step. The claim~\eqref{claim'} holds in fact for the more general case $\beta>(d-2)/2$, which can be achieved by an alternative approach.
\end{remark}

\begin{proof}[Sketch of proof]
The idea is that the symmetry properties of the given drift~$b$ allow to reduce the solution~$X$ of the sDE to a one-dimensional Bessel process, for which the probability of hitting~$0$ is known. Fix $R>0$ and, for any $\eps>0$,~$x \in B_R\setminus\bar{B}_\eps$, and denote by $\tau_\eps(x)$ the exit time from the annulus $B_R\setminus\bar{B}_\eps$ of the solution $Z(x)$ to~\eqref{SDE}, starting from~$x$ (note that~$Z$ exists up to $\tau_\eps(x)$ since the drift is regular in the annulus). If we prove that, for every $x\neq0$, there exist $T>0$, $\delta>0$ and $R$ large enough such that, for every $\eps>0$,
\begin{equation}
P\big( \tau_\eps(x)<T,Z(x,\tau_\eps(x))=\eps\big)>\delta ,\label{claim2}
\end{equation}
then we have shown the claim~\eqref{claim'}. In order to prove~\eqref{claim2} we notice that we have $P(\tau_\eps(x)<T,Z(x,\tau_\eps(x))=\eps)=u(0,x)$, where~$u$ solves the backward parabolic PDE, on $B_R\setminus\bar{B}_\eps$,
\begin{equation*}
\partial_tu+b\cdot\nabla u+\frac12\Delta u=0 ,
\end{equation*}
with final and boundary conditions
\begin{equation*}
u(T,\cdot)\equiv0 \text{ in } B_R\setminus\bar{B}_\eps,\ \ u(t,\cdot) \equiv 1 \text{ on } \partial B_\eps \text{ and } 
u(t,\cdot)\equiv 0 \text{ on } \partial B_R, \text{ for all } t \in [0,T].
\end{equation*}
By the symmetry properties of the drift~$b$, the solution~$u$ is given by $u(t,x)=v(t,|x|)$, where $v \colon [0,T]\times(\eps,R)\rightarrow\mathbb{R}$ solves the PDE
\begin{equation*}
\partial_t v+\bar{b}\cdot\nabla v+\frac12\Delta v=0
\end{equation*}
for $\bar{b}(r)=(-\beta+(d-1)/2)|r|^{-1} \1_{r\neq0}$, with final and boundary conditions
\begin{equation*}
v(T,\cdot)\equiv0,\ \ v(\cdot,\eps)\equiv1,\ \ v(\cdot,R)\equiv0 .
\end{equation*}
Then 
\begin{equation}
\label{tau_sigma_ident}
P\big(\tau_\eps(x)<T,Z(x,\tau_\eps(x))=\eps\big)=v(0,|x|)=P\big(\sigma_\eps(|x|)<T,\xi(|x|,\sigma_\eps(|x|))=\eps\big) , 
\end{equation}
where $\xi=\xi(r)$ is the one-dimensional process solution to the sDE
\begin{equation*}
d\xi= \Big(-\beta+\frac{d-1}{2} \Big) |\xi|^{-1} \1_{\xi\neq0} dt + dB_t
\end{equation*}
(where $B$ is a one-dimensional Brownian motion), with initial condition $\xi_0(r)=r$, and $\sigma_\eps(r)$ is the exit time of $\xi(r)$ from the interval $(\eps,R)$. Now standard tools of one-dimensional diffusion processes (speed measure and scale function, see~\cite[Chapter~16]{BREIMAN}) allow to deduce that, since $\beta>(d-2)/2$, for every $r>0$, $\xi(r)$ hits $0$ in finite time with positive probability. This implies that, for every $r>0$, there exist $T>0$, $\delta>0$ and $R$ large enough such that, for every $\eps>0$, $P(\sigma_\eps(r)<T,\xi(r,\sigma_\eps(r))=\eps)>\delta$. In view of~\eqref{tau_sigma_ident} we have thus established~\eqref{claim2}, and the sketch of proof of the final remark is complete.
\end{proof}

%%%%%%%%%%%%%%%%%%%%%%%%%%%%%%%%%%%%%%%%%%%%%%%%%%%%%%%%%%%%%%%%%%%
%%                                                               %%
%% Use the two commands below for producing your bibliography    %%
%% with bibtex, then comment again the commands and include the  %%
%% content of the .bbl file in this file below the commands.     %%
%%                                                               %%
%%%%%%%%%%%%%%%%%%%%%%%%%%%%%%%%%%%%%%%%%%%%%%%%%%%%%%%%%%%%%%%%%%%

%\bibliographystyle{amsplain}
%\bibliography{yourbibfilename}

\begin{thebibliography}{10}

\bibitem{AIZENMAN78}
M.~Aizenman, \emph{On vector fields as generators of flows: a counterexample to {N}elson's conjecture}, Ann. Math. \textbf{107} (1978), no.~2, 287--296.

\bibitem{AMBROSIO04}
L.~Ambrosio, \emph{Transport equation and {C}auchy problem for {$BV$} vector
  fields}, Invent. Math. \textbf{158} (2004), no.~2, 227--260.

\bibitem{AMBCRI08}
L.~Ambrosio and G.~Crippa, \emph{Existence, uniqueness, stability and
  differentiability properties of the flow associated to weakly differentiable vector fields}, Transport equations and multi-{D} hyperbolic conservation laws, Lect. Notes Unione Mat. Ital., vol.~\textbf{5}, Springer, Berlin, 2008, pp.~3--57.

\bibitem{ATTFLA11}
S.~Attanasio and F.~Flandoli, \emph{Renormalized solutions for stochastic
  transport equations and the regularization by bilinear multiplicative noise}, Commun. Partial Differ. Equations \textbf{36} (2011), no.~8, 1455--1474.

\bibitem{BAHCHEDAN11}
H.~Bahouri, J.-Y.~Chemin, and R.~Danchin, \emph{Fourier analysis and nonlinear partial differential equations}, Grundlehren der Mathematischen
  Wissenschaften [Fundamental Principles of Mathematical Sciences],
  vol.~\textbf{343}, Springer, Heidelberg, 2011.

\bibitem{BALGYOPAR94}
V.~Bally, I.~Gy\"{o}ngy, and \'{E}. Pardoux, \emph{White noise driven parabolic {SPDE}s with measurable drift}, J. Funct. Anal. \textbf{120} (1994), no.~2, 484--510.

\bibitem{BARROCZHA17}
V.~Barbu, M.~R\"{o}ckner, and D.~Zhang, \emph{Stochastic nonlinear
  {S}chr\"{o}dinger equations: no blow-up in the non-conservative case}, J.
  Differential Equations \textbf{263} (2017), no.~11, 7919--7940.

\bibitem{BASS11}
R.F.~Bass, \emph{Stochastic processes}, Cambridge Series in Statistical and
  Probabilistic Mathematics, vol.~\textbf{33}, Cambridge University Press,
  Cambridge, 2011.

\bibitem{BASCHE03}
R.F.~Bass and Z.-Q.~Chen, \emph{Brownian motion with singular drift}, Ann.
  Probab. \textbf{31} (2003), no.~2, 791--817.

\bibitem{BEIRAO97}
H.~Beir\~{a}o Da~Veiga, \emph{Remarks on the smoothness of the
  {$L^\infty(0,T;L^3)$} solutions of the {$3$}-{D} {N}avier-{S}tokes
  equations}, Portugal. Math. \textbf{54} (1997), no.~4, 381--391.

\bibitem{BIA13}
L.A.~Bianchi, \emph{Uniqueness for an inviscid stochastic dyadic model on a
  tree}, Electron. Commun. Probab. \textbf{18} (2013), no.~8, 1--12.

\bibitem{BIABLOYAN16}
L.A.~Bianchi, D.~Bl\"omker, and M.~Yang, \emph{Additive noise destroys the random attractor close to bifurcation}, Nonlinearity \textbf{29} (2016), no.~12, 3934--3960.  
  
\bibitem{BOGDAPROE11}
V.~Bogachev, G.~Da~Prato, and M.~R\"{o}ckner, \emph{Uniqueness for solutions of {F}okker-{P}lanck equations on infinite dimensional spaces}, Comm. Partial Differential Equations \textbf{36} (2011), no.~6, 925--939.

\bibitem{BREIMAN}
L.~Breiman, \emph{Probability}, Classics in Applied Mathematics, vol.~\textbf{7}, Society for Industrial and Applied Mathematics (SIAM), Philadelphia, PA, 1992, Corrected reprint of the 1968 original.

\bibitem{CANCHO18}
G.~Cannizzaro and K.~Chouk, \emph{Multidimensional {SDE}s with singular drift and universal construction of the polymer measure with white noise
  potential}, Ann. Probab. \textbf{46} (2018), no.~3, 1710--1763.

\bibitem{CARCRALANROB07}  
T.~Caraballo, H.~Crauel, J.A.~Langa, and J.C.~Robinson, \emph{The effect of noise on the Chafee-Infante equation: a nonlinear case study}, Proc. Amer. Math. Soc. \textbf{135} (2007), no.~2, 373--382.  
  
\bibitem{CATELLIER16}
R.~Catellier, \emph{Rough linear transport equation with an irregular drift}, Stoch. Partial Differ. Equ. Anal. Comput. \textbf{4} (2016), no.~3, 477--534.

\bibitem{CATGUB13}
R.~Catellier and M.~Gubinelli, \emph{Averaging along irregular curves and
  regularisation of {ODE}s}, Stochastic Process. Appl. \textbf{126}
  (2016), no.~8, 2323--2366.

\bibitem{CHAGAWHORKUPVER03}
M.~Chaves, K.~Gaw\c{e}dzki, P.~Horvai, A.~Kupiainen, and M.~Vergassola,
  \emph{Lagrangian dispersion in {G}aussian self-similar velocity ensembles}, J. Statist. Phys. \textbf{113} (2003), no.~5-6, 643--692, Progress in statistical hydrodynamics (Santa Fe, NM, 2002).

\bibitem{CHEZHA95}
Z.Q.~Chen and Z.~Zhao, \emph{Diffusion processes and second order elliptic
  operators with singular coefficients for lower order terms}, Math. Ann.
  \textbf{302} (1995), no.~2, 323--357.

\bibitem{CHEENG05}
A.S.~Cherny and H.-J.~Engelbert, \emph{Singular stochastic differential
  equations}, Lecture Notes in Mathematics, vol.~\textbf{1858}, Springer-Verlag,
  Berlin, 2005.

\bibitem{CHOGES19}
K.~Chouk and B.~Gess, \emph{Path-by-path regularization by noise for scalar
  conservation laws}, J. Funct. Anal. \textbf{277} (2019), no.~5,
  1469--1498.

\bibitem{CHOGUB14}
K.~Chouk and M.~Gubinelli, \emph{Nonlinear {PDE}s with modulated dispersion
  {II}: {K}orteweg--de {V}ries equation}, arXiv:1406.767 (2014).

\bibitem{CHOGUB15}
K.~Chouk and M.~Gubinelli, \emph{Nonlinear {PDE}s with modulated dispersion {I}: {N}onlinear {S}chr\"odinger equations}, Comm. Partial Differential Equations \textbf{40} (2015), no.~11, 2047--2081.

\bibitem{CRIDEL08}
G.~Crippa and C.~De~Lellis, \emph{Estimates and regularity results for the
  {D}i{P}erna-{L}ions flow}, J. Reine Angew. Math. \textbf{616} (2008), 15--46.

\bibitem{DAVIE07}
A.M.~Davie, \emph{Uniqueness of solutions of stochastic differential
  equations}, Int. Math. Res. Not. IMRN (2007), no.~24, Art. ID rnm124, 26.

\bibitem{DEBTSU11}
A.~Debussche and Y.~Tsutsumi, \emph{1{D} quintic nonlinear {S}chr\"odinger
  equation with white noise dispersion}, J. Math. Pures Appl. (9) \textbf{96} (2011), no.~4, 363--376.

\bibitem{DELDIE16}
F.~Delarue and R.~Diel, \emph{Rough paths and 1d {SDE} with a time dependent
  distributional drift: application to polymers}, Probab. Theory Related Fields \textbf{165} (2016), no.~1-2, 1--63.

\bibitem{DELFLAVIN14}
F.~Delarue, F.~Flandoli, and D.~Vincenzi, \emph{Noise prevents collapse of
  {V}lasov-{P}oisson point charges}, Comm. Pure Appl. Math. \textbf{67} (2014), no.~10, 1700--1736.

\bibitem{DIPELI89}
R.J.~DiPerna and P.-L.~Lions, \emph{Ordinary differential equations, transport theory and {S}obolev spaces}, Invent. Math. \textbf{98} (1989), no.~3, 511--547.

\bibitem{DUBREV16}
R.~Duboscq and A.~R\'{e}veillac, \emph{Stochastic regularization effects of
  semi-martingales on random functions}, J. Math. Pures Appl. (9) \textbf{106} (2016), no.~6, 1141--1173.

\bibitem{DUBREV17}
R.~Duboscq and A.~R\'{e}veillac, \emph{On a stochastic {H}ardy--{L}ittlewood--{S}obolev inequality with application to {S}trichartz estimates for the white noise dispersion}, arXiv:1711.07188 (2017).

\bibitem{ESCSERSVE03}
L.~Escauriaza, G.~Seregin, and V.~\v{S}ver\'{a}k, \emph{Backward uniqueness for parabolic equations}, Arch. Ration. Mech. Anal. \textbf{169} (2003), no.~2, 147--157.

\bibitem{FEDFLA11}
E.~Fedrizzi and F.~Flandoli, \emph{Pathwise uniqueness and continuous
  dependence of {SDE}s with non-regular drift}, Stochastics \textbf{83}
  (2011), no.~3, 241--257.

\bibitem{FEDFLA13b}
E.~Fedrizzi and F.~Flandoli, \emph{H\"{o}lder flow and differentiability for {SDE}s with nonregular drift}, Stoch. Anal. Appl. \textbf{31} (2013), no.~4, 708--736.

\bibitem{FEDFLA13a}
E.~Fedrizzi and F.~Flandoli, \emph{Noise prevents singularities in linear transport equations}, J. Funct. Anal. \textbf{264} (2013), no.~6, 1329--1354.

\bibitem{FEDNEVOLI18}
E.~Fedrizzi, W.~Neves, and C.~Olivera, \emph{On a class of stochastic transport equations for {$L^2_{\rm loc}$} vector fields}, Ann. Sc. Norm. Super. Pisa Cl. Sci. (5) \textbf{18} (2018), no.~2, 397--419.

\bibitem{FIG08}
A.~Figalli, \emph{Existence and uniqueness of martingale solutions for {SDE}s with rough or degenerate coefficients}, J. Funct. Anal. \textbf{254}
  (2008), no.~1, 109--153.

\bibitem{FLALNM}
F.~Flandoli, \emph{Random perturbation of {PDE}s and fluid dynamic models},
  Saint Flour summer school lectures 2010, Lecture Notes in Math.,
  vol.~\textbf{2015}, Springer, Berlin, 2011, p.~176.

\bibitem{FLAGUBPRI11}
F.~Flandoli, M.~Gubinelli, and E.~Priola, \emph{Does noise improve
  well-posedness of fluid dynamic equations?}, Stochastic partial differential equations and applications, Quad. Mat., vol.~\textbf{25}, Dept. Math., Seconda
  Univ. Napoli, Caserta, 2010, pp.~139--155.

\bibitem{FLAGUBPRI10}
F.~Flandoli, M.~Gubinelli, and E.~Priola, \emph{Well-posedness of the transport equation by stochastic perturbation}, Invent. Math. \textbf{180} (2010), no.~1, 1--53.

\bibitem{FLAGUBPRI10a}
F.~Flandoli, M.~Gubinelli, and E.~Priola, \emph{{Full well-posedness of point vortex dynamics corresponding to stochastic 2D Euler equations}}, Stoch. Proc. Appl. \textbf{{\bf121}} (2011),
  no.~7, 1445--1463.

\bibitem{FLAGUBPRI12}
F.~Flandoli, M.~Gubinelli, and E.~Priola, \emph{Remarks on the stochastic transport equation with {H}\"{o}lder drift}, Rend. Semin. Mat. Univ. Politec. Torino \textbf{70} (2012),
  no.~1, 53--73.

\bibitem{FLAISSRUS17}
F.~Flandoli, E.~Issoglio, and F.~Russo, \emph{Multidimensional stochastic
  differential equations with distributional drift}, Trans. Amer. Math. Soc.
  \textbf{369} (2017), no.~3, 1665--1688.

\bibitem{FLAMAUNEK14}
F.~Flandoli, M.~Maurelli, and M.~Neklyudov, \emph{Noise prevents infinite
  stretching of the passive field in a stochastic vector advection equation}, J. Math. Fluid Mech. \textbf{16} (2014), no.~4, 805--822.

\bibitem{FRISHE13}
P.~Friz and A.~Shekhar, \emph{Doob-{M}eyer for rough paths}, Bull. Inst. Math. Acad. Sin. (N.S.) \textbf{{\bf 8}} (2013), no.~1, 73--84.

\bibitem{GALDI00}
G.P.~Galdi, \emph{An introduction to the {N}avier-{S}tokes initial-boundary
  value problem}, Fundamental directions in mathematical fluid mechanics, Adv. Math. Fluid Mech., Birkh\"{a}user, Basel, 2000, pp.~1--70.

\bibitem{GASGES19}
P.~Gassiat and B.~Gess, \emph{Regularization by noise for stochastic
  {H}amilton-{J}acobi equations}, Probab. Theory Related Fields
  \textbf{173} (2019), no.~3-4, 1063--1098.

\bibitem{GESMAU18}
B.~Gess and M.~Maurelli, \emph{Well-posedness by noise for scalar conservation laws}, Comm. Partial Differential Equations \textbf{43} (2018), no.~12, 1702--1736.

\bibitem{GESSMI19}
B.~Gess and S.~Smith, \emph{Stochastic continuity equations with conservative noise}, J. Math. Pures Appl. (9) \textbf{128} (2019), 225--263.

\bibitem{GESSOU17}
B.~Gess and P.E.~Souganidis, \emph{Long-time behavior, invariant measures, and regularizing effects for stochastic scalar conservation laws}, Comm. Pure Appl. Math. \textbf{149} (2017), no.~8, 1562--1597.

\bibitem{GYOMAR01}
I.~Gy\"{o}ngy and T.~Mart\'{\i}nez, \emph{On stochastic differential equations with locally unbounded drift}, Czechoslovak Math. J. \textbf{149} (2001), no.~4, 763--783.

\bibitem{KISLAD57}
A.A.~Kiselev and O.A.~Lady{\v{z}}enskaya, \emph{On the existence and
  uniqueness of the solution of the nonstationary problem for a viscous,
  incompressible fluid}, Izv. Akad. Nauk SSSR. Ser. Mat. \textbf{21}
  (1957), 655--680.

\bibitem{KRYLOV96}
N.V.~Krylov, \emph{Lectures on elliptic and parabolic equations in {H}\"{o}lder spaces}, Graduate Studies in Mathematics, American Mathematical Society, Providence, RI, 1996.

\bibitem{KRYROE05}
N.V.~Krylov and M.~R{\"o}ckner, \emph{Strong solutions of stochastic equations with singular time dependent drift}, Probab. Theory Related Fields
\textbf{131} (2005), no.~2, 154--196.

\bibitem{KUNITA84a}
H.~Kunita, \emph{First order stochastic partial differential equations},
  Stochastic analysis ({K}atata/{K}yoto, 1982), North-Holland Math. Library,
  vol.~\textbf{32}, North-Holland, 1984, pp.~249--269.

\bibitem{KUNITA84}
H.~Kunita, \emph{Stochastic differential equations and stochastic flows of
  diffeomorphisms}, \'{E}cole d'\'et\'e de probabilit\'es de {S}aint-{F}lour, {XII}---1982, Lecture Notes in Math., vol.~\textbf{1097}, Springer, Berlin, 1984, pp.~143--303.

\bibitem{KUNITA90}
H.~Kunita, \emph{Stochastic flows and stochastic differential equations},
  Cambridge Studies in Advanced Mathematics, vol.~\textbf{24}, Cambridge University
  Press, Cambridge, 1990.

\bibitem{LADYZ67}
O.A.~Lady{\v{z}}enskaja, \emph{Uniqueness and smoothness of generalized
  solutions of {N}avier-{S}tokes equations}, Zap. Nau\v cn. Sem. Leningrad.
  Otdel. Mat. Inst. Steklov. (LOMI) \textbf{5} (1967), 169--185.

\bibitem{LIOPRO59}
J.-L.~Lions and G.~Prodi, \emph{Un th\'eor\`eme d'existence et unicit\'e dans les \'equations de {N}avier-{S}tokes en dimension 2}, C. R. Acad. Sci. Paris \textbf{248} (1959), 3519--3521.

\bibitem{LIONS96}
P.-L.~Lions, \emph{Mathematical topics in fluid mechanics. {V}ol. 1}, Oxford
  Lecture Series in Mathematics and its Applications, vol.~\textbf{3}, The
  Clarendon Press Oxford University Press, New York, 1996.

\bibitem{MAU11}
M.~Maurelli, \emph{Wiener chaos and uniqueness for stochastic transport
  equation}, C. R. Math. Acad. Sci. Paris \textbf{349} (2011), no.~11-12,
  669--672.

\bibitem{MAURELLIPHD}
M.~Maurelli, \emph{Regularization by noise in finite dimension}, Ph.D. thesis, Scuola Normale Superiore di Pisa, 2016.

\bibitem{MOHNILPRO15}
S.-E.A.~Mohammed, T.K. Nilssen, and F.N. Proske, \emph{Sobolev differentiable stochastic flows for {SDE}s with singular coefficients: applications to the transport equation}, Ann. Probab. \textbf{43} (2015), no.~3, 1535--1576.

\bibitem{NAM18}
K.~Nam, \emph{Stochastic differential equations with critical drifts},
  arXiv:1802.00074 (2018).

\bibitem{NEVOLI15}
W.~Neves and C.~Olivera, \emph{Wellposedness for stochastic continuity equations with {L}adyzhenskaya-{P}rodi-{S}errin condition}, NoDEA Nonlinear Differential Equations Appl. \textbf{22} (2015), no.~5, 1247--1258.

\bibitem{NEVOLI16}
W.~Neves and C.~Olivera, \emph{Stochastic continuity equations---a general uniqueness result}, Bull. Braz. Math. Soc. (N.S.) \textbf{47} (2016), no.~2, 631--639.

\bibitem{NIL15}
T.~Nilssen, \emph{Rough path transport equation with discontinuous coefficient - regularization by fractional brownian motion}, arXiv:1509.01154 (2015).

\bibitem{OLIVEIRA19}
C.~Olivera, \emph{Regularization by noise in one-dimensional continuity
  equation}, Potential Anal. \textbf{51} (2019), no.~1, 23--35.

\bibitem{PORTENKO90}
N.I.~Portenko, \emph{Generalized diffusion processes}, Translations of
  Mathematical Monographs, vol.~{\bf 83}, American Mathematical Society,
  Providence, RI, 1990, Translated from the Russian by H. H. McFaden.

\bibitem{PRI18}
E.~Priola, \emph{Davie's type uniqueness for a class of {SDE}s with jumps},
  Ann. Inst. Henri Poincar\'{e} Probab. Stat. \textbf{54} (2018), no.~2, 694--725.

\bibitem{PRODI59}
G.~Prodi, \emph{Un teorema di unicit\`a per le equazioni di {N}avier-{S}tokes}, Ann. Mat. Pura Appl. (4) \textbf{48} (1959), 173--182.

\bibitem{REVYOR94}
D.~Revuz and M.~Yor, \emph{Continuous martingales and {B}rownian motion},
  second ed., Grundlehren der Mathematischen Wissenschaften [Fundamental
  Principles of Mathematical Sciences], vol.~\textbf{293}, Springer-Verlag, Berlin, 1994.

\bibitem{SCHEUTZOW93}  
M.~Scheutzow, \emph{Stabilization and destabilization by noise in the plane}, Stochastic Anal. Appl. \textbf{11} (1993), no.~1, 97--113.  
  
\bibitem{SERRIN62}
J.~Serrin, \emph{On the interior regularity of weak solutions of the
  {N}avier-{S}tokes equations}, Arch. Rational Mech. Anal. \textbf{9}
  (1962), 187--195.

\bibitem{SHA16}
A.V.~Shaposhnikov, \emph{Some remarks on {D}avie's uniqueness theorem}, Proc. Edinb. Math. Soc. (2) \textbf{59} (2016), no.~4, 1019--1035.

\bibitem{STANNAT99}
W.~Stannat, \emph{\textup{(}{N}onsymmetric\textup{)} {D}irichlet operators on {$L^1$}: existence, uniqueness and associated {M}arkov processes}, Ann. Scuola Norm. Sup. Pisa Cl. Sci. (4) \textbf{28} (1999), no.~1, 99--140.

\bibitem{STRVAR79}
D.W.~Stroock and S.R.S.~Varadhan, \emph{Multidimensional diffusion processes}, Grundlehren der Mathematischen Wissenschaften [Fundamental Principles of Mathematical Sciences], vol.~\textbf{233}, Springer-Verlag, Berlin-New York, 1979.

\bibitem{VERAAR12}
M.~Veraar, \emph{The stochastic {F}ubini theorem revisited}, Stochastics
  \textbf{{\bf84}} (2012), no.~4, 543--551.

\bibitem{VERETENNIKOV81}
A.Y.~Veretennikov, \emph{Strong solutions and explicit formulas for solutions of stochastic integral equations}, Math. USSR-Sb. \textbf{39} (1981), no.~3, 387--403.

\bibitem{WLW17}
J.~Wei, G.~Lv, , and J.-L.~Wu, \emph{On weak solutions of stochastic
  differential equations with sharp drift coefficients}, arXiv:1711.05058
  (2017). 

\end{thebibliography}

% add below the content of your .bbl file produced by bibtex.

\end{document}